\documentclass[a4paper, 12pt]{amsart}
\usepackage{amsmath,mathrsfs,amsthm}
\usepackage{txfonts,rotating,tikz}
\usetikzlibrary{intersections,shapes.arrows,calc}
\usepackage[all,cmtip]{xy}
\usepackage{amssymb}
\usepackage{amsfonts}
\usepackage{amscd}
\usepackage{color} 
\usepackage[utf8]{inputenc}
\usepackage{hyperref}
\usepackage{fullpage}
\usepackage{epic}
\usepackage{mathtools}
\usepackage{tikz}
\usepackage{tikz-cd}
\usetikzlibrary{chains}
\usepackage[thinlines]{easytable}
\usepackage{array}
\usepackage{ctable}
\usepackage{pbox}
\usepackage[usestackEOL]{stackengine}
\usepackage[sort, numbers]{natbib}
\usepackage{relsize}
\usepackage{extarrows}
\usepackage{enumitem}
\tikzcdset{scale cd/.style={every label/.append style={scale=#1},
 cells={nodes={scale=#1}}}}
\newcolumntype{?}{!{\vrule width 1pt}}

\newtheorem{thm}{Theorem}[section]
\newtheorem{prop}[thm]{Proposition}
\newtheorem{lem}[thm]{Lemma}
\newtheorem{cor}[thm]{Corollary}
 
\theoremstyle{definition}
\newtheorem{definition}[thm]{Definition}
\newtheorem{example}[thm]{Example}
\theoremstyle{remark}
\newtheorem{remark}[thm]{Remark}
\numberwithin{equation}{section}

\newcommand{\mb}{\mathbb}
\newcommand{\mc}{\mathcal}
\newcommand{\mf}{\mathfrak}
\newcommand{\mr}{\mathrm}
\newcommand{\ms}{\mathscr}
\newcommand{\A}{\mathbb{A}}
\newcommand{\As}{\mathcal{A}^\sigma}
\newcommand{\bu}{\mathbin{\raisebox{0.2ex}{\scalebox{0.6}{\textbullet}}}}
\newcommand{\C}{\mathbb{C}}
\newcommand{\Conv}{\mr{Conv}}
\newcommand{\Ds}{\mathbb{D}^\sigma}
\newcommand{\G}{\mathcal{G}}
\newcommand{\Gr}{\mathrm{Gr}}
\newcommand{\Grb}{\overline{\mathrm{Gr}}}
\newcommand{\gs}{\mathfrak{g}[t]^\sigma}
\newcommand{\hs}{\mathfrak{h}[t]^\sigma}
\newcommand{\la}{\vec{\lambda}}
\newcommand{\lb}{\bar{\lambda}}
\newcommand{\Pic}{\mathcal{P}\mathit{ic}}
\newcommand{\simto}{\mathrel{\!\tilde{\,\to}}}
\newcommand{\U}{\mathrm{U}}

\newcommand{\doi}[1]{\href{https://doi.org/#1}{\texttt{doi:#1}}}
\newcommand{\arxiv}[1]{\href{https://arxiv.org/abs/#1}{\texttt{arXiv:#1}}}

\begin{document}
\title{Beilinson-Drinfeld Schubert varieties of parahoric group schemes and twisted global Demazure modules}

\author{Jiuzu Hong}
\address[J.\,Hong]{Department of Mathematics, University of North Carolina at Chapel Hill, Chapel Hill, NC 27599-3250, U.S.A.}
\email{jiuzu@email.unc.edu}

\author{Huanhuan Yu}
\address[H.\,Yu]{Beijing International Center for Mathematical Research, Peking University, No. 5 Yiheyuan Road Haidian District, Beijing, P.R.China 100871}
\email{huanhuan.yu@outlook.com}

\begin{abstract}
	Let $\mathcal{G}$ be a parahoric Bruhat-Tits group scheme arising from a $\Gamma$-curve $C$ and a certain $\Gamma$-action on a simple algebraic group $G$ for some finite cyclic group $\Gamma$. We prove the flatness of Beilinson-Drinfeld Schubert varieties of $\mathcal{G}$, we determine the rigidified Picard group of the Beilinson-Drinfeld Grassmannian ${\rm Gr}_{\mathcal{G}, C^n }$ of $\mathcal{G}$, and we establish the factorizable and equivariant structures on rigidified line bundles over ${\rm Gr}_{\mathcal{G}, C^n }$. 
	
	We develop an algebraic theory of global Demazure modules of twisted current algebras, and using our geometric results we prove that when $C=\mathbb{A}^1$, the spaces of global sections of line bundles on BD Schubert varieties of $\mathcal{G}$ are dual to the twisted global Demazure modules. This generalizes the work of Dumanski-Feigin-Finkelberg in the untwisted setting.
\end{abstract}

\maketitle
\tableofcontents

\section{Introduction}
Let $G$ be a simple algebraic group of adjoint type over $\mathbb{C}$, and let $\mf{g}$ be its Lie algebra. We fix a Borel subgroup $B$ and a maximal torus $T\subseteq B$. Let $\Gr_G$ be the affine Grassmannian of $G$. A (spherical) affine Schubert variety in $\Gr_G$ is parametrized by a dominant coweight $\lambda\in X_*(T)^+$ of $G$, denoted by $\Grb_G^\lambda$. Let $\ms{L}$ be the level one line bundle on $\Gr_G$, and let $\ms{L}^c$ be its $c$-th tensor power for some $c> 0$. Then, the space $H^0(\Grb_G^\lambda, \ms{L}^c)$ of global sections is dual to a level $c$ affine Demazure module $D(c, \lambda)$, cf.\,\cite{Ku}. By the work of Fourier-Littelmann \cite{FL06}, when $G$ is of type $A, D,E$, there is a fusion product $*$ on Demazure modules with the property that $D(c,\lambda)*D(c,\mu)\simeq D(c,\lambda+\mu)$ for any two dominant coweights $\lambda, \mu$ of $G$.   When $c=1$, $D(1, \lambda)$ is a Weyl module. The Weyl modules and Demazure modules have global counterparts, defined in \cite{FeLo04} and \cite{FM19,DF21} respectively. Similarly, the global counterparts of affine Grassmanians and affine Schubert varieties are Beilinson-Drinfeld (BD) Grassmannians and BD Schubert varieties, which play a central role in geometric Langlands program. In \cite{DFF21}, Dumanski-Feigin-Finkelberg proved that there is a duality between the spaces of global sections of level $c$ line bundle on BD Schubert varieties and the level $c$ global Demazure modules of the current algebra $\mf{g}[t]$, where under this duality the factorization of line bundles is compatible with the factorization of global Demazure module.  The fusion product of Demazure modules has actually been encoded in the factorization of line bundles on BD Schubert varieties, in the sense that the fusion of Demazure modules appears in a fiber of global Demazure module. 

Let $\sigma$ be a standard automorphism on $G$ of order $m$ preserving $B$ and $T$, where $\sigma$ is simply a diagram automorphism if $G$ is not of type $A_{2n} $, see Section \ref{133740}. We consider the twisted current algebra $\gs$, where $\sigma$ acts on $\mf{g}$ and $\sigma$ acts on $t$ by $\sigma(t)=e^{-\frac{2\pi \rm{i}}{m}} t$. In fact, $\gs$ is isomorphic to a special current algebra or hyperspecial current algebra defined in \cite{CIK14,KV16}, see \cite[Section 4.2]{KV16} and Theorem \ref{thm_hyperspecial}. 
The fusion theories of twisted Demazure modules and twisted Weyl modules are extensively studied, e.g.\,in \cite{FKKS12,FK13,CIK14,KV16}. Let $X_*(T)^+_\sigma$ be the set of $\sigma$-coinvariants of all dominant coweights of $G$. For each $\bar{\lambda}\in X_*(T)^+_\sigma$, we denote by $D^\sigma(c, \bar{\lambda})$ the level $c$ twisted Demazure module of $\gs$. Similar to the untwisted case, $D^\sigma(c, \bar{\lambda})$ is dual to the global section of a level $c$ line bundle on the twisted Schubert variety $\Grb^{\bar{\lambda}}_{\ms{G}}$ of the special parahoric group scheme $\ms{G}={\rm Res}_{\mc{O}/ \bar{\mc{O}}}(G_{\mc{O}})^\sigma$, see Section \ref{029728}. In \cite{BH}, Besson-Hong determined the smooth locus of $\Grb^{\bar{\lambda}}_{\ms{G}}$, where the Demazure modules of $\gs$ and the global Schubert variety of a parahoric Bruhat-Tits group scheme $\mc{G}:={\rm Res}_{\A^1/\bar{\A}^1}(G_{\A^1})^\sigma $ are crucially used. 

Our original purpose of this paper is to develop an algebraic theory of twisted global Demazure modules of $\gs$, and generalize the work \cite{DFF21} to the twisted setting, i.e. to establish a duality between twisted global Demazure modules of $\gs$ and the spaces of global sections of line bundles on BD Schubert varieties of $\mc{G}$. In order to establish this duality, the major part of our work is devoted to proving the flatness of BD Schubert varieties of $\mc{G}$, and establishing factorizable and equivariant structures on the rigidified line bundles over the BD Grassmannians, for more general parahoric Bruhat-Tits group scheme $\G$. On this part, we closely follow the works \cite{Zhu09,Zhu14,Zhu17} of Zhu. In fact, our work also provides proofs of some geometric ingredients needed in \cite{DFF21}, which were missing in literature. 
In spite of the parallels with the untwisted case,  there are some essential differences between twisted and untwisted cases. For example, the twisted theory over $\mathbb{A}^1$ is already very involved.  We need to employ many new ideas and techniques to take care of ramifications.

 Let $\Gamma$ be the finite group acting on $G$ generated by $\sigma$, and let $C$ be an irreducible smooth algebraic curve over $\mathbb{C}$ with a faithful $\Gamma$-action such that all ramified points in $C$ are totally ramified. Let $\bar{C}$ be the quotient curve $C/\Gamma$. Let $\mc{G}$ be the $\Gamma$-fixed point subgroup scheme of the Weil restriction group $\mr{Res}_{C/\bar C}(G\times C)$. Then $\mc{G}$ is a parahoric Bruhat-Tits group scheme over $\bar C$. In Section \ref{825344}, we define Beilinson-Drinfeld Schubert varieties of $\mc{G}$ and prove their flatness. 

Let $\Gr_{\mc{G},C^n}$ (resp.\,$L^+\mc{G}_{C^n}$) denote the Beilinson-Drinfeld Grassmannian (resp.\,the jet group scheme) of $\mc{G}$ over $C^n$. For any $n$-tuple of coweights $\la=(\lambda_1,\ldots, \lambda_n)$ of $G$, we first construct a section $s_{\la}:C^n\to \Gr_{\mc{G},C^n}$, and then we define the Beilinson-Drinfeld Schubert variety $\Grb_{\mc{G},C^n}^{\la}$ to be the reduced closure of the $L^+\mc{G}_{C^n}$-orbit at $s_{\la}$. The construction of $s_{\la}$ is achieved via a construction of a locally trivial $(\Gamma, T)$-torsor attached to $\vec{\lambda}$. By the Tannakian interpretation of $(\Gamma,T)$-torsor (cf.\,Theorem \ref{926976}), the construction of this $(\Gamma,T)$-torsor is reduced to the construction of a $\Gamma$-equivariant tensor functor, see Section \ref{993431}. The main result of Section \ref{825344} is Theorem \ref{368017}, which asserts that 
\begin{thm}\label{445316}
	The BD Schubert variety $\Grb_{\mc{G},C^n}^{\la}$ is flat over $C^n$. Moreover, the schematic fiber $\Grb_{\mc{G},\vec{p}}^{\la}$ of $\Grb_{\mc{G},C^n}^{\la}$ at each $\vec{p}\in C^n$ is reduced. More precisely, $\Grb_{\mc{G},\vec{p}}^{\la}$ is a product of affine Schubert varieties (possibly twisted). 
\end{thm}
In the untwisted case, this theorem is due to Zhu, see \cite{Zhu09} for $n=2$. In the twisted case, the main idea of the proof of Theorem \ref{445316} is similar. When $n=1$, it is also due to Zhu \cite{Zhu14}. For $n>1$, we need a new ingredient suggested by Zhu (by private communication), a generalization of cohomology and base change, which is formulated in Appendix \ref{419108}. This new method is of course applicable to the untwisted case. In order to employ Theorem \ref{259758}, in Section \ref{sect_conv} we define the convolution variety $\overline{\Conv}_{\G,C^n}^{\la}$ which admits a proper morphism to $\Gr_{\G,C^n}$. In fact, we give two definitions, one is defined via a sequence of fibrations of global Schubert varieties, and the other one is defined as the reduced closure of $L^+\G_{C^n}$-orbit at a section $\tilde{s}_{\la}$ of the convolution Grassmannian. In Proposition \ref{234953}, we show that these two definitions are the same. The first definition tells that $\overline{\Conv}_{\G,C^n}^{\la}$ is flat over $C^n$, and the second definition tells that the image of $\Conv_{\G,C^n}^{\la}$ in $\Gr_{\G,C^n}$ is exactly $\Grb_{\G,C^n}^{\la}$. As an application of Theorem \ref{445316}, we also prove that any BD Schubert variety is normal, see Theorem \ref{551564}. When $n=1$, this is due to Zhu, cf.\,\cite[Corollary 6.14]{Zhu14}.

Section \ref{717393} is devoted to the study of rigidified line bundles on $\Gr_{\mc{G},C^n}$. Let $\mr{Pic}^e(\Gr_{\mc{G},C^n})$ denote the Picard group of rigidified line bundles on $\Gr_{\mc{G},C^n}$, i.e. the line bundles on $\Gr_{\mc{G},C^n}$ together with a trivialization along a distinguished section $e$ over $C^n$, see Section \ref{201752}. We prove the following result in Theorem \ref{797528}. 
\begin{thm}\label{275410}
Suppose that $G$ is simply-connected. 
	\begin{enumerate}
		\item For any $n\geq 1$, the central charge map $\mr{Pic}^e(\Gr_{\mc{G},C^n})\to \mb{Z}$ is an isomorphism. 		
		\item Any line bundle in $\mr{Pic}^e(\Gr_{\mc{G},C^n})$ has factorizable property. 
	\end{enumerate}
\end{thm}
The proof of the injectivity of the central charge map is similar to the proof in the untwisted case, cf.\,\cite[Section 3]{Zhu17}. In the untwisted case, the level one line bundle can be obtained by pulling-back the level one line bundle on $\mr{Bun}_{G}$ over a projective smooth curve. However, in the twisted case the existence of the level one line bundle on $\mr{Bun}_{\G}$ is not known yet (actually it may not exist in certain cases).  Nevertheless, there is a level one line bundle $\mc{L}_{\A^n}$ on $\Gr_{\G,\A^n}$  constructed in \cite{BH}, and we are able to use this line bundle to construct line bundles on $\Gr_{\G,C^n}$.  
Our idea of the proof of Theorem \ref{275410} is as follows. 
For a general $\Gamma$-curve $C$, given any ramified point $p\in C$, there is a $\Gamma$-stable neighborhood $U$ of $p$ which admits a $\Gamma$-equivariant map to $\A^1$. Then, the pull-back of $\mc{L}_{\A^n}$ gives a level one line bundle on $\Gr_{\G,U^n\setminus Z}$ for some $Z$ closed in $U^n$, cf.\,Lemma \ref{423913}. These line bundles together with the level one line bundle over the unramified part (cf.\,Proposition \ref{284344}) can be glued via rigidity and extended to a line bundle $\mc{L}_{C^n}$ on $\Gr_{\G,C^n}$, which generates $\mr{Pic}^e(\Gr_{\mc{G},C^n})$. Moreover, this line bundle satisfies the factorization property. When $n>1$, we crucially use a relative version of seesaw principle (cf.\,Lemma \ref{086885}). 

Another main result of Section \ref{717393} is  Theorem \ref{206120}, which asserts that
\begin{thm}\label{346721}
With the same assumption as in Theorem \ref{275410},	there is a unique $L^+\mc{G}_{C^n}$-equivariant stucture on the level one line bundle $\mc{L}_{C^n}$ on $\Gr_{\mc{G},C^n}$.
\end{thm} 
The hard part of this theorem is the existence. In the proof, we first construct a $L^+\mc{G}_{C^n}$-equivariant line bundle $\mc{L}_{\rm det}$ by embedding $\Gr_{\mc{G},C^n}$ to a BD Grassmannian of $GL(\mc{V})$ for some vector bundle $\mc{V}$ over $C$. By the \'etale base change property of BD Grassmannians (cf.\,Corollary \ref{prop_etale_base}), we are reduced to the case $C=\A^1$. In this case, the line bundle $\mc{L}_{\rm det}$ has a rigidified structure, and hence is a power of $\mc{L}_{\A^n}$, where Theorem \ref{275410} is crucially used. This in fact implies that $\mc{L}_{\A^n}$ admits a $L^+\mc{G}_{\A^n}$-equivariant structure. This step is guaranteed by Proposition \ref{882893}, whose proof is quite technical. We first prove this proposition when $n=1$ and $n=2$ separately, and for general case we use Hartogs' Lemma. In our argument, we need to restrict line bundles to BD Schubert varieties so that all arguments are in the finite type geometry. 
\medskip
 
To prove Theorem \ref{thm_main} (or Theorem \ref{055073}) in full generality, we need to work with adjoint type group $G$. In this case, it is difficult (or impossible) to construct a factorizable level one line bundle on $\Gr_{\G,C^n}$ for a general $\Gamma$-curve $C$. Nevertheless, when $C=\A^1$ with the standard action of $\Gamma$. We are able to construct level one line bundles on BD Schubert varieties which admits factorization properties compatible with the factorization of BD Schubert varieties, see Proposition \ref{324995}. Moreover, using Theorem \ref{346721} we are able to prove that these line bundles admit $L^+\mc{G}'_{C^n}$-equivariant structures, where $L^+\mc{G}'_{C^n}$ is the jet group scheme arising from the simply-connected cover of $G$, see Proposition \ref{235859}. 

As mentioned earlier, the main motivation of this work is to establish a relation between twisted global Demazure modules and line bundles on BD Schubert varieties of $\G$. 
In Section \ref{320070}, for any $\vec{\lambda}\in (X_*(T)^+)^n$, we define the twisted global Demazure module $\mathbb{D}^\sigma(c,\la)$ of $\gs$. Using the results of Fourier-Littelmann \cite{FL06} and Kus-Vankatesh \cite{KV16}, we prove the following result in Theorem \ref{398063} and Theorem \ref{616072}.
\begin{thm} Given $\vec{\lambda}=(\lambda_1,\cdots, \lambda_n)\in (X_*(T)^+)^n$, let $\lb$ be the image of $\sum \lambda_i$ in $X_*(T)^+_\sigma$. 
\begin{enumerate}
\item The fiber $\Ds(c,\la)_{0}$ at $0\in \A^n$ is the twisted Demazure module $D^\sigma(c, \lb)$. 
\item The twisted global Demazure module $\Ds(c,\la)$ is a graded free $\As(c,\vec{\lambda})$-module, where $\As(c,\vec{\lambda})$ is the weight algebra associated to the cyclic generator of $\Ds(c,\la)$.
\end{enumerate}
\end{thm}

As a consequence of Theorem \ref{445316} and Proposition \ref{235859}, the space $H^0(\Grb^{\la}_{\mc{G},\A^n},\mc{L}_{\A^n})$ of global sections of the level one line bundle $\mc{L}_{\A^n}$ is a $(\gs,\C[\A^n])$-bimodule. In Section \ref{324893}, we further equip a $\mb{G}_m$-equivariant structure on the line bundle $\mc{L}_{\A^n}$, which follows that $H^0(\Grb^{\la}_{\mc{G},\A^n},\mc{L}_{\A^n})$ is a graded free $\C[\A^n]$-module, see Theorem \ref{243693}. Finally, we prove the following Borel-Weil type theorem in Theorem \ref{055073}. 
\begin{thm}\label{thm_main}
	There exists an isomorphism of $(\gs, \C[\A^n])$-bimodules
	\[ \Ds(c,\la)_{\A^n}\simeq H^0\big( \Grb_{\mc{G},\A^n}^{\la}, \mc{L}_{\A^n}^c\big)^\vee,\]
	where $\Ds(c,\la)_{\A^n}=\Ds(c,\la)\otimes_{\As(c,\vec{\lambda})}  \mathbb{C}[\A^n]$.  Moreover, $\Ds(c,\la)_{\A^n}$ has factorization property and it is compatible with the factorization of $\mc{L}_{\A^n}$ under the above isomorphism. 
\end{thm}
Our proof is similar to the proof in \cite{DFF21}, while there are some distinctions, see Remark \ref{rem_difference}. As an application of this theorem, the fiber $ \Ds(c,\la)_{\vec{p}}$ at any  $\vec{p}\in \A^n$ can be determined, see Corollary \ref{367230} and Remark \ref{remark_fiber}. In fact, the fiber $ \Ds(c,\la)_{\vec{p}}$ is a tensor product of affine Demazure modules (possibly twisted) depending on the type of the point $\vec{p}$. 

The Borel-Weil type theorem for twisted global Weyl module is proved by Braverman-Finkelberg in \cite[Theorem 4.6]{BF}, where they work with twisted Zastava spaces. In fact, when the level $c=1$, $\Ds(c,\vec{\lambda})$  is isomorphic to a twisted global Weyl module, see Theorem \ref{cor_dem_weyl}. 

Finally, we make a remark when $G$ is $A_{2\ell}$. In this case,  our automorphism $\sigma$ is chosen to be the standard automorphism of order 4. The twisted current algebra $\gs$ corresponds to a special parahoric subalgebra of $A^{(2)}_{2\ell}$. In fact, there is another special parahoric subalgebra of $A^{(2)}_{2\ell}$, which can be realized as the twisted current algebra $\mf{g}[t]^\tau$ corresponding to the diagram automorphism $\tau$ of order $2$. For this automorophism $\tau$, all results in Section \ref{825344} hold, and in Section \ref{717393} similar results hold except that ${\rm Pic}^e(\Gr_{\G, C^n})$ is generated by a level $2$ rigidified line bundle. Moreover, we expect that Theorem \ref{thm_main} still hold for even levels. 

\medskip

\textbf{Acknowledgements}: We would like to thank I.\,Dumanski, E.\,Feigin, M.\,Finkelberg, G.\,Fourier and D.\,Kus for helpful email correspondences. We thank R.\,Travkin and X.\,Zhu for inspiring discussions on geometric aspects of this work.  We also thank Wyatt Reeves for helpful discussions and sharing his unpublished results.  
 J.\,Hong was partially supported by the NSF grant DMS- 2001365.

\section{Affine Demazure modules of twisted current algebra}\label{656829}
In this section, we collect some known results of affine Demazure modules, Weyl modules, and their twisted analogues, which will be used in Section \ref{716168}. 
\subsection{Notations}
Let $G$ be an almost simple algebraic group over $\C$. We fix a maximal torus $T$ and a Borel subgroup $B$ containing $T$. Let $X^*(T)$ (resp.\,$X_*(T)$) denote the weight lattice (resp.\,the coweight lattice) of $T$. Let $X_*(T)^+$ denote the set of dominant coweights of $G$ with respect to $B$. Let $R$ be the set of roots of $G$. Let $R^+$ denote the set of positive roots of $G$ with respect to $B$. Let $\{\alpha_i \mid i\in I\}$ be the set of simple roots, where $I$ is the set of vertices of the Dynkin diagram of $G$. 

Let $\mf{g}$, $\mf{b}$, $\mf{h}$ denote the Lie algebras of $G$, $B$, $T$ respectively. One has the triangular decomposition $\mf{g}=\mf{n}^-\oplus\mf{h}\oplus\mf{n}^+$, where $\mf{n}^+$ is the nilpotent radical of $\mf{b}$ and $\mf{n}^-$ is the nilpotent radical of the opposite Borel subalgebra $\mf{b}^-$. Let $\{e_{\alpha} \mid \alpha \in R \}$ be a Chevalley basis of $\mf{g}$. For simplicity of notations, we also set 
 $e_i=e_{\alpha_i}$, $f_i=e_{-\alpha_i}$, and $h_i=\alpha_i^\vee$ for any $i\in I$. 
 
In this section, we always assume that $G$ is simply-laced and of \textbf{adjoint type}. Then, the set $X_*(T)^+$ can be identified with set of all dominant coweights of $\mf{g}$. In Section \ref{656829} and Section \ref{716168}, we only deal with Lie algebras, but to be consistent with the notations in later sections on geometry, we keep using the notation $X_*(T)^+$. 

\subsection{Twisted current algebra}\label{133740}
Let $\tau$ be a nontrivial diagram automorphism preserving $\mf{b}, \mf{h}$ and the pinning $\{e_i,f_i; i\in I\}$. Let $m=2,3$ or $4$ and fix an $m$-th primitive root $\epsilon$. 
Following \cite[Section 2.1]{BH}, let $\sigma$ be a standard automorphism of order $m$ on $\mf{g}$ preserving $\mf{b}$ and $\mf{h}$. More specifically, when $\mf{g}$ is not $A_{2\ell}$, $\sigma=\tau$; when $\mf{g}$ is of type $A_{2\ell}$, 
\begin{equation}\label{eq_st_h}
\sigma=\tau\circ \epsilon^h,\end{equation}
 where $h\in \mf{h}^\tau$ such that 
\[\alpha_i(h)=
\begin{cases}
	0, & \text{if } i\neq\ell, \ell+1;\\
	1, & \text{if } i=\ell,\ell+1.
\end{cases}\]
In particular, 
\[m=
\begin{cases}
	2, & \text{if } \mf{g} \text{ is } A_{2\ell-1}, D_{\ell+1} \text{ or } E_6;\\
	3, & \text{if } \mf{g} \text{ is } D_4;\\
	4, & \text{if } \mf{g} \text{ is } A_{2\ell}.
\end{cases}\]
Let $\mf{g}^\sigma$ be the $\sigma$-fixed point Lie subalgebra in $\mf{g}$. We have the following descriptions of $\mf{g^\sigma}$,
 \begin{equation}
\label{Fix_table}
 \begin{tabular}{|c | c | c |c |c |c | c| c| c|c |c|c|c|c|c |c ||} 
 \hline
$(\mf{g}, m)$ & $(A_{2\ell-1}, 2 ) $ & $(A_{2\ell}, 4)$ & $(D_{\ell+1} , 2 ) $ & $ (D_4, 3)$ & $ (E_6, 2)$ \\ [0.8ex] 
 \hline
$ \mf{g}^\sigma $ & $ C_\ell $ & $ C_\ell $ & $B_\ell $ & $G_2 $ & $F_4 $ \\ [0.8ex] 
 \hline
\end{tabular},
\end{equation}
where by convention $C_1$ is $A_1$ and $\ell\geq 3$ for $D_{\ell+1}$.

 We extend $\sigma$ to an automorphism on the loop algebra $L(\mf{g})=\mf{g}((t))$ by
\begin{equation}\label{eq_act_1}
\sigma( x\otimes t^j ) =\epsilon^{-j}\sigma ( x) \otimes t^j.\end{equation}
We define the twisted affine algebra $\hat{L}(\mf{g},\sigma):=\mf{g}((t))^\sigma\oplus \C K$ as a central extension of $L(\mf{g})^\sigma$ with the canonical center $K$ as follows:
\[[x[f]+zK, x'[f']+z'K]=[x,x'][ff']+m^{-1}\mr{Res}_{t=0}((df)f')(x,x')K\]
for $x[f], x'[f']\in L(\mf{g})$, $z,z'\in \C$, where $\mr{Res}_{t=0}$ denotes the coefficient of $t^{-1}dt$, and $(,)$ is the normalized Killing form on $\mf{g}$ such that $(\theta,\theta)=2$ for the highest root $\theta$. 

Let $\gs$ be the twisted current algebra, i.e. the $\sigma$-fixed point subalgebra of $\mf{g}[t]$. In fact, $\gs$ is a subalgebra of $\hat{L}(\mf{g},\sigma)$. 
\begin{enumerate}[label=(\roman*)]
	\item \label{en_basis_A_2n} When $\mf{g}$ is of type $A_{2\ell}$, let $\alpha_{i,j}$ be the positive root $\alpha_i+\cdots +\alpha_j$, for $1\leq i\leq j\leq 2\ell$, and let $e_{\pm \alpha}$ be the standard basis in $\mf{g}=sl_{2\ell+1}$. Then $\gs$ has the following basis:	
	\begin{align}\label{802623}
		\begin{split}
		\quad 
		& e_{\pm \alpha,2k}:= e_{\pm\alpha}\otimes t^{2k}+(-1)^{i+j+k}e_{\pm\tau(\alpha)}\otimes t^{2k}, \text{ when } \alpha=\alpha_{i,j} \text{ with } 1\leq i\leq j <\ell\\
		& e_{\pm \alpha,2k}:= e_{\pm\alpha}\otimes t^{2k}-(-1)^{i+j+k}e_{\pm\tau(\alpha)}\otimes t^{2k}, \text{ when } \alpha=\alpha_{i,2\ell-j} \text{ with } 1\leq i\leq j <\ell\\
		& e_{\pm \alpha,4k}:= e_{\pm\alpha}\otimes t^{4k}, \text{ when } \alpha=\alpha_{i,2\ell+1-i} \text{ with } 1\leq i\leq \ell\\
		& e_{\pm \alpha,2k+1}:= e_{\pm\alpha}\otimes t^{2k+1} \pm (-1)^{i+\ell+k}e_{\pm\tau(\alpha)}\otimes t^{2k+1}, \text{ when } \alpha=\alpha_{i,\ell} \text{ with } 1\leq i\leq \ell\\
		& h_{i,2k}:= h_i\otimes t^{2k}+(-1)^k h_{2\ell+1-i}\otimes t^{2k}, \text{ for any } 1\leq i\leq \ell
		\end{split}
	\end{align}
(Note that the actions of $\tau$ and $\sigma$ agree on the roots.)
	\item When $\mf{g}$ is not of type $A_{2\ell}$, let $R^+_s$ (resp.\,$R^+_l$) be the set of roots $\alpha\in R^+$ such that the restriction $\alpha|_{\mf{h}^\sigma}$ is a short (resp.\,long) root of $\mf{g}^\sigma$. Then $\gs$ has the following basis: 
	\begin{align*}
		\quad 
		& e_{\pm \alpha,k}:= \sum_{i=0}^{m-1} e_{\pm\sigma^i(\alpha)}\otimes (\epsilon^{i}t)^{k}, \ \alpha \in R^+_{s}; & e_{\pm \alpha,mk}:= \sum_{i=0}^{m-1} e_{\pm\sigma^i(\alpha)}\otimes t^{mk}, \ \alpha\in R^+_{l};\\
		& h_{\alpha,k}:= \sum_{i=0}^{m-1} h_{\sigma^i(\alpha)}\otimes (\epsilon^i t)^{k},\ \alpha\in R^+. &
	\end{align*}
\end{enumerate}
When $\mf{g}$ is not $A_{2\ell}$, $\gs$ is isomorphic to the special twisted current algebra of $\hat{L}(\mf{g},\sigma)$, see \cite[Section 4.2]{KV16}. When $\mf{g}$ is $A_{2\ell}$, we now show $\gs$ is isomorphic to the hyperspecial twisted current algebra $\mf{Cg}$ defined in \cite[Section 2.5]{CIK14} and \cite[Section 1.5]{KV16}.

Assume that $\mf{g}$ is of type $A_{2\ell}$. Following \cite[Section 4.4]{CIK14} and \cite[Section 1.8]{KV16}, the hyperspecial current algebra $\mathfrak{C} \mathfrak{g} \subset \mf{g}[t,t^{-1}]^\tau$ consists of the following basis elements:
\begin{enumerate}[itemsep=0.5em]
	\item $e_{\pm \alpha} \otimes t^k + (-1)^{i+j} e_{\pm \tau(\alpha) } \otimes (-t)^k$, $\alpha=\alpha_{ij}$ with $1 \leq i \leq j <\ell$, and $k\geq 0$; 
	\item $e_{\pm \alpha}\otimes t^{k\pm 1} + (-1)^{i+j} e_{\pm \tau(\alpha)} \otimes (-t)^{k\pm 1}$, $\alpha=\alpha_{i,2\ell-j}$ with $1 \leq i \leq j <\ell$, and $k\geq 0$; 
	\item $e_{\pm \alpha} \otimes t^{2k\pm 1} $, $\alpha=\alpha_{i, 2\ell+1-i}$ with $1 \leq i \leq \ell$, and $k\geq 0$; 
	\item $e_{\pm \alpha} \otimes t^{ (2k+1\pm 1)/2} + (-1)^{\ell+i}e_{\pm \tau(\alpha)} \otimes (-t)^{ (2k+1\pm 1)/2}$, $\alpha=\alpha_{i\ell} $ with $1\leq i \leq \ell $, and $k\geq 0$; 
	\item $h_i\otimes t^k + h_{2\ell+1-i}\otimes (-t)^k $, $1\leq i \leq \ell$, and $k\geq 0$,
\end{enumerate}
where $\alpha_{ij}$ denotes the positive root $\alpha_i+\cdots + \alpha_j$, for any $1\leq i\leq j \leq 2\ell$. Note that $\tau(\alpha_{ij})= \alpha_{ 2\ell+1-j,2\ell+1-i }$. 
\begin{lem}
Assume that $\mf{g}$ is of type $A_{2\ell}$.
	We have the following formula
	\begin{equation}
		\sigma (e_{\pm \alpha_{ij}}) = (-1)^{j-i } \epsilon^{\pm \alpha_{ij}(h) }e_{ \pm\alpha_{2\ell+1-j,2\ell+1-i} }, \text{ for any } 1\leq i\leq j \leq 2\ell+1,
	\end{equation}
	and $\sigma(h_i)= h_{2\ell+1-i}$, for any $1\leq i\leq 2\ell+1$, where $h$ is defined in (\ref{eq_st_h}).
\end{lem}
\begin{proof}
	In \cite[Section 7.10]{Kac90}, the formula for the action of $\tau$ on every basis vector in $\mf{g}$ is given. Combining this formula and $ \sigma= \tau  \epsilon^{h}$, we can deduce the formula for $\sigma$. 
\end{proof}

Assume that $\mf{g}$ is of type $A_{2\ell}$. There exists an isomorphism $ \eta_k: \hat{L}(\mf{g}, \tau)\simeq \hat{L}(\mf{g}, \sigma)$ due to Kac (cf. \cite[p8]{HK} and \cite[Section 8.5]{Kac90}), which is defined as follows 
\[ \eta_k( x[t^j])= x[t^{ 2j+s }], \]
for any $x$ a simultaneous $(-1)^j$-eigenvector of $\tau$, and a $s$-eigenvector of $\mr{ad}h$. 
Let $\phi$ be the Cartan involution on $\mf{g}$ such that $\phi(e_i)=-f_i$ and $\phi(h_i)=-h_i$. Notice that for any root $\alpha $, $\phi(e_{ \alpha})=-e_{- \alpha}$. Now, we define an automorphism $\eta_c: \hat{L}(\mf{g}, \tau)\to \hat{L}(\mf{g}, \tau)$ as follows, 
\[ \eta_c(x[t^s])=\phi(x)[t^s], \text{ for any } x[t^s]\in \hat{L}(\mf{g}, \tau) . \]
It is easy to verify that $\eta_c$ is well-defined, since $\eta_c\circ \tau=\tau\circ \eta_c$, where $\tau$ is the automorphism on the untwisted affine Lie algebra $\hat{\mf{g}}$ induced from the action in (\ref{eq_act_1}) . 
We define the following isomorphism of twisted affine algebras 
\begin{equation}\label{def_eta}	
	\eta:= \eta_k\circ \eta_c: \hat{L}(\mf{g}, \tau)\simeq \hat{L}(\mf{g}, \sigma) .
\end{equation}

The following theorem originally appeared in an earlier version of \cite{BH}, where some calculation mistakes are fixed in this paper. 
\begin{thm}
	\label{thm_hyperspecial}
	The map $\eta$ restricts to the following isomorphism of Lie algebras
	\[\eta: \mathfrak{C} \mathfrak{g} \simeq \mf{g}[t]^\sigma. \] 
\end{thm}
\begin{proof}
	From the definition \eqref{def_eta}, we have the following calculations: 
	\begin{enumerate}
		\item When $\alpha=\alpha_i+\dots +\alpha_j$ with $1 \leq i \leq j <\ell$, and $k\geq 0$, we have
		\[ \eta(e_{\pm \alpha} \otimes t^k + (-1)^{i+j} e_{\pm \tau(\alpha) } \otimes (-t)^k)=-(e_{\mp\alpha}\otimes t^{2k}+(-1)^{i+j+k}e_{\mp\tau(\alpha)}\otimes t^{2k}) . \]
		\item When $\alpha=\alpha_{i,2\ell-j}$ with $1 \leq i \leq j <\ell$, and $k\geq 0$, we have 
		\[ \eta(e_{\pm \alpha}\otimes t^{k\pm 1} + (-1)^{i+j} e_{\pm \tau(\alpha)} \otimes (-t)^{k\pm 1} )=- (e_{\mp\alpha}\otimes t^{2k}-(-1)^{i+j+k}e_{\mp\tau(\alpha)}\otimes t^{2k} ). \]
		\item When $\alpha=\alpha_{i, 2\ell+1-i}$ with $1 \leq i \leq \ell$, and $k\geq 0$, we have
		\[ \eta( e_{\pm \alpha} \otimes t^{2k\pm 1} ) =-e_{\mp \alpha} \otimes t^{4k}. \]
		\item When $\alpha=\alpha_{i\ell} $ with $1\leq i \leq \ell $, and $k\geq 0$, we have 
		\[ \eta(e_{\pm \alpha} \otimes t^{ (2k+1\pm 1)/2} + (-1)^{\ell+i}e_{\pm \tau(\alpha)} \otimes (-t)^{ (2k+1\pm 1)/2} ) =- (e_{\mp\alpha}\otimes t^{2k+1} \mp (-1)^{i+\ell+k}e_{\mp\tau(\alpha)}\otimes t^{2k+1} ). \]
		\item For any $1\leq i \leq \ell $, and $k\geq 0$, we have
		\[ \eta(h_i\otimes t^k + h_{2\ell+1-i}\otimes (-t)^k ) =-(h_i\otimes t^{2k}+(-1)^k h_{2\ell+1-i}\otimes t^{2k}) .\]
	\end{enumerate}
	Then, it is easy to see that the image of the described basis of the hyperspecial current algebra $\mf{Cg}$ under the map $\eta$ is exactly the basis of $\mf{g}[t]^\sigma$ described above. 
\end{proof}

\subsection{Weyl modules}\label{427932}
We denote by $P$ (resp.\,$P_\sigma$) the weight lattice of $\mf{g}$ (resp. $\mf{g}^\sigma$), and by $P^+$ (resp.\,$P^+_\sigma$) the set of dominant weights of $\mf{g}$ (resp.\,$\mf{g}^\sigma$) with respect to the Borel subalgebra $\mf{b}$ (resp. $\mf{b}^\sigma$). We define the map $\iota: X_*(T)\to P$ given by 
\[\iota(\lambda)(h)=(\lambda, h), \ \forall \lambda \in X_*(T), h\in \mf{h},\]
where $\lambda$ is regarded as an element in $\mf{h}$ and $(\ ,\ )$ is the normalized Killing form on $\mf{h}$. By restricting $\iota(\lambda)$ to $\mf{h}^\sigma$, it induces a map $X_*(T)\to P_\sigma$, and it descends to the map $X_*(T)_\sigma \to P_\sigma$. We still denote it by $\iota$.

Following \cite{FL06}, for any dominant weight $\mu\in P^+$ of $\mf{g}$, we define the Weyl module $W(\mu)$ of $\mf{g}[t]$ to be the cyclic module generated by an element $w_{\mu}$ subject to the following relations
\[\mf{n}^-\otimes \mathbb{C}[t]\cdot w_{\mu}=0, \ \mf{h}\otimes t\C[t]\cdot w_{\mu}=0,\ (h\otimes 1)\cdot w_\mu =-\mu(h)w_{\mu},\ (e_{\alpha_i}\otimes 1)^{\mu(\alpha_i^\vee)+1}\cdot w_{\mu}=0,\]
for any simple roots $\alpha_i\in R^+$, and any $h\in \mf{h}$.

Following \cite{KV16}, for each dominant weight $\bar{\mu}\in P^+_\sigma$ of $\mf{g}^\sigma$, we define the twisted Weyl module $W^\sigma(\bar{\mu})$ of $\gs$ to be the cyclic module generated by an element $w_{\bar{\mu}}$ subject to the following relations
\begin{equation}\label{989363}
	\mf{n}^-[t]^\sigma\cdot w_{\bar{\mu}}=0,\ (\mf{h}\otimes t\C[t])^\sigma\cdot w_{\bar{\mu}}=0,\ (h\otimes 1)\cdot v =-\bar{\mu}(h)w_{\bar{\mu}},\ (e_{\alpha,0})^{k_\alpha+1}\cdot w_{\bar{\mu}}=0 ,
 \end{equation}
for any $h\in \mf{h}^\sigma$, ${\alpha}\in R^+$, and $k_\alpha=\bar{\mu}( \bar{\alpha}^\vee)$, where $ \bar{\alpha}$ is the root ${\alpha}|_{\mf{h}^\sigma}$ of $\mf{g}^\sigma$. 

We define the twisted global Weyl module $\mb{W}^\sigma(\bar{\mu})$ to be the cyclic $\U(\gs)$-module generated by an element $w_{\bar{\mu}}$ with the following relations
\begin{equation*}
	\mf{n}^-[t]^\sigma\cdot w_{\bar{\mu}}=0,\ (h\otimes 1)\cdot w_{\bar{\mu}} =-\bar{\mu}(h)w_{\bar{\mu}},\ (e_{\alpha,0})^{k_\alpha+1}\cdot w_{\bar{\mu}}=0 ,
\end{equation*}
where $h\in \mf{h}^\sigma$, ${\alpha}\in R^+$, and $k_\alpha=\bar{\mu}( \bar{\alpha}^\vee)$. Similar to the untwisted case in \cite{CFK10}, there is a well-defined right action of $\U(\hs)$ on $\mb{W}^\sigma(\bar\mu)$ given by 
\[(u\cdot w_{\bar \mu})\ h:= u\cdot (h\cdot w_{\bar\mu})\]
where $u\in \U(\gs), h\in \U(\hs)$. This action commutes with $\U(\gs)$-action. Let ${A}^\sigma(\bar\mu):=\U(\hs)/\mr{Ann}(w_{\bar\mu})$ be the quotient of $\U(\hs)$ by the annihilator of $w_{\bar\mu}$.
\begin{remark}
In this paper,  the (twisted) Weyl modules and the twisted global counterparts are defined using the lowest weight cyclic generators. By composing the map $\eta$ (when $\mf{g}$ is $A_{2\ell}$) or a Cartan involution (otherwise), our definitions agree with the usual definitions in literature.  Also see Remark \ref{rem_dem_def} on the definitions of Demazure modules. 
\end{remark}
Let $I_\sigma$ be the set of vertices of the Dynkin diagram of $\mf{g}^\sigma$. We follow the Bourbaki labeling of the vertices of the Dynkin diagram, which gives a total order on $I$ and $I_\sigma$. Let $\{\beta_j\}_{j\in I_\sigma}$ be the set of simple roots of $\mf{g}^\sigma$. Then, there is a map $\eta:I\to I_\sigma$ such that 
\begin{equation}\label{328480}
	\check\beta_i=\sum_{j\in \eta^{-1}(i)} \check\alpha_j
\end{equation}
for any simple coroot $\check\beta_i$, cf.\,\cite[Section 2.1]{BH}. Let $\omega_j$ be the fundamental weight of $\mf{g}$ associated to $j\in I$. For any $i \in I_\sigma$, let $\bar{\omega}_i:=\omega_j|_{\mf{h}^\sigma}$ for some $j\in \eta^{-1}(i)$. Then $\{\bar{\omega}_i\}_{i\in I_\sigma}$ gives the set of fundamental weights of $\mf{g}^\sigma$, cf.\,\cite[Section 2.1]{BH}. Moreover, by \cite[Lemma 3.4]{BH}, we have 
\begin{equation}
\label{eq_match_wt}
 \iota(\check{\omega}_j)=\bar{\omega}_i, \quad  \forall j\in \eta^{-1}(i).  \end{equation}

For each $i\in I_\sigma$, set $m_i:=\frac{m}{|\eta^{-1}(i)|}$. For example, when $\mf{g}$ is $A_{2\ell}$, we have $m_i=2$ for any $i\in I_\sigma$. The following theorem on twisted global Weyl modules was proven in \cite{CIK14}, and it will be used in Section \ref{sect_t_weyl}.
\begin{thm}\label{350320} 
	Let $\bar\mu=\sum_{i\in I_\sigma} n_i\bar\omega_i\in P^+_\sigma$. 
	\begin{enumerate}
		\item \label{thm_part_1} The algebra $A^\sigma(\bar\mu)$ is a graded polynomial algebra in variables $T_{i,r}$ of grade $m_ir$, where $i\in I_\sigma$ and $1\leq r\leq n_i$. In particular, when $\mf{g}$ is $A_{2\ell}$, $m_i=2$ for any $i\in I_\sigma$.
		
		\item Let $I_{\bar\mu,0}$ be the unique graded maximal ideal in $A^\sigma(\bar\mu)$. Then, there is a $\gs$-isomorphism
		\[\mb{W}^\sigma(\bar\mu)\otimes_{A^\sigma(\bar\mu)} (A^\sigma(\bar\mu)/I_{\bar\mu,0})\simeq W^\sigma(\bar\mu).\]
		
		\item $\mb{W}^\sigma(\bar\mu)$ is a free $A^\sigma(\bar\mu)$-module of rank equal to
		$\prod_{i\in I_\sigma} \big(\dim W^\sigma(\bar\omega_i)\big)^{n_i}$.
	\end{enumerate}
\end{thm}
\begin{proof}
For part (1), see \cite[Theorem 1]{CIK14} when $\mf{g}$ is $A_{2\ell}$ and \cite[Theorem 8]{CIK14} when $\mf{g}$ is not $A_{2\ell}$. For part (2), it follows from the definition of twisted global Weyl module. For part (3),  see \cite[Theorem 4, Theorem11]{CIK14}. 
\end{proof}

\subsection{Affine Demazure modules}\label{925386}
The integrable highest weight modules of $\hat{L}(\mf{g},\sigma)$ of level $c$ can be parameterized by a set $P(\sigma,c)$ of certain highest weights of $\mf{g}^\sigma$, see \cite[Section 2]{HK}. In particular, for any $\kappa \in P(\sigma,1)$, we always have $c\kappa \in P(\sigma, c)$. We denote by $\ms{H}_c(\bar\mu)$ the integrable highest weight module associated to a highest weight $\bar\mu\in P(\sigma, c)$. 

Set $\ms{H}_c=\oplus_{\kappa\in P(\sigma,1)} \ms{H}_c(c\kappa)$. For any $\bar{\lambda}\in X_*(T)_\sigma^+$, let $v_{\bar{\lambda}}\in \ms{H}_c$ be the extremal weight vector whose $\mf{h}^\sigma$-weight is $-c\iota(\bar{\lambda})$, see \cite[Lemma 3.6]{BH}. We define the twisted affine Demazure $D^\sigma(c,\bar{\lambda})$ as the following $\gs$-module,
\begin{equation}\label{eq_Demazure}
 D^\sigma(c,\bar{\lambda}):=\U(\gs) v_{\bar{\lambda}}\subseteq \ms{H}_c.\end{equation}
For any $\lambda\in X_*(T)^+$,  the untwisted affine Demazure module $D(c,\lambda)$ is defined in a similar way by simply taking $\sigma=\mr{id}$.
\begin{remark}\label{rem_dem_def}
	The affine Demazure module $D(c,\lambda)$ defined here can be identified with the module $D(c, \lambda)$ defined in \cite[Section 2.2]{FL06} with a $\mf{g}[t]$-action via the Cartan involution $\eta_c: \mf{g}[t]\to \mf{g}[t]$. Similarly, our twisted affine Demazure module $D^\sigma(c,\bar{\lambda}) $ can be identified with the module $D(c, c \iota(\bar{\lambda}))$ defined in \cite[Section 3.2]{KV16} via a twist of the Cartan involution $\eta_c:\gs\to\gs$ when $\mf{g}$ is not $A_{2\ell}$, or via a twist of the map $\eta: \mf{Cg}\to \gs$ defined in (\ref{def_eta}) when $\mf{g}$ is of type $A_{2\ell}$.
\end{remark}
Thanks to the works \cite{FL06,KV16} of Fourier-Littelmann and Kus-Venkatesh, we have the following theorem, which will be used in the proof of Theorem \ref{398063}.
\begin{thm}\label{401910} 
	Let $\mf{g}$ be a simply-laced simple Lie algebra with a standard automorphism $\sigma$. Let $\lambda$ be a dominant coweight of $\mf{g}$. Denote by $\lb$ the image of $\lambda$ under the projection map $X_*(T)^+\to X_*(T)^+_\sigma$. Then 
	\begin{enumerate}
		\item The untwisted Demazure module $D(c,\lambda)$ is isomorphic to the quotient of the Weyl module $W(c\iota(\lambda))$ by the submodule generated by 
		\[\{(e_{\alpha}\otimes t^s)^{k_{\alpha,s}+1}\cdot w_{\lambda}\mid\alpha \in R^+, s\geq 0 \},\]
		where $k_{\alpha,s}=c\cdot \max\{0,(\lambda,\alpha^\vee)-s\}$, and $w_{\lambda}$ is the generator whose $\mf{h}$-weight is $-c\iota(\lambda)$.
		\item The twisted Demazure module $D^\sigma(c,\bar{\lambda})$ is isomorphic to the quotient of the twisted Weyl module $W^\sigma(c\iota(\bar{\lambda}))$ by the submodule generated by 
		\begin{align*}
		& \{e_{\alpha,m s_{\alpha}}\cdot w_{\lb} \mid \alpha \in R^+_{l} \} \cup \{ e_{\alpha, s_{\alpha} }\cdot w_{\lb} \mid \alpha \in R^+_{s}\}, \text{ when } \mf{g} \text{ is not } A_{2\ell}; \\[0.5em]
		& \{e_{\alpha,4s_\alpha}\cdot w_{\lb}\mid \alpha=\alpha_{i,2\ell+1-i}, 1\leq i<\ell \} \cup \{e_{\alpha, 2s_{2\alpha}+1}\cdot w_{\lb} \mid \alpha=\alpha_{i,\ell}, 1\leq i\leq \ell \}\\
		& \cup \{e_{\alpha,2s_\alpha}\cdot w_{\lb} \mid \alpha=\alpha_{i,j} \text{ or } \alpha_{i,2\ell-j}, 1\leq i<\ell\}, \text{ when } \mf{g} \text{ is } A_{2\ell},
		\end{align*}
		where $s_\alpha=c\cdot (\lambda, \bar{\alpha}^\vee)$, and $w_{\bar \lambda}$ is the generator whose $\mf{h}^\sigma$-weight is $-c\iota(\bar\lambda)$. Here, $(, )$ is the normalized Killing form on $\mf{g}$ and $\bar{\alpha}^\vee$ is regarded as an element in $\mf{h}^\sigma$.
	\end{enumerate}
\end{thm}
\begin{proof}
	See \cite[Lemma 4, Corollary 1]{FL06} for part (1), and \cite[Theorem 3, Theorem 5]{KV16} for part (2).
\end{proof}
\begin{cor}\label{623012} 
	With the same setup and assumption as in Theorem \ref{401910}, we have
	\begin{enumerate}
		\item For any $\lambda\in X_*(T)^+$, the Demazure module $D(1,\lambda)$ is isomorphic to the Weyl module $W(\iota(\lambda))$. 
		\item For any $\bar{\lambda}\in X_*(T)_\sigma^+$, the twisted Demazure module $D^\sigma(1,\bar{\lambda})$ is isomorphic to the twisted Weyl module $W^\sigma(\iota(\bar{\lambda}))$. 
	\end{enumerate}
\end{cor}
\begin{proof}
	For part (1), see \cite[Theorem 7]{FL06}. For part (2), see \cite[Theorem 5.1]{FK13} and \cite[Theorem 9(ii)]{CIK14} when $\mf{g}$ is not type $A_{2\ell}$; see \cite[Theorem 2]{CIK14} and \cite[Corollary 3.3]{KV16} when $\mf{g}$ is $A_{2\ell}$.
\end{proof}
\begin{remark}
	When $\mf{g}$ is of type $A_{2\ell}$, part (2) of Corollary \ref{623012} does not hold for the diagram automorphism. This is one of the reasons that we use the `standard' automorphism $\sigma$ instead of the diagram automorphism.
\end{remark}
In view of (\ref{eq_match_wt}), this corollary implies that any Weyl module of $\gs$ is isomorphic to a twisted Demazure module.

\section{Global Demazure modules of twisted current algebra}\label{716168}
In this section, we define twisted global Demazure modules and study their basic properties. We also prove that when the level is $1$, the global Demazure modules can be identified with twisted global Weyl modules. 
\subsection{General twisted global modules}\label{109490}
Let $V$ be a cyclic graded finite dimensional $\mf{g}[t]$-module with a cyclic vector $v_\mu$ of weight $\mu\in P^+$ such that 
\begin{equation}\label{324482}
	t\mf{h}[t]\cdot v_\mu=0,
\end{equation}
and the degree of $t$ is $1$. For any $x\in \mf{g}$ and $s\geq0$, we denote by $xt^s$ the element $x\otimes t^s$ in $\mf{g}[t]$. We denote by $v[f]$ the vector $v\otimes f$ in the vector space $V[z]:=V\otimes \C[z]$, for any $v\in V$ and $f\in \mathbb{C}[z]$. Following \cite[subsection 1.3]{FM19}, we define a $\mf{g}[t]$-module structure on $V[z]$ as follows,
\begin{equation}\label{554908}
	xt^s \bu v[f]=\sum_{j=0}^{s} \binom{s}{j} ( xt^j\cdot v)[f\cdot z^{s-j}]
\end{equation}
where $s\geq 0, x\in \mf{g}, v[f]\in V[z]$. It was shown in \cite[Section 1.3]{FM19} that $V[z]$ is a cyclic module with the cyclic generator $v[1]$.

Consider the cyclic graded finite dimensional $\mf{g}[t]$-modules $V_1,\ldots, V_n$, where $V_i$ has a cyclic vector $v_i$ of weight $\mu_i$ such that $t\mf{h}[t]\cdot v_i=0.$ Set 
\begin{equation}
\label{eq_w}
\textstyle w:=v_1[1]\otimes v_2[1]\otimes \cdots \otimes v_n[1]\in \bigotimes_i V_i[z_i]. \end{equation}
\begin{definition}
	Similar to the untwisted case in \cite{DF21}, we define the twisted global module $R^\sigma(V_1,\ldots, V_n)$ to be a cyclic $\U(\gs)$-module generated by the vector $w$ as follows,
\[R^\sigma(V_1,\ldots, V_n):=\U(\gs) \bu w \subseteq \bigotimes_{i=1}^n V_i[z_i].\]
Moreover, we define a right $\U(\hs)$-action on $R^\sigma(V_1,\ldots, V_n)$ by
\begin{equation}\label{936369}
	(u\bu w)\ h=u \bu (h\bu w),
\end{equation}
where $u\in \U(\gs), h\in\U(\hs)$.
\end{definition}
\begin{lem}
\label{lem_well_def}
	This action \eqref{936369} is well-defined.
\end{lem}
\begin{proof}
	It suffices to show that if $u\bu w=0$ for some $u\in \U(\gs)$, then 
	\[u\bu (ht^{s}\bu w)=0\]
	for any $ht^s\in \hs$. It follows from the definition \eqref{554908} with the following calculations:
	\begin{align}\label{423239}
		\begin{split}
		u\bu (ht^{s}\bu w) 
		& = u\bu \big(ht^{s}\bu ( v_1[1]\otimes\cdots\otimes v_n[1]) \big)\\
		& = u\bu \bigg(\sum_{i=1}^n v_1[1]\otimes \cdots \otimes (ht^{s}\bu v_i[1])\otimes \cdots\otimes v_n[1]\bigg)\\
		& = u\bu\bigg(\sum_{i=1}^n \mu_i(h) v_1[1]\otimes \cdots \otimes v_i[z_i^{s}]\otimes \cdots\otimes v_n[1]\bigg)\\
		& = \sum_{i=1}^n  \mu_i(h) u\bu \big(v_1[1]\otimes \cdots \otimes v_i[z_i^{s}]\otimes \cdots\otimes v_n[1]\big)\\
		& = 0 \qedhere
		\end{split}
	\end{align}
\end{proof}
From the proof of Lemma \ref{lem_well_def}, we know the action \eqref{936369} commutes with the $\gs$-action, and only depends on the weights $\mu_1,\ldots,\mu_n$. Set $\vec{\mu}=(\mu_1,\ldots,\mu_n)$ and we define
\begin{equation}\label{749134}
	\As(\vec\mu):=\U(\hs)/\mr{Ann}(w)
\end{equation}
as the quotient of $\U(\hs)$ by the annihilator of $w$ defined in (\ref{eq_w}). Following \cite[Section 2.2]{DFF21}, we call $\As(\vec\mu)$ a twisted weight algebra associated to $w$. Then, the twisted global module $R^\sigma(V_1,\ldots, V_n)$ is a $(\gs,\As(\vec\mu))$-bimodule. 

Consider the following embedding
\begin{equation}\label{234983}
	\As(\vec\mu)\simeq \As(\vec\mu)\bu w\hookrightarrow \C[z_1,z_2,\ldots, z_n],
\end{equation}
given by $ht^s\mapsto \mu_1(h)z_1^s+\cdots +\mu_n(h) z_n^s$. Then, $\As(\vec\mu)$ can be realized as the algebra generated by the following elements
\[\{\mu_1(h)z_1^{ms+j} +\mu_2(h)z_2^{ms+j} +\cdots +\mu_n(h)z_n^{ms+j} \mid h\in \mf{h}_j, 0\leq j \leq m-1, s\geq 0\},\]
where $\mf{h}_j:=\{h\in \mf{h} \mid \sigma(h)=\epsilon^j h\}$. 

\begin{example}\label{ex_algebra}
	Let $v$ be a vector of weight $\omega_j$, where $j\in I$. For any $ht^s\in \hs$, $ht^s\bu v[1]=\omega_j(h)v[z^s]$ is non-zero if and only if $m_j\mid s$, where $m_j$ is the order of the stablizer of the vertex $j$ under the action of the cyclic group $\langle \sigma\rangle$. Thus, $\As(\omega_j)\simeq \C[z^{m_j}]$. In particular, when $\mf{g}$ is $A_{2\ell}$, $m_j=2$ for any $j$.
\end{example}

Similar to the untwisted case in \cite{DF21}, we have the following result.
\begin{prop}
	$\As(\vec\mu)$ is a Noetherian ring. 
\end{prop}
\begin{proof}
	Consider an element $h\in\mf{h}^\sigma$ such that $(\mu_{i_1}+\cdots+\mu_{i_k})(h)\neq 0$ for any subset $\{i_1,\ldots i_k\}\subset \{1,\ldots, n\}$. Let $\mc{A}^\sigma(\vec\mu)_h$ be the subalgebra of $\mb{C}[z_1,z_2,\ldots, z_n]$ generated by
	\[\{\mu_1(h)z_1^{ms} +\mu_2(h)z_2^{ms} +\cdots +\mu_n(h)z_n^{ms} \mid s\geq 0\}.\]
	Then we have the following embeddings
	\[\mc{A}^\sigma(\vec{\mu})_h \hookrightarrow \As(\vec{\mu}) \hookrightarrow \bigotimes_{i=1}^n \mc{A}^\sigma(\mu_i). \]
	It was shown in \cite[Proposition 2.6]{BCES16} that $\bigotimes_{i=1}^n \mc{A}^\sigma(\mu_i)$ is finite over $\mc{A}^\sigma(\vec\mu)_h $, and hence is finite over $\As(\vec\mu)$. Note that $\mc{A}^\sigma(\vec\mu)_h $ is Noetherian. By Artin-Tate Lemma \cite[Proposition 7.8]{AM69}, $\As(\vec\mu)$ is finitely generated over $\mc{A}^\sigma(\vec\mu)_h $. Thus, $\As(\vec\mu)$ is also Noetherian.
\end{proof}
\begin{remark}
	The definitions and results for $R^\sigma(V_1,\ldots, V_n)$ and $\As(\vec{\mu})$ in this section also works when every $\mu_i$ is anti-dominant.
\end{remark}
%In the rest part of this subsection, we discuss the fibers of twisted global modules. 

For a point $p\in \A^1-\{0\}$ and a $\mf{g}[t]$-module $V$, we denote by $V_p$ the module $V$ with a new $\mf{g}[t]$-action with a shift by $p$, i.e. it is given by 
\begin{equation}\label{342834}
	xt^s\bu v:=x(t+p)^s\cdot v,
\end{equation}
where $xt^s\in \mf{g}[t]$, $v\in V$.
\begin{prop} 
\label{prop_cyclic_mod}
	Let $V_1$ be a cyclic finite dimensional graded $\gs$-module, $V_2,\ldots, V_n$ be cyclic finite dimensional graded $\mf{g}[t]$-modules. Let $\vec{p}\in \A^n$ such that $p_1=0, p_i^m\neq p_j^m$ for any $i\neq j$. Then, the tensor product $V_1\otimes \bigotimes_{i=2}^n V_{i,{p_i}}$ is a cyclic $\U(\gs)$-module.
\end{prop}
\begin{proof} 
	See \cite[Section 6.3]{KV16}. One can also prove this proposition by using a similar method in \cite[Proposition 1.4]{FeLo99}.
\end{proof}

When $V_1$ is a trivial module, we get the following special case.
\begin{cor}\label{906266}
	Let $\vec{p}\in \A^n$ such that $p_i^m\neq p_j^m\neq 0$ for any $i\neq j$. The tensor product $\bigotimes_{i=1}^n V_{i,p_i}$ is a cyclic $\U(\gs)$-module. \qed
\end{cor}
The $\U(\gs)$-module $ \bigotimes_{i=1}^n V_{i,{p_i}}$ is cyclic, and hence it has a filtration induced from the $t$-grading on $\gs$. Let $ \mr{gr}\, \big(\bigotimes_{i=1}^n V_{i,{p_i}}\big)$ denote the associated graded $\gs$-module of $ \bigotimes_{i=1}^n V_{i,{p_i}}$. 

\begin{prop}\label{063500}
	Let $\vec{p}\in \A^n$ such that $p_i^m\neq p_j^m\neq 0$ for any $i\neq j$.  For each $1\leq i\leq n$, let $V_i$ be a cyclic graded finite dimensional $\mf{g}[t]$-module with a cyclic vector $v_{i}$ of weight $\mu_i\in P^+$ such that $t\mf{h}[t]\cdot v_i=0$.  There is a surjective morphism of $\gs$-modules
	\begin{equation}\label{719856} 
		R^\sigma(V_1,\ldots,V_n)\otimes_{\As(\vec\mu)} \C_{\vec{p}}\twoheadrightarrow \bigotimes_{i=1}^n V_{i,{p_i}}.
	\end{equation}
\end{prop}
\begin{proof}
	Consider the composition of morphisms of $\gs$-modules
	\[ R^\sigma(V_1,\ldots,V_n)\otimes_{\As(\vec\mu)} \C_{\vec{p}}\to \bigg(\bigotimes_{i=1}^n V_i[z_i]\bigg)\otimes_{\As(\vec\mu)} \C_{\vec{p}} \twoheadrightarrow \bigotimes_{i=1}^n V_{i,{p_i}}.\]
	The proposition follows from Corollary \ref{906266} which asserts that the right hand side is a cyclic $\gs$-module.
\end{proof}
We will show in Theorem \ref{616072} that the map \eqref{719856} is an isomorphism when $V_i$ are Demazure modules. The following theorem is a special case of \cite[Theorem 7]{KV16}, which will be crucially used in Section \ref{320070}. 
\begin{thm}\label{257230}
	Let $\vec{p}\in \A^n$ such that $p_i^m\neq p_j^m\neq0$ for any $i\neq j$. Given $\lambda_1,\ldots,\lambda_n \in X_*(T)^+$, let $\bar{\lambda}$ be the image of $\sum_i\lambda_i $ in $X_*(T)_\sigma$. Then, there is an isomorphism of $\gs$-modules
	\[ \pushQED{\qed} D^\sigma(c,\bar{\lambda})\simeq \mr{gr}\,\bigg( \bigotimes_{i=1}^n D(c,\lambda_i)_{p_i}\bigg). \qedhere\popQED \]
\end{thm}

\subsection{Twisted global Demazure modules}\label{320070}
In this subsection, we define the global Demazure modules of twisted current algebras, and study their basic properties. 

Given $\la=(\lambda_1,\ldots,\lambda_n)\in (X_*(T)^+)^n$, for each $i$, let $v_i$ be the cyclic generator of the Demazure module $D(c,\lambda_i)$ of $\mf{g}[t]$ which is of weight $-c\iota(\lambda_i)$, as in $(\ref{eq_Demazure})$. Recall Section \ref{109490}, one can define the following twisted global module 
\begin{equation*}
R^\sigma\big( D(c,\lambda_1), \ldots, D(c,\lambda_n)\big)=\U(\gs)\bu w,
\end{equation*}
where $w=v_{1}[1]\otimes\cdots\otimes v_{n}[1]\in \bigotimes D(c,\lambda_i)[z_i]$. Moreover, this global module has a right $\mc{A}^\sigma(c,\la)$-action, where $\As(c,\la)=\U(\hs)/\mr{Ann}(w)$ is the twisted weight algebra associated to $w$. We remark that $\As(c,\la)=\As(-c\iota(\la))$, see \eqref{749134}. By \eqref{234983}, we have an embedding $\As(c,\la)\hookrightarrow \C[\A^n]:=\C[z_1,\ldots,z_n]$
\begin{definition}\label{234798}
	The twisted global Demazure module $\Ds(c,\la)$ is defined to be the following $(\gs, \As(c,\la))$-bimodule
	\begin{equation}\label{eq_D_tilde}
		 \Ds(c,\la):= R^\sigma\big( D(c,\lambda_1), \ldots, D(c,\lambda_n)\big).
	\end{equation}
	Let $\Ds(c,\la)_{\vec{p} }:=\Ds(c,\la)\otimes_{\As(c,\la)} \C_{\vec{p}}$ denote the fiber at a point $\vec{p}\in \A^n$.
\end{definition}
We now determine the generic fibers of the twisted global Demazure module.
\begin{lem}
	With the same setup as in Theorem \ref{257230}, there is a surjective $\gs$-morphism
\begin{equation}\label{485770}
	\Ds(c,\la)_{\vec{p}} \twoheadrightarrow \bigotimes_{i=1}^n D (c, \lambda_i )_{p_i}. 
\end{equation}
As a consequence, there is an inequality
\begin{equation}\label{394520}
	\dim \big(\Ds(c,\la)_{\vec{p}}\big)\geq \dim \bigg(\bigotimes_{i=1}^n D (c, \lambda_i )_{p_i} \bigg) = \dim \big( D^{\sigma} (c, \bar{\lambda}) \big) .
\end{equation}
\end{lem}
\begin{proof}
	The surjective map is given by Proposition \ref{063500}. The inequality follows from the surjection \eqref{485770} and Theorem \ref{257230}.
\end{proof}
\begin{thm}\label{398063} 
	Given $\lambda_1,\ldots,\lambda_n \in X_*(T)^+$, let $\bar{\lambda}$ be the image of $\lambda:=\sum_i\lambda_i $ in $X_*(T)^+_\sigma$. There is a $\gs$-isomorphism $\Ds(c,\la)_0 \simeq D^\sigma(c,\lb)$.
\end{thm}
 Our proof is similar to the proofs in \cite[Proposition 4.9]{DFF21} and \cite[Proposition 2.13]{DF21}.
\begin{proof}
	It is known that the function $\varphi(\vec{p}):=\dim \big(\Ds(c,\la)_{\vec{p}}\big)$ is upper semicontinuous, cf.\,\cite[Expample 12.7.2]{Hartshorne77}. Note that for any $\vec{p}\neq 0$, the function $\varphi$ is a constant on the punctured line $\{a\vec{p}\mid a\in \C^*\}$. Thus,
	\begin{equation}\label{213578}
		\varphi({0})\geq \varphi(\vec{p}), \text{ for any } \vec{p}\in \A^n.
	\end{equation}
	Combining with \eqref{394520}, we get
	\begin{equation}\label{634830}
		\dim \big(\Ds(c,\la)_0\big)\geq \dim \big(D^\sigma(c,\lb)\big).
	\end{equation}
	From Theorem \ref{401910}, $D^\sigma( c, \bar{\lambda} )$ is isomorphic to the quotient of twisted Weyl module $W^{\sigma}(c\iota(\bar{\lambda}))$ by a submodule $M$, where $M$ is generated by the elements in Theorem \ref{401910} part (2). To prove the theorem, it suffices to show there is a surjective map 
	\[ D^\sigma(c,\lb) \simto W^{\sigma}(c\iota(\bar{\lambda}))/M \twoheadrightarrow \Ds(c,\la)_0.\]

	We first show that the map 
	\begin{equation}\label{428476}
		\psi: W^{\sigma}(c\iota(\bar{\lambda})) \twoheadrightarrow \Ds(c,\la)\otimes_{\As(c,\la)} \C_0
	\end{equation}
	sending $w_{\bar{\lambda}}$ to $w\otimes 1$ is well-defined. It is clear that the first three defining relations in \eqref{989363} for $w_{\bar{\lambda}}$ also holds for $w\otimes 1$. The reason is that $w=v_1[1]\otimes v_2[1]\otimes \cdots \otimes v_n[1]$ and by part (1) of Theorem \ref{401910} $v_i$ satisfies the following relations
	\begin{equation}\label{428623}
		\mf{n}^-[t]\cdot v_i=0, \ \mf{h}\otimes t\C[t]\cdot v_i=0 ,\ (h\otimes 1)\cdot v =-c(\lambda_i,h)v_i,\ (e_{\alpha}\otimes t^s)^{k_{\alpha,s}+1}\cdot v_i=0, 
	\end{equation}
	where $\alpha \in R^+$, $k_{\alpha,s}=c\max\{0,(\lambda,\alpha^\vee)-s\}$. 

	We now check the last relation in \eqref{989363}. When $\mf{g}$ is not $A_{2\ell}$, for any $\alpha\in R^+$, if $\alpha$ is not fixed by $\sigma$, then the roots $\alpha, \sigma(\alpha),\ldots, \sigma^{m-1}(\alpha)$ are all distinct and orthogonal, cf.\,\cite[Lemma 1]{Springer06}. In this case, we have $k_\alpha=c(\lambda,\bar{ \alpha}^\vee)\geq c\sum_{j=0}^{m-1} ({\lambda},\ {\sigma^{j}(\alpha^\vee)})$, and
	\begin{equation}\label{eq_last_rel}
		(e_{\alpha,0})^{k_\alpha+1}\bu w= \bigg(\sum_{j=0}^{m-1} e_{\sigma^j(\alpha)}\otimes 1\bigg)^{k_\alpha+1}\bu w=\sum_{\sum s_j =k_\alpha+1} \prod_{j=0}^{m-1} (e_{\sigma^j(\alpha)}\otimes 1)^{s_j} \bu w.
	\end{equation}
	In each term of the right side, there must be an index $j$ such that $s_j\geq c(\lambda, \sigma^{j}(\alpha^\vee))+1=k_{\sigma^j(\alpha),0}+1$. For this $s_j$, by the last relation in \eqref{428623}, we have $(e_{\sigma^j(\alpha)}\otimes 1)^{s_j} \bu (v_i[1])=0$, $1\leq i\leq n$. This shows that the right side of \eqref{eq_last_rel} is $0$. Hence, the last relation in \eqref{989363} holds in this case. When $\alpha$ is fixed by $\sigma$, this relation can be checked similarly (actually easier). 
	
	When $\mf{g}$ is of type $A_{2\ell}$ and $m=4$, the element $e_{\alpha,0}\in \gs$ if and only if $\alpha\neq\alpha_{i,\ell}$, see Section \ref{133740}. Thus, $\alpha$ and $\sigma(\alpha) $ are either equal or orthogonal. By the same argument as above, we still have $(e_{{\alpha},0})^{k_\alpha+1}\bu w= 0$, where $\alpha\neq\alpha_{i,\ell}$ and $k_\alpha=c(\lambda, \bar{\alpha}^\vee)$. This finishes the checking of all relations in \eqref{989363}. Thus, the map \eqref{428476} is well-defined.

	To complete the proof, we still need to show $\psi(M)=0$. When $\mf{g}$ is not $A_{2\ell}$, the submodule $M$ is generated by the following elements (see Theorem \ref{401910})
	\[\{e_{\alpha,m s_{\alpha}}\cdot w_{\lb} \mid \alpha \in R^+_{l} \} \cup \{ e_{\alpha, s_{\alpha} } \cdot w_{\lb} \mid \alpha \in R^+_{s}\}. \]
	By the equalities in \cite[(5.1.1)]{Kac90} and \eqref{328480}, we always have $s_\alpha=c(\lambda,\bar{ \alpha}^\vee)\geq c(\lambda, \sigma^j(\alpha^\vee))$ for any $\alpha\in R^+$ and $0\leq j\leq m-1$. When $\alpha\in R^+_l$, by definition $\bar{\alpha}$ is a long root of $\mf{g}^\sigma$. Then, $(\lambda,\sigma^j(\alpha^\vee))-ms_{\alpha}\leq (\lambda,\sigma^j(\alpha^\vee))-mc({\lambda},\ {\sigma^{j}(\alpha^\vee) })\leq 0$, which implies
	\begin{equation}\label{eq_psi_M}
	e_{\alpha,ms_\alpha}\bu w= \sum_{i=0}^{m-1} e_{\sigma^i(\alpha)}\otimes t^{ms_\alpha}\bu w=0.
	\end{equation}
	Similarly, one can show that $ e_{\alpha,s_\alpha}\bu w=0$ for any $\alpha\in R^+_s$. Thus, $\psi(M)=0$.

	When $\mf{g}$ is of type $A_{2\ell}$, the submodule $M$ is generated by the following elements
	\begin{align*}
		&\{e_{\alpha,4s_\alpha}\cdot w_{\lb} \mid \alpha=\alpha_{i,2\ell+1-i}, 1\leq i<\ell \} \cup \{e_{\alpha, 2s_{2\alpha}+1}\cdot w_{\lb} \mid \alpha=\alpha_{i,\ell}, 1\leq i\leq \ell \}\\
		&\cup \{e_{\alpha,2s_\alpha}\cdot w_{\lb} \mid \alpha=\alpha_{i,j} \text{ or } \alpha_{i,2\ell-j}, 1\leq i<\ell\} .
	\end{align*}
	When $\alpha\neq \alpha_{i,\ell}$, the relations can be checked similarly as in (\ref{eq_psi_M}). Now, suppose $\alpha=\alpha_{i,\ell}$, and let $s=s_{2\alpha}=c(\lambda,\alpha^\vee+\tau(\alpha^\vee))$. We have
	\[(e_{\alpha, 2s+1})\bu w= (e_{\alpha}\otimes t^{2s+1} + (-1)^{i+\ell+s}e_{\tau(\alpha)}\otimes t^{2s+1})\bu w=0,\]
	since $(\lambda,\tau^p(\alpha^\vee))-(2s+1)=(\lambda,\tau^p(\alpha^\vee) )-2 c(\lambda, \alpha^\vee+\tau(\alpha^\vee))-1< 0$ for $p=0$, or $1$. Thus, in this case we also have $\psi(M)=0$.
	
	Now, the map $\psi$ induces a surjective morphism 
	\[\bar\psi: D^\sigma(c,\lb) \simto W^{\sigma}(c\iota(\bar{\lambda}))/M \twoheadrightarrow \Ds(c,\la)_0.\]
	By the dimension inequality \eqref{634830}, $\bar\psi$ must be an isomorphism.
\end{proof}
\begin{thm}\label{616072}
	 Given $\lambda_1,\ldots,\lambda_n \in X_*(T)^+$, let $\bar{\lambda}$ be the image of $\sum_i\lambda_i $ in $X_*(T)^+_\sigma$. 
	\begin{enumerate}
		\item For any $\vec{p}\in \A^n$ such that $p_i^m\neq p_j^m\neq0$ for any $i\neq j$, there is an isomorphism of $\gs$-modules
		\[\Ds(c,\la)_{\vec{p}} \simeq \bigotimes_{i=1}^n D (c, \lambda_i )_{p_i}.\]
		\item The twisted global Demazure module $\Ds(c,\la)$ is free over $\As(c,\la)$.
	\end{enumerate}
\end{thm}
\begin{proof}
	By Theorem \ref{398063}, we have $\dim \big(\Ds(c,\la)_0\big)= \dim \big( D^{\sigma} (c, \bar{\lambda}) \big)$. Combining with \eqref{213578}, we get
	\begin{equation}\label{528923}
			\dim \big( D^{\sigma} (c, \bar{\lambda}) \big) \geq \dim \big(\Ds(c,\la)_{\vec{p}}\big) 
	\end{equation}
	for any $\vec{p}\in\A^n$. It follows that all three terms in the inequality \eqref{394520} are equal. Hence the surjection \eqref{485770} is an isomorphism. This proves part (1).
	
	From part (1), we get $\dim \big(\Ds(c,\la)_{\vec{p}}\big)= \dim \big( D^{\sigma} (c, \bar{\lambda}) \big)$ for any $\vec{p}$ such that $p_i^m\neq p_j^m\neq 0, \forall i\neq j$. By the semicontinuity of the function $\varphi(\vec{p}):=\dim \big(\Ds(c,\la)_{\vec{p}}\big)$, we get
	\[ \dim \big(\Ds(c,\la)_{\vec{p}}\big)\geq \dim \big( D^{\sigma} (c, \bar{\lambda}) \big)\]
	for any $\vec{p}\in \A^n$. Combining with the inequality \eqref{528923}, the function $\varphi(\vec{p})$ must be constant. Thus, $\Ds(c,\la)$ is a graded projective $\As(c,\la)$-module. By the graded Nakayama lemma (cf.\,\cite[\href{https://stacks.math.columbia.edu/tag/0EKB}{Tag 0EKB}]{stacks-project}), $\Ds(c,\la)$ is a free $\As(c,\la)$-module.
\end{proof}
\begin{remark}
In Section \ref{753118}, we will use the geometry of Beilinson-Drinfeld Schubert varieties to fully describe the fibers of $\Ds(c,\la)$ at any point in $\A^n$, see Corollary \ref{367230}.
\end{remark}
We now introduce some notations that will be frequently used throughout this paper. 
\begin{definition}\label{524794}
	Denote by $[n]$ the finite set $\{1, 2,\ldots, n\}$. Let $\xi=(I_1, I_2,\ldots, I_k)$ be a partition of $[n]$. We define $\A^n_{\xi}$ to be the following open subset in $\A^n$,
	\begin{equation}\label{345790}
		\A^n_{\xi}=\{ \vec{p}\in\A^n \mid {p}_i^m \neq {p}_j^m,\ \forall i\in I_\alpha, j\in I_\beta, \alpha\neq\beta \}.
	\end{equation}
	For any subset $I=\{i_1,\ldots,i_s\}\subseteq [n]$, set $\A^{I} ={\rm Spec\,}\mathbb{C}[z_{i_1}, \ldots, z_{i_s}] $. 
\end{definition}
\begin{remark}
In this paper, the partition $\xi=(I_1,\cdots, I_k)$ is always ordered by lexicographical order, i.e. $\min(I_1)<\min(I_2)<\cdots <\min(I_k)$.
\end{remark}
\begin{example}
	When $\xi=([n])$ is the trivial partition of $[n]$, we have $\A^n_{\xi}=\A^n$; when $\xi_0=(\{1\},\{2\},\ldots,\{n\})$ is the finest partition, we have $\A^n_{\xi_0}=\{ \vec{p}\in\A^n \mid \bar{p}_i\neq \bar{p}_j,\ \forall i\neq j \}$.
\end{example}

Given a partition $\xi=(I_1, I_2,\ldots, I_k)$ of $[n]$, for any $1\leq j\leq k$, we label elements in $I_j=\{j_1,\ldots,j_s\}$ by the natural order of integers, i.e.\,$j_1<j_2<\cdots <j_s$, where $s$ is the cardinality of $I_j$.
For any point $\vec{p}=(p_1,\ldots,p_n) \in \A^n_{\xi}$ and $\la\in (X_*(T)^+)^n$, we make the following notations:
\begin{gather}
	\label{053076} \vec{p}_{I_j}=(p_{j_1},\ldots,p_{j_s})\in\A^{I_j},\\
	\label{249296} \vec{\lambda}_{I_j}=(\lambda_{j_1},\ldots,\lambda_{j_s}),\\
	\nonumber w_{I_j}=v_{j_1}[1]\otimes \cdots\otimes v_{j_s}[1]\in \bigotimes_{i=1}^s D(c,\lambda_{j_i})[z_{j_i}].
\end{gather}
For simplicity, we denote by $w_{I_j}[1]$ the generator $w_{I_j}\otimes 1$ of \[\Ds(c,\la_{I_j})_{\A^{I_j}}:=\Ds(c,\la_{I_j})\otimes_{\As(c,\la_{I_j})} \C[\A^{I_j}].\]
\begin{prop}\label{348932} 
	\begin{enumerate}
		\item For any $\lambda\in X_*(T)^+$, there is an embedding of $\gs$-modules: 
		\[\phi: \Ds(c,\lambda)_{\A^1}\hookrightarrow D(c, \lambda)[z]\]
		defined by $(u\bu w)[f]\mapsto u\bu (v[f])$ for any $u\in \gs$ and $f\in\C[z]$, where $w:=v[1]$ and $v$ is the cyclic generator of $D(c, \lambda)$ as in (\ref{eq_Demazure}).
		\item Let $\xi=(I,J)$ be a partition of $[n]$. There is an embedding of $\gs$-modules: 
		\begin{equation}\label{eq_psi}
		\psi_\xi:\Ds(c,\la)_{\A^n}\otimes_{\C[\A^n]} \C[\A^n_{\xi}]\hookrightarrow \big(\Ds(c,\vec{\lambda}_I)_{\A^I} \otimes \Ds(c,\vec{\lambda}_J)_{\A^J} \big)\otimes_{\C[\A^I\times \A^J]} \C[\A^n_{\xi}]\end{equation}
 		defined by $(u\bu w)[f]\otimes g\mapsto \big((u\bu w_I)[1]\otimes (u\bu w_J)[1]\big)\otimes (fg)$ for any $u\in\gs$, $f\in \C[\A^n]$, and $g\in \C[\A^n_\xi]$.
	\end{enumerate}
\end{prop}
\begin{proof}
	(1) By Theorem \ref{616072}, the map $\phi$ restricts to an isomorphism $\Ds(c,\lambda)_p \simeq D(c,\lambda)_p$ at any $p\in \A^1\backslash \{0\}$. Thus, $\phi$ is injective. 
	
	(2) The map $\psi_\xi$ is well defined since $u$ acts on $w$ via Leibniz rule. Again, by Theorem \ref{616072}, it restricts to an isomorphism 
	\[\Ds(c,\la)_{\vec{p}}\simeq \Ds(c,\la_{I})_{\vec{p}_I}\otimes \Ds(c,\la_{J})_{\vec{p}_J},\]
	for any $\vec{p}$ such that $p_i^m\neq p_j^m\neq0, \forall i\neq j$. It follows that the map $\psi_\xi$ is injective.
\end{proof}
\begin{remark}
	In Section \ref{753118}, we will show that the injective map $\psi_\xi$ is actually an isomorphism, see Theorem \ref{055073}.
\end{remark}

\subsection{Twisted global Weyl modules}\label{sect_t_weyl}
In this subsection, we identify the twisted global Weyl module defined in Section \ref{427932} with the level one twisted global Demazure module. 

Taking $c=1$ in Theorem \ref{398063} and Theorem \ref{616072} and applying Corollary \ref{623012}, we get the following results.
\begin{lem}\label{342893}
	Given a tuple of weights $\vec\mu=(\mu_1,\ldots,\mu_n)\in (P^+)^n$, let $W(\mu_i)$ be the local Weyl module defined in Section \ref{427932}. Then
	\begin{enumerate}
		\item The  global module $R^\sigma(W(\mu_1),\ldots,W(\mu_n))$ is a graded free $\As(-\vec\mu)$-module.
		\item Let $\bar\mu\in P_\sigma^+$ be the restriction of $\sum \mu_i$ to $\mf{h}^\sigma$. There is a $\gs$-isomorphism 
		\[\pushQED{\qed} R^\sigma(W(\mu_1),\ldots,W(\mu_n))\otimes_{\As(-\vec\mu)}\C_0\simeq W^\sigma(\bar\mu). \qedhere\popQED \]
	\end{enumerate}
\end{lem}

\begin{prop}\label{234924}
	Given any fundamental weights $\mu_1,\ldots,\mu_n\in P^+$, let $\bar\mu\in P_\sigma^+$ be the restriction of $\sum \mu_i$ to $\mf{h}^\sigma$. Then, there is a $\gs$-isomorphism
	\[\mb{W}^\sigma(\bar\mu)\simeq R^\sigma(W(\mu_1),\ldots,W(\mu_n)),\]
	and an isomorphism $A^\sigma(\bar\mu)\simeq \As(-\vec{\mu})$ of $\C$-algebras, where $\vec{\mu}=(\mu_1,\ldots,\mu_n)$.
\end{prop}
\begin{proof}
	Let $w_i\in W(\mu_i)$ (resp.\,$w_{\bar\mu}\in \mb{W}^\sigma(\bar\mu)$) denote the cyclic generator with weight $-\mu_i$ (resp.\,$-\bar\mu$). Similar to the map (\ref{428476}), we have a surjective $\gs$-morphism
	\begin{equation}\label{234571}
		\mb{W}^\sigma(\bar\mu)\twoheadrightarrow R^\sigma(W(\mu_1),\ldots,W(\mu_n)),
	\end{equation}
	given by $w_{\bar{\mu}}\mapsto w:=w_1[1]\otimes\cdots\otimes w_n[1]$. This induces a surjective map
	\begin{equation}\label{324821}
		A^\sigma(\bar\mu)=\U(\hs)/\mr{Ann}(w_{\bar\mu})\twoheadrightarrow \U(\hs)/\mr{Ann}(w)=\As(-\vec\mu).
	\end{equation}
	We claim this is an isomorphism. We only prove the case when $\mf{g}$ is of type $A_{2\ell}$. The other types can be checked similarly.
	
	Recall Theorem \ref{350320} that $A^\sigma(\bar\mu)$ is a a graded polynomial algebra in variables $T_{i,r}$ of grade $2r$, where $1\leq i\leq \ell$ and $1\leq r\leq \bar\mu(\beta_i^\vee)$. Set $n_i:=\bar\mu(\beta_i^\vee)$. Then, $n=\sum_{i=1}^\ell n_i$. In view of \cite[Section 5.2 and 5.3]{CIK14}, $T_{i,r}$ is the image of $P_{i,r}\in \U(\hs)$ which is defined as follows. For any $1\leq i\leq \ell$, set $P_{i,1}:=h_{i,2}$, which is defined in \eqref{802623}. For $r\geq 2$, we inductively define 
	\[P_{i,r}:=\frac{1}{r}\sum_{s=0}^r h_{i,2s+2}P_{i,r-s-1}.\]
	Now, we determine the image of $T_{i,r}$ under the map \eqref{324821} by examining the action of $P_{i,r}$ on $w$, where $1\leq i\leq \ell$ and $1\leq r\leq n_i$. Suppose $n_i\geq 1$, otherwise $P_{i,r}$ will not appear. From \cite[Section 5.4]{CIK14}, we have 
	\begin{equation}\label{324983}
		P_{i,r}\bu w=\sum_{j_1+\cdots+j_n=r} \big(P_{i,j_1}\bu w_1[1]\big)\otimes\cdots\otimes \big(P_{i, j_n}\bu w_n[1]\big),
	\end{equation}
	where $j_k\geq 0$ and by convention $P_{i,0}=1$. Observe that $h_{i,2s+2}\bu w_k[1]=\mu_k(h_i+(-1)^{s+1}h_{2\ell+1-i})w_k[z_k^{2s+2}]$ is nonzero only when $\mu_k=\omega_i \text{ or } \omega_{2\ell+1-i}$. Without loss of generality, suppose $\mu_k=\omega_i$ for $1\leq k\leq n_i'$, and $\mu=\omega_{2\ell+1-i}$ for $n_i'+1\leq k\leq n_i$, where $n_i'$ is the number of $\omega_i$ appearing in $\vec{\mu}$. Then, $P_{i,j_k}\bu w_k[1]=0$ whenever $j_k>1$, or $k>r$ and $j_k\geq 1$. Thus, the equation \eqref{324983} becomes 
	\begin{equation}\label{432983}
		 P_{i,r}\bu w= \sum_{j_1+\cdots+j_{n_i}=r} (-1)^{a}\big(w_1[z_1^{2j_1}]\otimes\cdots\otimes w_{n_i}[z_{n_i}^{2j_{n_i}}]\otimes w_{n_i+1}[1]\otimes \cdots\otimes w_{n}[1]\big), 
	\end{equation}
	where $0\leq j_k\leq 1$ and $a=j_{n_i'+1}+\cdots+j_{n_i}$. 
	
	Recall the embedding $\As(-\vec\mu)\hookrightarrow \C[z_1,\ldots,z_n]$ defined in \eqref{234983}. Composing it with the map \eqref{324821}, we get a morphism 
	\begin{equation}\label{298232}
		A^\sigma(\bar\mu)\to \C[z_1,\ldots,z_n].
	\end{equation}
	Given $1\leq i\leq \ell$, we define $K_i= \{1\leq k\leq n\mid \mu_k=\omega_i\text{ or } \omega_{2\ell+1-i}\}$, which is of cardinality $n_i$. The symmetric group $S_{n_i}$ permutes $K_i$. Let $S$ be the Young subgroup $S_{n_1}\times\cdots\times S_{n_\ell}$ of  $S_n$. For any $1\leq k\leq n$, set $\delta_k=1$ if $\mu_k=\omega_i$; $\delta_k=\sqrt{-1}$ if $\mu_k=\omega_{2\ell+1-i}$. We define the following polynomial 
	\[f_{i,r}:=\sum_{k\in K_i,\ {\scriptscriptstyle\sum} j_k=r} (\delta_kz_k)^{2j_k}\in \C\big[(\delta_1z_1)^2,\ldots, (\delta_nz_n)^2\big]^S. \]
	From the computation \eqref{432983}, the map \eqref{298232} sends $T_{i,r}$ to $f_{i,r}$. Note that $\{f_{i,r}\mid 1\leq i\leq \ell, 1\leq r\leq n_i\}$ is exactly a set of  algebraically independent generators of the algebra $\C\big[(\delta_1z_1)^2,\ldots, (\delta_nz_n)^2\big]^S$. We conclude that \eqref{298232} is an embedding. It follows that the surjective map \eqref{324821} is an isomorphism.	
	
	Now, both sides of \eqref{234571} are graded free $A^\sigma(\bar\mu)\simeq \As(-\vec\mu)$-modules, see Theorem \ref{350320} and Lemma \ref{342893}. Moreover, their fibers at $0$ are the twisted local Weyl module $W^\sigma(\bar\mu)$. By comparing their graded characters (as graded $\mf{h}^\sigma$-modules), we conclude that \eqref{234571} is an isomorphism.	
\end{proof}
\begin{thm}\label{cor_dem_weyl}
	Given fundamental coweights $\lambda_1,\ldots, \lambda_n \in X_*(T)^+$ with $\bar{\lambda}$ being the image of $\sum \lambda_i$ in $X_*(T)^+_\sigma$, there is an isomorphism of $\gs$-modules
	\[\Ds(1,\vec{\lambda})\simeq \mb{W}^\sigma(\iota(\lb)),\] 
	and an isomorphism of $\C$-algebras $\As(1,\la)\simeq A^\sigma(\iota(\bar\lambda))$.
\end{thm}
\begin{proof}
	The theorem directly follows from Theorem \ref{234924} and Corollary \ref{623012}.
\end{proof}
\begin{remark}\label{rem_hw_alg}
From part (1) of Theorem \ref{350320} and Theorem \ref{cor_dem_weyl}, we are able to determine the algebra $\As(c,\vec{\lambda})$ when $c=1$ and $\vec{\lambda}$ is a collection of fundamental coweights. 
\end{remark}

\section{Beilinson-Drinfeld Grassmannians of parahoric Bruhat-Tits group schemes}\label{825344}
In this section, we define Beilinson-Drinfeld Schubert varieties of certain parahoric Bruhat-Tits group schemes, and we prove their flatness over the base.  
\subsection{Beilinson-Drinfeld Grassmannian} 
\label{BD_Gr}
Let $G$ be a reductive group over $\mathbb{C}$. Let $\Gamma$ be a finite group acting on $G$ which preserves a maximal torus $T$. Let $C$ be an irreducible smooth algebraic curve over $\C$ with a faithful $\Gamma$-action. Denote by $\bar C$ the quotient curve $C/\Gamma$. Let $\pi: C\to \bar{C}$ be the quotient map, and denote by $\bar{p}:=\pi(p)$ the image of a point $p\in C$. For each point $p\in C$, denote by $\Gamma_p$ the stabilizer of $p$ in $\Gamma$. We say a point $p\in C$ is unramified (resp.\,ramified) if $\Gamma_p$ is trivial (resp.\,nontrivial).

\begin{lem}\label{031385}
	Let $C_1$, $C_2$ be irreducible smooth curves over $\C$ with faithful $\Gamma$-actions. Let $f: C_1\to C_2$ be a $\Gamma$-equivariant \'etale morphism such that $\Gamma_p=\Gamma_{f(p)}$ for any $p\in C_1$. Then, the induced map $\bar{f}: \bar C_1\to \bar C_2$ is also \'etale, and there is an isomorphism 
	\begin{equation}\label{320259}
		C_1\simeq C_2\times_{\bar C_2} \bar C_1.
	\end{equation}
\end{lem}
\begin{proof}
	We can assume $C_1$, $C_2$ are affine curves. We denote by $\pi_i: C_i\to \bar{C}_i$ the projection maps for $i=1,2$. 
	Let $K_1$, $K_2$ be their function fields, and $\ms{O}_1$, $\ms{O}_2$ be the coordinate rings. Then we have a diagram of field extensions
	\[\xymatrix{
		K_1 & \ar[l] K_2 \\
		K_1^\sigma \ar[u] & \ar[l] K_2^\sigma \ar[u]
	}.\]
	Observing that $K_2\cap K_1^\sigma= K_2^\sigma$, we have 
	\[ [K_2K_1^\sigma: K_2^\sigma]=[K_2K_1^\sigma: K_2][K_2: K_2^\sigma]=[K_1^\sigma:K_2^\sigma][K_2: K_2^\sigma]=[K_1:K_2^\sigma]. \] Thus, $K_1=K_2K_1^\sigma$. Since $K_1/K_1^\sigma$ is finite, we have $K_1=K_1^\sigma(\theta)$ for some $\theta\in K_1$, and $f(\theta)=0$ for some polynomial $f$ over $K_1^\sigma$. By multiplying elements in $\ms{O}^\sigma_1$, all coefficients of $f$ can be made to be in $\ms{O}^\sigma_1$. Moreover, by modifying $\theta$, we can further assume $f$ is monic and $f(\theta)=0$. Thus, $\theta$ is integral over $\ms{O}_1^\sigma$. It follows that $\theta$ is integral over $\ms{O}_1$. Since $\ms{O}_1$ is integrally closed in $K_1$, we must have $\theta\in \ms{O}_1$. Hence, $K_1\simeq K^\sigma_1\otimes_{\ms{O}_1^\sigma}\ms{O}_1$, since $\ms{O}_1$ is a free $\ms{O}_1^\sigma$-module of rank $[K_1: K_1^\sigma]$. Similarly, we have $K_2=K^\sigma_2\otimes_{\ms{O}_2^\sigma}\ms{O}_2$. Therefore, we have
	\[ \ms{O}_2\otimes_{\ms{O}_2^\sigma} K_1^\sigma \simeq \ms{O}_2\otimes_{\ms{O}_2^\sigma} \ms{O}_1^\sigma \otimes_{\ms{O}_1^\sigma} K_1^\sigma\simeq K_1\simeq \ms{O}_1\otimes_{\ms{O}_1^\sigma} K_1^\sigma.\]

	By assumption, $\Gamma_p=\Gamma_{f(p)}$ for any $p\in C_1$. The map $f$ induces an isomorphism $\mb{D}_p\simeq \mb{D}_{f(p)}$ of formal discs around $p$ and $f(p)$, for any $p\in C_1$. Thus, $f$ induces an \'etale map $\bar f: \bar C_1\to \bar C_2$. It follows that $C_2\times_{\bar C_2}\bar C_1$ is smooth. For any point $(p, q)\in C_2\times_{\bar C_2}\bar C_1$, we have $\pi_2(p)=\bar f(q)$. We pick a point $x$ in $C_1$ such that $\pi_1(x)=q$. Then $\pi_2(f(x) )=\pi_2(p)$. Thus, $f(x)=\gamma ( p)$ for some $\gamma\in \Gamma$. Set $x'=\gamma^{-1}(x)$. Then, $f(x')=p$ and $\pi_1(x')=q$. This shows the surjectivity of the map $C_1\to C_2\times_{\bar C_2}\bar C_1$. Since $C_2\times_{\bar C_2} \bar C_1$ is reduced and $\ms{O}_2\otimes_{\ms{O}_2^\sigma} K_1^\sigma\simeq K_1$, the curve $C_2\times_{\bar C_2}\bar C_1$ is integral and has the same function field as $C_1$. Therefore, we have an isomorphism $C_1\simeq C_2\times_{\bar C_2}\bar C_1$.
\end{proof}
Let $\mc{G}_{\bar C}$ denote the $\Gamma$-fixed point subgroup scheme $\mr{Res}_{C/\bar{C}}\big( G\times C\big)^\Gamma $ of the Weil restriction $\mr{Res}_{C/\bar{C}}\big( G\times C\big)$. We will simply denote it by $\mc{G}$ if there is no confusion. Now, we prove an \'etale base change property of $\mc{G}$ which will be used in Section \ref{029728}.
\begin{prop}\label{218368}
	Under the same setup as in Lemma \ref{031385} and let $\Gamma$ act on $G$. We have an isomorphism of group schemes over $\bar C_1$ 
	\[\mc{G}_{\bar C_1}\simeq \mc{G}_{\bar C_2}\times_{\bar C_2} \bar C_1.\]
\end{prop}
\begin{proof}
	Given any scheme $\varphi:S\to \bar C_1$ over $\bar C_1$, by definition $\mc{G}_{\bar C_1}(S)=\mr{Map}^\Gamma(C_1\times_{\bar C_1}S,G)$ consists of $\Gamma$-equivariant morphisms from $C_1\times_{\bar C_1} S$ to $G$. 
	On the other hand, $(\mc{G}_{\bar C_2}\times_{\bar C_2} \bar C_1)(S)$ consists of $\Gamma$-equivariant morphisms from $C_2\times_{\bar C_2}S$ to $G$, where $S$ is regarded as a scheme over $\bar C_2$ via $f\circ \varphi:S\to \bar C_2$. By Lemma \ref{031385}, we have an isomorphism 
	\[C_1\times_{\bar C_1} S \simeq (C_2\times_{\bar C_2}\bar C_1)\times_{\bar C_1} S\simeq C_2\times_{\bar C_2}S.\]
	Thus, $\mc{G}_{\bar C_1}(S)\simeq (\mc{G}_{\bar C_2}\times_{\bar C_2} \bar C_1)(S)$ for any scheme $S$ over $\bar C_1$.
\end{proof}
For any $\C$-algebra $R$, and any morphism $p:\mr{Spec}(R)\to C$, we denote by $\Delta_p\subseteq C\times \mr{Spec}(R)$ the graph of $p$. Let $C_R:=C\times \mr{Spec}(R)$. We denote by $\hat{\Delta}_p$ the formal completion of $C_R$ along $\Delta_p$, and by $\hat{\Delta}_p^*$ the punctured formal completion along $\Delta_p$, cf.\,\cite[Chapter II Section 9]{Hartshorne77}. By \cite[\href{https://stacks.math.columbia.edu/tag/0GBA}{Tag 0GBA},\href{https://stacks.math.columbia.edu/tag/0AMC}{Tag 0AMC}]{stacks-project}, we have an equivalent definition of the formal completion $\hat{D}$ of $C_R$ along a closed subscheme $D$ as the following functor, for any $R$-scheme $u:T\to \mr{Spec}(R)$, \begin{equation}\label{326925}
\hat{D}(T)= \{\phi:T\to C\,|\,(\phi, u) \circ \alpha\in D(T_{\mr{red}})\},
\end{equation}
where $T_{\mr{red}}$ is the reduced scheme of $T$ with the natural embedding $\alpha: T_{\mr{red}}\to T$. 
%We remark that this equivalent definition will only be used in the proof of Proposition \ref{329321}.

\medskip
We recall the definition of Beilinson-Drinfeld Grassmannians of general group schemes.
\begin{definition}\label{def_BD_jet}
Let $X$ be a smooth algebraic curve over $\C$. Let $\G$ be a smooth affine group scheme over $X$. The Beilinson-Drinfeld Grassmannian $\Gr_{\mc{G},C^n}$ and the jet group scheme $L^+\mc{G}_{X^n}$ over $X^n$ are defined as follows: for any $\C$-algebra $R$, define
\[\Gr_{\mc{G},X^n}(R):= \left\{ (p_1, \ldots, p_n, \mc{F}, \beta ) \,\middle|\, p_i \in X(R), \mc{F} \text{ a } \mc{G}\text{-torsor on } X_R, \beta: \mc{F}|_{X_{R} \setminus \cup\Delta_{p_i}}\simeq \mathring{\mc{F}}|_{X_{R} \setminus \cup\Delta_{p_i}} \right\}\]
and 
\[L^+\mc{G}_{X^n}(R):=\left\{ (p_1, \ldots, p_n,\eta ) \,\middle|\, p_i \in X(R), \eta: \mathring{\mc{F}}|_{ \widehat{ \cup{\Delta}_{p_i} } }\simeq \mathring{\mc{F}}|_{ \widehat{ \cup{\Delta}_{p_i} } }\right\},\]
where $\mathring{\mc{F}}$ is the trivial $\mc{G}$-torsor on $X_R$, and $ \widehat{ \cup{\Delta}_{p_i} } $ is the completion of $X_R$ along $ { \cup{\Delta}_{p_i} } $. 
\end{definition}

%\begin{remark}
%One can define the BD Grassmannian $\Gr_{\G, \bar{C}^n }$ over $\bar{C}$, see \cite[Section 3.1.1]{Zhu17}. The BD Grassmannian $\Gr_{\G, C^n}$ defined above is the base change $\Gr_{\G,\bar{C}^n } \times_{\bar{C}^n } C^n$.
%\end{remark}
There is a left $L^+\mc{G}_{X^n}$-action on $\Gr_{\mc{G},X^n}$ given by the following map 
\begin{align}\label{522693}
	L^+\mc{G}_{X^n}\times \mr{Gr}_{\mc{G}, X^n} & \to \mr{Gr}_{\mc{G}, X^n} \\
	\big( (\vec{p},\eta), (\vec{p}, \mc{F}, \beta) \big) &\mapsto \big(\vec{p}, \mc{F}', \beta' \big)\nonumber,
\end{align}
where by Beauville-Laszlo's Lemma, there exists a unique $\mc{G}$-torsor $\mc{F}'$ (up to a unique isomorphism) with isomorphisms $\delta:{\mc{F}}'|_{ \widehat{\cup\Delta_{p_i}} }\simeq \mc{F}|_{ \widehat{\cup\Delta_{p_i}} }$ and $\beta': \mc{F}'|_{C_R\setminus \cup\Delta_{p_i} } \simeq \mathring{\mc{F}}|_{C_R\setminus \cup \Delta_{ p_i}}$ such that $\beta'\circ\delta^{-1}|_{{\widehat{\cup\Delta_{p_i}} }^*}=\eta\circ \beta|_{{\widehat{\cup\Delta_{p_i}} }^*}$. Here ${\widehat{\cup\Delta_{p_i}} }^*:=\big({{\widehat{\cup\Delta_{p_i}} }}\big)\cap \big(X_R\setminus \cup\Delta_{ p_i}\big)$.

\begin{prop}\label{329321}
	Let $f:X\to Y$ be an \'etale morphism of  irreducible smooth curves. Let $\G_Y$ be a smooth affine group scheme over $Y$, and let $\G_X= \G_Y\times_Y X$. Let $Z=\{(x_1,\ldots,x_n)\in X^n \,|\, x_i\neq x_j, f(x_i)= f(x_j) \text{ for some } i\neq j\}$. We have an isomorphism
	$$ \Gr_{\G_{X}, X^n\setminus Z}\simeq \Gr_{\G_{ Y}, Y^n}\times_{Y^n}(X^n\setminus Z).$$
\end{prop}
\begin{proof}	
 Let $S$ be a scheme over $\mathbb{C}$. When $n=2$, an $S$-point $(x_1,x_2)\in (X^2\setminus Z)(S)$ is a pair of points $x_1,x_2\in X(S)$ such that the equalizer $\mr{Eq}( x_1, x_2)$ of $ x_1, x_2:S\to  X$ is naturally isomorphic to the equalizer $\mr{Eq}( y_1, y_2)$ of $ y_1, y_2:S\to  Y$, where $y_i=f\circ x_i\in Y(S)$. We shall show ${f}_S= f\times \mr{Id}_S: X_S=X\times S\to  Y\times S=Y_S$ induces an isomorphism $\widehat{\Delta_{ x_1}\cup\Delta_{ x_2}}\simeq \widehat{\Delta_{ y_1}\cup\Delta_{ y_2}}$. Firstly, $f_S$ is \'etale and  restricts to isomorphisms $\Gamma_{ x_1}\simeq \Delta_{ y_1}$ and $\Delta_{ x_2}\simeq\Delta_{ y_2}$. Since $\mr{Eq}( x_1, x_2)\simeq \mr{Eq}( y_1, y_2)$, we have an isomorphism
	$$\Delta_{ x_1}\times_{ X_S}\Delta_{ x_2}\simeq \mr{Eq}\big(({x}_1\times \mr{Id}_S), ({x}_2\times \mr{Id}_S)\big)\simeq \mr{Eq}\big(({y}_1\times \mr{Id}_S), ({y}_2\times \mr{Id}_S)\big)\simeq \Delta_{ y_1}\times_{ Y_S}\Delta_{ y_2}.$$ 
	Thus, ${f}_S$ induces an isomorphism 
\begin{equation}\label{eq_Delta}
{\Delta_{ x_1}\cup\Delta_{ x_2}}\simeq {\Delta_{ y_1}\cup\Delta_{ y_2}}.\end{equation} 

For any $S$-scheme $u: T\to S$, by \eqref{326925}, we have 
\begin{equation}\label{eq_comp_functor}
\widehat{\Delta_{ x_1}\cup\Delta_{ x_2}}(T)= \{  \phi:T\to  X \mid (\phi, u)\circ\alpha\in (\Delta_{ x_1}\cup\Delta_{ x_2})(T_{\mr{red}}) \},\end{equation}  where $T_{\mr{red}}$ is the reduced scheme of $T$ with the natural embedding $\alpha: T_{\mr{red}}\to T$. Then,  
we have a map $\iota: \widehat{\Delta_{ x_1}\cup\Delta_{ x_2}}(T)\to \widehat{\Delta_{ y_1}\cup\Delta_{ y_2}}(T)$ given by $\phi\mapsto f\circ \phi$.  By Lemma \ref{lem_lift},  $\iota$ is bijective.  Thus, we have the following isomorphisms 
	\begin{equation} \label{eq_iso_graphs}
	\widehat{\Delta_{ x_1}\cup\Delta_{ x_2}}\simeq \widehat{\Delta_{ y_1}\cup\Delta_{ y_2}},  \quad \widehat{\Delta_{ x_1}\cup\Delta_{ x_2}}^*\simeq \widehat{\Delta_{ y_1}\cup\Delta_{ y_2}}^*\end{equation}
 where $^*$ represents the punctured formal completion. Let $\mc{F}$ be a $\G_{ X}$-torsor over $ X_S$ with a trivialization $\beta:\mc{F}|_{ X_S\setminus(\Delta_{ x_1}\cup\Delta_{ x_2})}\simeq \mc{F}^0|_{ X_S\setminus(\Delta_{ x_1}\cup\Delta_{ x_2})}$. By Beaville-Laszlo lemma, this is equivalent to the data consisting of a torsor $\mc{F}'$ over $\widehat{\Delta_{ x_1}\cup\Delta_{ x_2}}$ and a trivialization $\beta'$ over $\widehat{\Delta_{ x_1}\cup\Delta_{ x_2}}^*$. By isomorphisms in (\ref{eq_iso_graphs}), we get a torsor $\mc{P}$ over $\widehat{\Delta_{ y_1}\cup\Delta_{ y_2}}$ and a trivialization $\eta$ over $\widehat{\Delta_{ y_1}\cup\Delta_{ y_2}}^*$. Again, this is equivalent to the data of a $\G_{ Y}$-torsor over $ Y_S$ with a trivialization over $ Y_S\setminus (\Delta_{ y_1}\cup\Delta_{ y_2})$,  which will still be denoted  by $\mc{P}$ and $\eta$. Thus, we get an isomorphism 
	$$\Gr_{\G_{ X}, X^2\setminus Z}\to  \Gr_{\G_{ Y}, Y ^2}\times_{X^2}(X^2\setminus Z),$$
	given by $(x_1,x_2,\mc{F},\beta)\mapsto \big((y_1,y_2, \mc{P},\eta), (x_1,x_2)\big)$.
	
	When $n\geq 3$, by induction on $n$, we assume that $ f_S$ induces an isomorphism $$\Delta_{ x_1}\cup\cdots\cup\Delta_{ x_{n-1}}\simeq \Delta_{ y_1}\cup\cdots\cup\Delta_{ y_{n-1}}.$$
	Then, $(\bigcup_{i=1}^{n-1}\Delta_{ x_{i}})\times_{ X_S}\Delta_{ x_n}\simeq \bigcup_{i=1}^{n-1}(\Delta_{ x_i}\times_{ X_S} \Delta_{ x_n})\simeq \bigcup_{i=1}^{n-1}(\Delta_{ y_i}\times_{ Y_S} \Delta_{ y_n})\simeq (\bigcup_{i=1}^{n-1}\Delta_{ y_{i}})\times_{ Y_S}\Delta_{ y_n}$. Thus, $ f_S$ induces an isomorphism $\Delta_{ x_1}\cup\cdots\cup\Delta_{ x_{n}}\simeq \Delta_{ y_1}\cup\cdots\cup\Delta_{ y_{n}}$. By the same argument as in the case when $n=2$, we have isomorphisms $\widehat{\bigcup_i \Delta_{ x_i}}\simeq \widehat{\bigcup_i \Delta_{ y_i}}$, $\widehat{\bigcup_i \Delta_{ x_i}}^*\simeq \widehat{\bigcup_i \Delta_{ y_i}}^*$, and 
	\begin{equation*} 
		\Gr_{\G_{ X}, X^n\setminus Z}\simeq \Gr_{\G_{ Y}, Y^n}\times_{Y^n}(X^n\setminus Z).\qedhere
	\end{equation*}
\end{proof}

\begin{remark}
When $n=1$, $Z$ is the empty set.  In this case, Proposition \ref{329321} is exactly \cite[Lemma 3.2]{Zhu14}. When $n\geq 2$, a statement was formulated incorrectly in \cite[Prop.3.1.20]{Zhu17}.  \end{remark}
\begin{lem}\label{lem_lift}
Suppose we are in the same setup and assumptions as in Proposition \ref{329321} and its proof.
For any $\psi\in  Y(T)$ and $u:T\to S$ such that $(\psi\circ \alpha, u \circ\alpha) \in  (\Delta_{ y_1}\cup\Delta_{ y_2})(T_{\rm red})$, 
we have a unique lifting $\phi\in X(T)$ such that  the following diagram commutes
\begin{equation}
 \xymatrixcolsep{2pc}
 \xymatrix{
T_{\mr{red}} \ar[r]^-{} \ar[d]_{\alpha} &  X \ar[d]^-{f} \\
T \ar[r]^-{\psi}\ar[ru]^{\exists \, \phi}&  Y}.
\end{equation}
 and $(\phi\circ \alpha, u \circ \alpha) \in 
(\Delta_{ x_1}\cup\Delta_{ x_2})(T_{\rm red})$.
\end{lem}

\begin{proof}
We may assume that $S=\mr{Spec}(A)$ and $T=\mr{Spec}(B)$ for some commutative rings $A, B$ over $\mathbb{C}$.   By assumption $(\psi\circ \alpha, u\circ \alpha )\in (\Delta_{ y_1}\cup\Delta_{ y_2})(T_{\rm red})$. Then the isomorphism (\ref{eq_Delta}) induces a morphism $(\phi_{\rm red}, u\circ \alpha):  T_{\rm red}\to X\times S$ such that $f\circ \phi_{\rm red}=\psi\circ \alpha$ and $(\phi_{\rm red}, u\circ \alpha)\in(\Delta_{ x_1}\cup\Delta_{ x_2})(T_{\rm red}) $.  By \cite[\href{https://stacks.math.columbia.edu/tag/02HF}{Tag 02HF}]{stacks-project}, $f$ is \'etale implies that $f$ is formally \'etale. When $B$ is Noetherian, its Jacobson ideal $J(B)$ is nilpotent, i.e. there exists an integer $k$ such that $J(B)^k=0$. Then, by the definition of formally \'etale morphism, there exists a unique morphism $\phi\in X(T)$ such that $\phi\circ \alpha= \phi_{\rm red}$ and 
 $f\circ \phi= \psi$.

We now consider the general case.  Let $h: A\to B$ be the  ring homomorphism induced from $u: T\to S$. Since $X$ is of finite type, there exists a finitely generated $\mathbb{C}$-subalgebra $A'$ of $A$ and $x'_1,x'_2\in X(A')$ such that $p(x'_i)=x_i$, i=1,2, where $p:S\to S':=\mr{Spec}(A')$ is induced from $A'\to A$, see the left diagram in \eqref{5298385}. 
\begin{equation}\label{5298385}
\xymatrix{
S \ar[r]^-{x_i} \ar[d]_{p} &  X \\
\exists\, S' \ar[ru]_{ x_i'}&}\hspace{6em} 
\xymatrixcolsep{3pc}
\xymatrix{
T \ar[d]_{q} \ar[r]^-{(\psi,u)}   &  Y\times S \ar[d]^{\mr{Id}_Y\times p}\\
\exists\, T' \ar[r]^-{(\psi',u')} & Y\times S'}
\end{equation}

For any $\psi\in Y(T)=Y(B)$ such that $(\psi, u) \circ\alpha \in  (\Delta_{ y_1}\cup\Delta_{ y_2})(T_{\rm red})$, there exists a finitely generated $\mathbb{C}$-subalgebra $B'$ of $B$ and $\psi' \in Y(B')$ such that $h(A')\subset B'$, 
and $q(\psi')=\psi$, where $q: T\to T'$ is induced from $B'\subset B$.  Let $u': T'\to S'$ be the morphism associated to $h: A'\to B'$. Then, we have the commutative diagram on the right side of \eqref{5298385}. Since $S'$ and $T'$ are Noetherian, there exists a unique $\phi'\in X(S')$ such that $f\circ \phi'=\psi'$ and $(\phi', u') \circ \alpha' \in 
(\Delta_{ x'_1}\cup\Delta_{ x'_2})(T'_{\rm red})$, where $\alpha': T'_{\rm red}\to T'$ is associated to the quotient morphism $B'\to B'/J(B')$ with $J(B')$ being the Jacobson ideal of $B'$. Thus, we have the following comutative diagram.
\begin{equation}
\xymatrixcolsep{2pc}
\xymatrix{
T_{\mr{red}}\ar[r]^-{q_{\rm red}}\ar[d]_\alpha & T'_{\mr{red}} \ar[r]^-{} \ar[d]_{\alpha'} &  X \ar[d]^-{f} \\
T\ar[r]^q & T' \ar[r]^-{\psi'}\ar[ru]^-{\phi'}&  Y}.
\end{equation}
where $q_{\rm red}: T'_{\rm red}\to T_{\rm red}$ is the induced from $q: T'\to T$.
Set $\phi=\phi' \circ q$. We shall show that $\phi$ is the desired lifting. 

It is clear that $f\circ \phi=\psi$.  Note that the following diagram commutes,
\begin{equation}
\label{eq_complicated}
\xymatrixcolsep{6pc}
\xymatrix{
T_{\mr{red}}\ar[r]^-{(\phi \circ \alpha,  u\circ \alpha)}\ar[d]_{(q_{\rm red}, u\circ \alpha)} &  X\times S\ar@{=}[d] \\
U:=T'_{\mr{red}}\times_{S'} S \ar[r]^-{(\phi'\circ \alpha',  u'\circ \alpha'  )\times  {\mr{Id}_S}}&  (X\times S')\times_{S'} S}.
\end{equation}
Since $(\phi'\circ \alpha',  u'\circ \alpha'  )\in 
(\Delta_{ x'_1}\cup\Delta_{ x'_2})(T'_{\rm red})$ and $(\Delta_{ x'_1}\cup\Delta_{ x'_2})\times_{S'} S  \simeq \Delta_{ x_1}\cup\Delta_{ x_2}$ , we have $ (\phi'\circ \alpha',  u'\circ \alpha'  )\times  {\mr{Id}_S}  \in ((\Delta_{ x'_1}\cup\Delta_{ x'_2})\times_{S'} S  )(U) \simeq (\Delta_{ x_1}\cup\Delta_{ x_2})(U)$. Composing with $(q_{\rm red}, u\circ \alpha)$, by the commutative diagram (\ref{eq_complicated}) we get $(\phi \circ \alpha,  u\circ \alpha)\in (\Delta_{ x_1}\cup\Delta_{ x_2})(T_{\rm red})$. This concludes the proof of the lemma.
\end{proof}
From now on, we will focus on the parahoric Bruhat-Tits group scheme $\G_{\bar C}=\mr{Res}_{C/\bar{C}}\big( G\times C\big)^\Gamma $ over $\bar C$. 
\begin{definition}
	We denote by $\Gr_{\G_{\bar C}, C^n}:=\Gr_{\G_{\bar C}, \bar C^n}\times_{\bar C^n} C^n$ the base change of BD grassmannian of $\G_{\bar C}$ to $C^n$. More precisely, for every $\C$-algebra $R$, we define
	\[\Gr_{\mc{G}_{\bar C},C^n}(R):= \left\{ (p_1, \ldots, p_n, \mc{F}, \beta ) \,\middle|\, p_i \in C(R), \mc{F} \text{ a } \mc{G}\text{-torsor on } \bar{C}_R, \beta: \mc{F}|_{\bar{C}_{R} \setminus \cup\Delta_{\bar{p}_i}}\simeq \mathring{\mc{F}}|_{\bar{C}_{R} \setminus \cup\Delta_{\bar{p}_i}} \right\}\]
	If there is no confusion, we will simply use the notation $\Gr_{\G,C^n}$. 
\end{definition}
	The BD Grassmannian $\Gr_{\mc{G},C^n}$ is an ind-scheme, ind-of-finite type and ind-projective over $C^n$, see \cite[Remark 3.1.4]{Zhu17}. The following result is a consequence of Proposition \ref{329321}.
\begin{cor}\label{prop_etale_base}
	\begin{enumerate}
		\item Let $X,Y$ be  irreducible smooth curves with faithful $\Gamma$-actions. Suppose $f:X\to Y$ is an $\Gamma$-equivariant \'etale morphism such that $\Gamma_x=\Gamma_{f(x)}$ for any $x\in X$. Let $Z=\{(x_1,\ldots,x_n)\in X^n \,|\, \bar x_i\neq\bar x_j, \overline{f(x_i)}=\overline{ f(x_j)} \text{ for some } i\neq j\} $. We have an isomorphism
		$$ \Gr_{\G_{\bar X}, X^n\setminus Z}\simeq \Gr_{\G_{\bar Y}, Y^n}\times_{Y^n}(X^n\setminus Z).$$
		\item Let $C$ be an irreducible smooth curve with a faithful $\Gamma$-action such that all  points are unramified.  Let $Z=\{(x_1,\ldots,x_n)\in C^n \,|\, x_i\neq  x_j, \bar x_i=\bar x_j \text{ for some } i\neq j\}$. Then, 
		\[ \Gr_{G_C,C^n\setminus Z}\simeq \Gr_{\G_{\bar C}, \bar C^n}\times _{\bar C^n} (C^n\setminus Z)=\Gr_{\G_{\bar C},C^n\setminus Z},\]
		where $G_C$ is the constant group scheme over $C$.
	\end{enumerate}
\end{cor}
\begin{proof}
	For part (1), by Lemma \ref{031385} and Proposition \ref{218368}, $f$ induces an \'etale morphism $\bar{f}:\bar X\to \bar Y$ and there is an isomorphism $\G_{\bar X}\simeq \G_{\bar Y}\times_{\bar Y} \bar X$. Let $\bar Z=\{(p_1,\ldots,p_n)\in \bar X^n \,|\,  p_i\neq p_j, \bar{f}(p_i)=\bar{ f}(p_j) \text{ for some } i\neq j\} $. Then $(x_1,\ldots,x_n)\in Z$ if and only if $(\bar x_1,\ldots,\bar x_n)\in \bar Z$.  By Proposition \ref{329321}, we have an isomorphism
	\begin{equation}\label{424910}
		\Gr_{\G_{\bar X},\bar X^n\setminus \bar Z}\simeq \Gr_{\G_{\bar Y},\bar Y^n}\times_{\bar Y^n} (\bar X^n\setminus \bar Z).
	\end{equation}
	Note that there are isomorphisms
	$$\Gr_{\G_{\bar X}, X^n\setminus Z}=\Gr_{\G_{\bar X}, \bar X^n}\times_{\bar X^n} (X^n\setminus Z)\simeq \Gr_{\G_{\bar X}, \bar X^n}\times_{\bar X^n}(\bar X^n\setminus \bar Z)\times_{\bar X^n\setminus \bar Z} (X^n\setminus Z),$$
	$$\Gr_{\G_{\bar Y}, Y^n}\times_{Y^n}(X^n\setminus Z)=\Gr_{\G_{\bar Y},\bar Y^n}\times_{\bar Y^n}  (X^n\setminus Z)\simeq\Gr_{\G_{\bar Y},\bar Y^n}\times_{\bar Y^n} (\bar X^n\setminus \bar Z)\times_{\bar X^n\setminus \bar Z} (X^n\setminus Z).$$ 
	By base change via $X^n\setminus Z\to \bar X^n\setminus \bar Z$, the isomophism \eqref{424910} becomes
	$$ \Gr_{\G_{\bar X}, X^n\setminus Z}\simeq \Gr_{\G_{\bar Y}, Y^n}\times_{Y^n}(X^n\setminus Z).$$
	
	For part (2), we consider the \'etale morphism $\pi:C\to \bar C$. There is a natural isomorphism $G_C\simeq \G_{\bar C}\times_{\bar C} C$. By Proposition \ref{329321}, we get an isomorphism
	\begin{equation*}
		\Gr_{G_C,C^n\setminus Z}\simeq \Gr_{\G_{\bar C}, \bar C^n}\times _{\bar C^n} (C^n\setminus Z). \qedhere
	\end{equation*}
\end{proof}

Let $\xi=(I_1,\ldots, I_k)$ be a partition of $[n]$. Similar to Definition \ref{524794}, we define $C_\xi^n$ to be the following open subset in $C^n$,
\[C_\xi^n:=\{(p_1,\ldots,p_n)\in C^n\mid \bar{p}_i\neq \bar{p}_j, \ \forall i\in I_\alpha, j\in J_\beta, \alpha\neq \beta\}.\]
By \cite[Theorem 3.2.1]{Zhu17}, $\Gr_{\mc{G},C^{n}}$ and $L^+\mc{G}_{C^n}$ have a factorization property over $C_\xi^n$ as follows
\begin{align}
	\label{447893} \Gr_{\mc{G},C^{n}_\xi}&:=\Gr_{\mc{G},C^{n}}\times_{C^n} C^{n}_\xi \simeq\bigg(\prod_{j=1}^k\Gr_{\mc{G},C^{I_j}}\bigg)\times_{C^n} C^{n}_\xi, \\
	\label{346288} L^+\mc{G}_{C^n_\xi}&:=L^+\mc{G}_{C^n}\times_{C^n} C^{n}_\xi \simeq\bigg(\prod_{j=1}^kL^+\mc{G}_{C^{I_j}}\bigg)\times_{C^n} C^{n}_\xi.
\end{align}

\subsection{Beilinson-Drinfeld Schubert varieties}\label{993431} 
With the same setup as in Section \ref{BD_Gr}, we define a $\Gamma$-action on $C^{n+1}=C^n\times C$ by letting $\Gamma$ act on the last copy of $C$. More precisely, this action is given by 
\begin{equation}\label{038287}
	\gamma(p_1,\ldots, p_n, p_{n+1}):=(p_1,\ldots, p_n, \gamma p_{n+1}),
\end{equation}
for any $\gamma\in \Gamma$, $(p_1,\ldots, p_n, p_{n+1})\in C^{n+1}$.

Let $T$ be a $\Gamma$-stable maximal torus in $G$. Given any $\la=(\lambda_1,\lambda_2,\ldots, \lambda_n)$ with $\lambda_i\in X_*(T)$, we shall construct a $\Gamma$-equivariant tensor functor $F$ from the category $\mr{Rep}(T)$ of finite dimensional representations of $T$ to the category $\mr{Coh}^{\mr{lf}}(C^{n+1})$ of locally free sheaves of finite rank over $C^{n+1}$. For each character $\nu$ of $T$, we define
\begin{equation}\label{214694}
	F( \C_{\nu}):= \ms{O}_{C^{n+1}}\bigg(\sum_{i=1}^n\sum_{\gamma\in \Gamma} \langle \gamma(\lambda_i), \nu \rangle \Delta_{\gamma, i}\bigg),
\end{equation} 
where $\Delta_{\gamma,i}:=\{(a_1, a_2, \ldots, a_n, a_{n+1})\mid a_{n+1}=\gamma a_{i}\}$ 
is a divisor in $C^{n+1}$. 

\begin{lem}\label{lem_torosr}
	This assignment \eqref{214694} extends to a tensor functor $F:\mr{Rep}(G)\to \mr{Coh}^{\mr{lf}}(C^{n+1})$ which is $\Gamma$-equivariant in the sense of Definition \ref{964424}. As a consequence, by Theorem \ref{926976}, this functor $F$ defines a $(\Gamma, T)$-torsor $\mc{F}$ over $C^{n+1}$. 
\end{lem}
\begin{proof}
For each $\eta\in \Gamma$, denote by $\eta_T: T\to T$ the map sending $h$ to $\eta (h)$. Let $\eta_T^*:\mr{Rep}(T)\to \mr{Rep}(T)$ be the pullback functor induced from $\eta_T$. Define the $\Gamma$-action on $X^*(T)$ by 
\[\eta(\nu)(t):=\nu(\eta^{-1}(t)),\] 
where $\eta\in \Gamma$, $\nu\in X^*(T)$, and $t\in T$. Then, we have $\eta_T^*(\C_\nu)=\C_{\eta^{-1}(\nu)}$, and 
\begin{equation}\label{826785}
	F\eta_T^*( \C_{\nu})= \ms{O}_{C^{n+1}}\bigg(\sum_{i=1}^n\sum_{\gamma\in \Gamma} \langle \gamma(\lambda_i), \eta^{-1}(\nu) \rangle \Delta_{\gamma, i}\bigg)= \ms{O}_{C^{n+1}}\bigg(\sum_{i=1}^n\sum_{\gamma\in \Gamma} \langle \eta\gamma(\lambda_i), \nu \rangle \Delta_{\gamma, i}\bigg),
\end{equation}
for any $\eta\in \Gamma$ and $\nu\in X^*(T)$.

Let $Y:=C^{n+1}$ be the variety with the $\Gamma$-action given in \eqref{038287}. For each $\eta\in \Gamma$, denote by $\eta_Y: Y\to Y$ the map sending $y$ to $\eta(y)$. Let $\eta_Y^*:\mr{Coh}^{\mr{lf}}(Y)\to \mr{Coh}^{\mr{lf}}(Y)$ be the pullback functor induced from $\eta_Y$. Then, there is an isomorphism 
\begin{equation}\label{060044}
	\eta_Y^*F(\C_\nu) \simto \ms{O}_{C^{n+1}}\bigg(\sum_{i=1}^n\sum_{\gamma\in \Gamma} \langle \gamma(\lambda_i), \nu \rangle \Delta_{\eta^{-1}\gamma, i}\bigg) = \ms{O}_{C^{n+1}}\bigg(\sum_{i=1}^n\sum_{\gamma\in \Gamma} \langle \eta\gamma(\lambda_i), \nu \rangle \Delta_{\gamma, i}\bigg).
\end{equation}

Combining \eqref{826785} and \eqref{060044}, we have an isomorphism between tensor functors $\theta_\eta: F\eta_T^*\to \eta_Y^*F$ for each $\eta\in \Gamma$. One can check that $F$ together with the collection of isomorphisms $\{\theta_\eta\}_{\eta\in \Gamma}$ satisfy the axioms of Definition \ref{964424}. This extends to a $\Gamma$-equivariant functor from $\mr{Rep}(T)$ to $\mr{Coh}^{\mr{lf}}(C^{n+1}) $. 
\end{proof}
For any $\C$-scheme $X$ with a $\Gamma$-action, we denote by $\mathring{\mc{F}}_X:= X\times G$ the trivial $G$-torsor over $X$ with a $\Gamma$-action acting diagonally on $X\times G$, and call it the trivial $(\Gamma, G)$-torsor over $X$. For the general definition of $(\Gamma, G)$-torsors, see Definition \ref{653962}. 
\begin{definition}\label{117973}
	A $(\Gamma, G)$-torsor $\mc{P}$ over a scheme $X$ is called locally trivial if for any $x\in X$, there exists an open neighborhood $U_x$ of $x$ such that 
	\begin{enumerate}
		\item $U_x$ is stable under the stabilizer $\Gamma_x$ of $x$,
		\item $\mc{P}|_{U_x}$ is isomorphic to the trivial $(\Gamma_x, G)$-torsor $\mathring{\mc{F}}_{U_x}$ over $U_x$.
	\end{enumerate} 
\end{definition}
Note that not every $(\Gamma, G)$-torsor is locally trivial, and in fact this is determined by local types of $(\Gamma, G)$-torsors at ramified points, see \cite{DH}. 
\begin{lem}\label{370984}
	The $(\Gamma, T)$-torsor $\mc{F}$ constructed in Lemma \ref{lem_torosr} is locally trivial over $C^{n+1}$.
\end{lem}
\begin{proof}
	Let $x$ be any point in $C^{n+1}$. We will find a $\Gamma$-stable neighborhood $U_x$ of $x$ in $C^n$ such that $\mc{F}$ is trivializable over $U_x$ as $(\Gamma,T)$-torsors. 
	
	Let $\Gamma_x$ be the stabilizer of $x$ in $\Gamma$. Given $1\leq i\leq n$, suppose that $x\in \Delta_{\eta_i, i}$ for some $\eta_i\in \Gamma$. In this case, we observe that $x \in \Delta_{ \gamma, i}$ if and only if $\gamma\in \eta_i\Gamma_x$. Hence there exists an open neighborhood $U_i$ of $x$ with a regular function $f_i$ over $U_i$ such that $f_i$ defines $\Delta_{\eta_i, i}\cap U_i$. Moreover, we can shrink $U_i$ so that $U_i\cap \Delta_{\gamma, i}=\emptyset$ for any $\gamma\not \in \eta_i\Gamma_x$. For the given $i$, if $x\not\in \Delta_{\eta, i}$ for any $\eta\in \Gamma$, by convention we take $\eta_i=e$, $U_i=C^{n+1}\setminus \bigcup_{\gamma\in \Gamma}\Delta_{\gamma, i}$, and $f_i=1$. Combining the above constructions, for any $x\in C^{n+1}$, we have the following open neighborhood $U_x$ of $x$ in $C^{n+1}$, 
	\[U_x:=\bigcap_{i=1}^n\bigcap_{\gamma\in \Gamma_x} \gamma U_i,\]
	and a collection of functions $\{ f_i \}_{i=1,\cdots,n}$ defined over $U_x$. 
	Clearly $U_x$ is $\Gamma_x$-stable. Moreover, for any $1\leq i\leq n$ and $\gamma\in \eta_i\Gamma_x$, the regular function $\gamma\eta_i^{-1} f_i$ defines the divisor $\Delta_{\gamma, i}=\gamma\eta_i^{-1}\Delta_{\eta_i, i}$ over $U_x$.

	Now, for each character $\nu$ of $T$, denote by $D_\nu$ the divisor $\sum_{i=1}^n\sum_{\gamma\in \Gamma} \langle \gamma(\lambda_i), \nu \rangle \Delta_{\gamma, i}$. Set 
	\[f_\nu=\prod_{i=1}^n\prod_{\gamma\in\eta_i\Gamma_x} (\gamma\eta_i^{-1} f_i)^{\langle \gamma (\lambda_i), \nu \rangle}.\] 
	Then $f_\nu$ defines the divisor $D_\nu$ locally in $U_x$, since $U_x\cap \Delta_{\gamma,i}=\emptyset$ for any $\gamma\not\in \eta_i\Gamma_x$. For any characters $\nu, \omega\in X^*(T)$, we have $f_\nu\cdot f_\omega=f_{\nu+\omega}$. Consider the tensor functor $F_x^0:\mr{Rep}(T)\to \mr{Coh}^{\mr{lf}}(U_x)$ sending $\C_{\nu}$ to $\ms{O}_{U_x}\otimes_{\C} \C_{\nu}$. It is easy to see $F_x^0$ is $\Gamma_x$-equivariant and this gives a trivial $(\Gamma_x, T)$-torsor $\mathring{\mc{F}}_{U_x}$ over $U_x$. Moreover, there is an isomorphism
	\begin{equation*}
		\phi_{\nu}: F|_{U_x}(\C_{\nu})=\ms{O}_{U_x}\bigg(\sum_{i=1}^n\sum_{\gamma\in \Gamma} \langle \gamma(\lambda_i), \nu \rangle \Delta_{\gamma, i}\bigg) \simto \ms{O}_{U_x}\simeq F_x^0(\C_{\nu}),
	\end{equation*}
	where the first isomorphism is given	by multiplying by $f_\nu$. One can check that these isomorphisms $\{\phi_\nu\}_{\nu\in X^*(T)}$ respect the tensor structure since $\ms{O}(D_{\nu})\otimes\ms{O} (D_{\omega})\simeq \ms{O}(D_{\nu+\omega})$. Thus, there is an isomorphism $\phi:F|_{U_x}\to F_x^0$ between tensor functors. By Theorem \ref{926976}, $\phi$ induces an isomorphism $\tilde\phi:\mc{F}|_{U_x}\to \mathring{\mc{F}}_{U_x}$ between $(\Gamma_x, T)$-torsors over $U_x$. Therefore, $\mc{F}$ is locally trivial following Definition \ref{117973}.
\end{proof}
Set $U=C^{n+1}\setminus \bigcup_{\gamma,i}\Delta_{ \gamma, i}$. Let $F^0:\mr{Rep}(T)\to \mr{Coh}^{\mr{lf}}(U)$ be the tensor functor sending $\C_{\nu}$ to $\ms{O}_{U}\otimes_{\C} \C_{\nu}$. Then, $F^0$ is $\Gamma$-equivariant, and it gives a trivial $(\Gamma, T)$-torsor $\mathring{\mc{F}}_{U}$ over $U$. For each $\nu\in X_*(T)$, there is a natural isomorphism \begin{equation}\label{342573}
	\eta_\nu:F|_U(\C_\nu)\simto F^0(\C_\nu).
\end{equation}
Then, $\{\eta_\nu\}_{\nu\in T^*}$ defines an isomorphism $\eta: F|_U\simto F^0$. By Theorem \ref{926976}, this is equivalent to an isomorphism $\mc{F}|_U\simto \mathring{\mc{F}}_U$, which will still be denoted by $\eta$.

Let $\mc{F}_G:=\mc{F}\times^T G$ be the $(\Gamma,G)$-torsor over $C^{n+1}$ induced from $\mc{F}$. Then, $\eta$ induces a trivialization $\eta_G$ of $\mc{F}_G$ outside $\bigcup_{\gamma,i}\Delta_{ \gamma, i}$. By Lemma \ref{370984}, $\mc{F}_G$ is a locally trivial $(\Gamma,G)$-torsor. We define a morphism $\pi:C^{n+1}\to C^n\times \bar C$ by $(p_1,\ldots,p_n,p_{n+1})\mapsto (p_1,\ldots p_n, \bar p_{n+1})$. From the local triviality of $\mc{F}_G$, $\pi_*(\mc{F}_G)^\Gamma$ is a $\mc{G}$-torsor over $C^n\times\bar C$, denoted by $\bar{\mc{F}}$. Moreover, $\eta_G$ induces a trivialization $\bar\eta$ of $\bar{\mc{F}}$ outside $\pi\big(\bigcup_{\gamma,i}\Delta_{ \gamma, i}\big)=\bigcup_i\Delta_{\bar P_i}$,
where $P_i:C^n \to C$ is the projection map sending $\vec{p}$ to $p_i$.

We construct a section $s_{\la}:C^n\to \Gr_{\mc{G},C^n}$ as follows:
\begin{equation}\label{926470}
	s_{\vec{\lambda}}=\big(P_1,\ldots, P_n, \bar{\mc{F}}, \bar\eta \big),
\end{equation}
\begin{definition}\label{046143}
	The Beilinson-Drinfeld Schubert variety $\Grb_{\mc{G}, C^n}^{\la}$ of $\mc{G}$ over $C^n$ is defined to be the schematic image of $o_{\la}:L^+\mc{G}_{C^n} \to \Gr_{\mc{G}, {C^n}}$, where $o_{\la}$ is defined as the following composition,
 \[ L^+\mc{G}_{C^n}\simeq L^+\mc{G}_{C^n}\times_{C^n} C^n \xrightarrow{{\rm Id}\times s_{\la} } L^+\mc{G}_{C^n}\times_{C^n}   \Gr_{\mc{G}, {C^n}} \xrightarrow{\eqref{522693}} \Gr_{\mc{G}, {C^n}} . \]
 When $n=1$, $\Grb_{\mc{G}, C}^{\lambda}$ is also called a global Schubert variety of $\G$. 
\end{definition}
\begin{prop}\label{684242} Given any $\la\in (X_*(T)^+)^n$ and any partition $\xi=(I_1,\cdots, I_k)$ of $[n]$, the BD Schubert variety $\Grb_{\mc{G}, C^n}^{\la}$ has a factorization property over $C^n_\xi$, i.e.
		\begin{equation}\label{468789}
			\Grb_{\mc{G},C^n_{\xi}}^{\la}\simeq\bigg(\prod_{j=1 }^k\Grb_{\mc{G},C^{I_j}}^{\la_{I_j}}\bigg)\times_{C^n} C^n_\xi,
		\end{equation}
		where $\la_{I_j}$ is defined in \eqref{249296}. 
\end{prop}
\begin{proof}
	It suffices to prove the case when $\xi=(I,J)$. The map \eqref{447893} can be given explicitly as follows:
	\begin{align}\label{134631}
		\Phi: \Gr_{\mc{G},C^{n}_\xi}& \simeq \big(\Gr_{\mc{G},C^{I}}\times \Gr_{\mc{G},C^{J}} \big)|_{C^{n}_\xi}, \\
		\nonumber (\vec{p},\mc{F},\beta)&\mapsto \big((\vec{p}_I,\mc{F}_I,\beta_I),(\vec{p}_J,\mc{F}_J,\beta_J)\big)
	\end{align}
	where by Beauville-Laszlo's Lemma, there exists a $\mc{G}$-torsor $\mc{F}_I$ with isomorphisms $\delta_I:\mc{F}_I|_{ \widehat{\cup_{i\in I} \Delta_{\bar p_i} }} \simeq {\mc{F}}_{ \widehat{\cup_{i\in I} \Delta_{\bar p_i} }}$ and $\beta_I: \mc{F}_I|_{ \bar C_R\setminus{\cup_{i\in I} \Delta_{\bar p_i} }}\to \mathring{\mc{F}}|_{\bar C_R \setminus{\cup_{ i\in I} \Delta_{\bar p_i} }}$ such that $\beta_I\circ\delta_I^{-1}|_{\widehat{\cup_{i\in I} \Delta_{\bar p_i} }^*}=\beta|_{\widehat{\cup_{i\in I} \Delta_{\bar p_i} }^*}$; $\mc{F}_J,\beta_J$ are defined similarly.
	
	Let $s_{\la_I}:C^I\to \Gr_{\mc{G},C^I}$ (resp.\,$s_{\la_J}$) be the section defined by \eqref{926470}. We shall show there is an equality of sections over $C^n_\xi$
	\begin{equation}
		\Phi\circ s_{\la}|_{C^n_\xi}=s_{\la_I}\times s_{\la_J}|_{C^n_\xi}.
	\end{equation}
	Recall the definition \eqref{926470} of $\bar{\mc{F}}:=\pi_*(\mc{F}\times^T G)^\Gamma$ and $\bar \eta$, where $\mc{F}$ is a $(\Gamma, T)$-torsor defined via the functor $F$ in \eqref{214694}, and $\eta$ is the map induced from the natural maps $\eta_\nu$ in \eqref{342573}. 
	
	Denote by $D_{I,\nu}$ the divisor $\big(\sum_{i\in I}\sum_{\gamma\in \Gamma} \langle \gamma(\lambda_i), \nu \rangle \Delta_{\gamma, i}\big)\cap \big(C^n_\xi\times C\big)$ in $C^n_\xi\times C$, and define $D_{J,\nu}$ similarly. One may check that $D_{I,\nu}\cap D_{J,\nu}=\emptyset$ for any $\nu$. Let $U_I\subseteq C^n_\xi\times C$ be a neighborhood of $D_I:=\big(\sum_{i\in I}\sum_{\gamma\in \Gamma} \Delta_{\gamma, i}\big)\cap \big(C^n_\xi\times C\big)$ such that $U_{I}\cap D_{J,\nu}=\emptyset$ for all $\nu$, and let $\mr{pr}_I:C^n_\xi\times C\to C^I\times C$ be the projection map, then 
	\[\ms{O}_{C^{n+1}}\bigg(\sum_{i=1}^n\sum_{\gamma\in \Gamma} \langle \gamma(\lambda_i), \nu \rangle \Delta_{\gamma, i}\bigg)\Big|_{U_I}\simeq \ms{O}_{U_{I}}(D_{I,\nu})\simeq \mr{pr}_I^*\Big(\ms{O}_{C^I\times C}\big( \mr{pr}_I(D_{I,\nu}) \big) \Big)\big|_{U_{I}}.\]
	By the definition of map \eqref{134631}, we have an equality of sections over $C^n_\xi$
	\[\Phi\circ s_{\la}|_{C^n_\xi} = \big((\{P_i\}_{i\in I},\bar{\mc{F}}_{I},\bar{\eta}_I),(\{P_j\}_{j\in J},\bar{\mc{F}}_{J},\bar{\eta}_J)\big)|_{C^n_\xi},\]
	where $\bar{\mc{F}}_{I}=\pi_*(\mc{F}_I\times^T G)^\Gamma$ is a $\mc{G}$-torsor, $\mc{F}_I$ is the $(\Gamma,T)$-torsor constructed via the functor $F_I:\mr{Rep}(T)\to\mr{Coh}^{\mr{lf}}(C^I\times C)$ given by
	$F_I( \C_{\nu}) = \ms{O}_{C^I\times C}\big(\mr{pr}_I(D_{I,\nu})\big)$,
	and $\eta_I$ is the trivialization induced from the natural maps $\eta_{I,\nu}:\ms{O}_{C^I\times C\setminus \mr{pr}_I(D_{I})}\big(\mr{pr}_I(D_{I,\nu})\big) \simto \ms{O}_{C^I\times C\setminus \mr{pr}_I(D_{I})}$, $\nu\in X^*(T)$.
	By definition \eqref{926470}, $(\{P_i\}_{j\in I},\bar{\mc{F}}_{I},\bar{\eta}_I)$ is exactly the section $s_{\la_I}:C^I\to \Gr_{\mc{G},C^I}$. Similarly, $(\{P_j\}_{i\in J},\bar{\mc{F}}_{J},\bar{\eta}_J)=s_{\la_J}$. This shows that $\Phi\circ s_{\la}|_{C^n_\xi}=s_{\la_I}\times s_{\la_J}|_{C^n_\xi}$. We also notice that the map $\Phi$ is compatible with the factorization of $L^+\G_{C^n}$.	It follows that the map \eqref{134631} gives an isomorphism
	$\Phi: \Grb^{\la}_{\mc{G},C^{n}_\xi} \simeq \Big(\Grb^{\la_I}_{\mc{G},C^{I}}\times \Grb^{\la_J}_{\mc{G},C^{J}} \Big)\big|_{C^{n}_\xi}$.
\end{proof}

\subsection{Convolution Grassmannian}\label{sect_conv}
With the same setup as in Section \ref{993431}, we define the convolution Grassmannian $\mr{Conv}_{\mc{G},C^n}$ as follows: for each $\C$-algebra $R$,
\[ \mr{Conv}_{\mc{G},C^n}(R):=\left\{ \big( \vec{p}, (\mc{F}_i, \beta_i)_{i=1}^n \big) \,\middle|\, 
\begin{aligned}
	&\vec{p}=(p_1, \ldots, p_n) \in C^n(R),\\
	&\mc{F}_1,\ldots, \mc{F}_n \text{ are } \mc{G}\text{-torsors on } \bar{C}, \\
	&\beta_i: \mc{F}_i|_{\bar C_{R} \setminus \Delta_{\bar{p}_i}}\simeq \mc{F}_{i-1}|_{\bar C_{R} \setminus \Delta_{\bar{p}_{i}}}, \mc{F}_0:=\mathring{\mc{F}}.
\end{aligned}\right\}.\]

There is a left $L^+\mc{G}_{C^n}$-action on $\mr{Conv}_{\mc{G},C^n}$ as follows. Given $(\vec{p},\zeta)\in L^+\mc{G}_{C^n}(R)$ and $ \big(\vec{p}, (\mc{F}_i, \beta_i)_{i=1}^n \big)\in \mr{Conv}_{\mc{G},C^n}(R)$, by Beauville-Laszlo's Lemma, there exists a $\mc{G}$-torsor $\mc{F}_i'$ with isomorphisms $\delta_i:{\mc{F}}'_i|_{ \hat{D}_i }\simeq \mc{F}_i|_{ \hat{D}_i }$ and $\alpha_i: \mc{F}_i'|_{\bar C_R\setminus D_i } \simeq \mathring{\mc{F}}|_{\bar C_R\setminus D_i}$ such that $\alpha_i\circ\delta^{-1}_i|_{\hat{D}_i^*}=\zeta\circ\beta_1\circ\cdots\circ \beta_i|_{\hat{D}_i^*}$, where $\hat{D}_i$ is the formal completion of $\bar C_R$ along the divisor $D_i:=\underset{\scriptscriptstyle1\leq j\leq i }{\bigcup}\Delta_{\bar p_j}$ and $\hat{D}_i^* :=\hat{D}_i \setminus D_i$. Thus, we get the following map 
\begin{align*} 
	L^+\mc{G}_{C^n}\times \mr{Conv}_{\mc{G}, C^n} & \to \mr{Conv}_{\mc{G}, C^n} \\
	\Big( (\vec{p},\zeta), \big(\vec{p}, (\mc{F}_i, \beta_i)_{i=1}^n \big) \Big) &\mapsto \big(\vec{p}, (\mc{F}'_i, \beta'_i)_{i=1}^n \big),
\end{align*}
where $\beta_i':\mc{F}_i'|_{\bar C_R\setminus\Delta_{\bar p_i}} \to \mc{F}_{i-1}'|_{\bar C_R\setminus\Delta_{\bar p_i}}$ is given by
$ \beta_i' =
\begin{cases}
	\delta_{i-1}^{-1}\circ \beta_i\circ\delta_i|_{\hat{D}_{i-1}}\\
	\alpha_{i-1}^{-1} \circ\alpha_i|_{\bar C_R\setminus D_i}
\end{cases}$.

Moreover, there is a convolution map
\begin{equation}\label{818710}
	\mf{m}: \mr{Conv}_{\mc{G},C^n} \to \Gr_{\mc{G}, C^n},
\end{equation}
sending $\big(\vec{p}, (\mc{F}_i, \beta_i)_{i=1}^n \big) $ to $\big(\vec{p}, \mc{F}_n, \beta_1\circ\cdots\circ\beta_n\big) $. One can check that this map is $L^+\mc{G}_{C^n}$-equivariant. Let $\xi_0=( \{1\},\{2\},\ldots, \{n\})$ be the finest partition of $[n]$. We have the following isomorphism over $C^n_{\xi_0}=\{ \vec{p}\in C^n \mid \bar{p}_i\neq \bar{p}_j, \ \forall i\neq j \}$,
\begin{equation}\label{626689}
	\varphi_n:\mr{Conv}_{\mc{G},C^n}|_{C_{\xi_0}^n} \xrightarrow{\mf{m}} \Gr_{\mc{G},{C}^n_{\xi_0}} \simto \big(\Gr_{\mc{G},C}\times\cdots\times \Gr_{\mc{G},C}\big)|_{C_{\xi_0}^n},
\end{equation}
sending $\big(\vec{p}, (\mc{F}_i, \beta_i)_{i=1}^n \big) $ to $\big( ( p_i, \mc{F}_i|_{\hat{\Delta}_{\bar p_i}}, \beta_1\circ\cdots\circ\beta_i|_{\hat{\Delta}^*_{\bar p_i}} )_{i=1}^n \big)$.

Following \cite[(3.1.22)]{Zhu17}, we define an $L^+\mc{G}_C$-torsor $\mb{E}$ over $\mr{Conv}_{\mc{G},C^{n-1}}\times C$ by
\[ \mb{E}(R):=\left\{ \big( \vec{p}, (\mc{F}_i, \beta_i)_{i=1}^{n-1}, \alpha \big) \,\middle|\, 
\begin{aligned}
	&\vec{p}=(p_1, \ldots, p_n) \in C^n(R),\\
	&\mc{F}_1,\ldots, \mc{F}_{n-1} \text{ are } \mc{G}\text{-torsors on } \bar{C}, \\
	&\beta_i: \mc{F}_i|_{\bar C_{R} \setminus \Delta_{\bar{p}_i}}\simeq \mc{F}_{i-1}|_{\bar C_{R} \setminus \Delta_{\bar{p}_{i}}}, \mc{F}_0:=\mathring{\mc{F}},\\
	&\alpha: \mc{F}_{n-1}|_{\hat \Delta_{\bar p_n}} \to \mathring{\mc{F}}|_{\hat \Delta_{\bar p_n}}.
\end{aligned}\right\}.\]
Then, there is an isomorphism 
\begin{align}
	\label{643298} f: \mb{E}\times^{L^+\mc{G}_C} \Gr_{\mc{G}, C} &\simto \mr{Conv}_{\mc{G},C^n},\\
	\nonumber \Big( \big( \vec{p}, (\mc{F}_i, \beta_i)_{i=1}^{n-1}, \alpha \big), (\mc{F},\beta) \Big) & \mapsto \big( \vec{p}, (\mc{F}_i, \beta_i)_{i=1}^{n}\big)
\end{align}
where by Beauville-Laszlo's Lemma, there exists a $\mc{G}$-torsor $\mc{F}_n$ with isomorphisms $\beta_n: \mc{F}_n|_{\bar C_R\setminus \Delta_{\bar p_n}}\simeq \mc{F}_{n-1}|_{\bar C_R\setminus \Delta_{\bar p_n}}$ and $\delta:\mc{F}_n|_{\Delta_{\bar p_n}}\simeq \mc{F}|_{\Delta_{\bar p_n}}$ such that $\beta_n\circ\delta^{-1}|_{\hat{\Delta}^*_{\bar p_n}}=\alpha^{-1}\circ\beta|_{\hat{\Delta}^*_{\bar p_n}}$. Composing $f$ with the map \eqref{626689}, we get an isomorphism over $C^n_{\xi_0}$
\begin{align}
	\label{086999} (\varphi\circ f)|_{C^n_{\xi_0}}: \big(\mb{E}\times^{L^+\mc{G}_C} \Gr_{\mc{G}, C}\big)|_{C_{\xi_0}^n} &\simto \big(\Gr_{\mc{G},C}\times\cdots\times \Gr_{\mc{G},C}\big)|_{C_{\xi_0}^n} \\
	\nonumber \Big( \big( \vec{p}, (\mc{F}_i, \beta_i)_{i=1}^{n-1}, \alpha \big), (\mc{F},\beta) \Big) & \mapsto \Big( \big( p_i, \mc{F}_i|_{\hat{\Delta}_{\bar p_i}}, \beta_1\circ\cdots\circ\beta_i|_{\hat{\Delta}^*_{\bar p_i}} \big)_{i=1}^{n} \Big).
\end{align}
By the constructions of $\mc{F}_n$ and $\beta_n$ above, we have 
\[(p_n, \mc{F}_n|_{\hat{\Delta}_{\bar p_n}}, \beta_1\circ\cdots\circ\beta_n|_{\hat{\Delta}^*_{\bar p_n}}) =(p_n, \mc{F}|_{\hat{\Delta}_{\bar p_n}}, \beta_1\circ\cdots\circ\beta_{n-1}\circ\alpha^{-1}\circ\beta|_{\hat{\Delta}^*_{\bar p_n}}).\]
Combining \eqref{086999} and \eqref{626689}, we get an isomorphism over $C^n_{\xi_0}$
\begin{align}
	\label{346788} \psi:\big(\mb{E}\times^{L^+\mc{G}_C} \Gr_{\mc{G}, C}\big)|_{C_{\xi_0}^n} &\simto \big(\mr{Conv}_{\mc{G},C^{n-1}}\times \Gr_{\mc{G},C}\big)|_{C_{\xi_0}^n} \\
	\nonumber \Big( \big( \vec{p}, (\mc{F}_i, \beta_i)_{i=1}^{n-1}, \alpha \big), (\mc{F},\beta) \Big) & \mapsto \Big( \big( (p_i, \mc{F}_i, \beta_i)_{i=1}^{n-1}\big), (p_n,\mc{F}|_{\hat{\Delta}_{\bar p_n}}, \zeta\circ\beta|_{\hat{\Delta}^*_{\bar p_n}}) \Big),
\end{align}
where $\zeta=(\beta_1\circ\cdots\circ\beta_{n-1})|_{\hat{\Delta}_{\bar p_n}}\circ\alpha^{-1}\in L^+\mc{G}_{p_n}$.

In summary, we have the following commutative diagram
\begin{equation}\label{129303}
	\xymatrixcolsep{3pc}
	\xymatrix{
		\big(\mb{E}\times^{L^+\mc{G}_C} \Gr_{\mc{G}, C}\big)|_{C_{\xi_0}^n} \ar[d]_{\psi}^[@]{\sim} \ar[r]^-{\sim}_-{f|_{C^n_{\xi_0}}} & \mr{Conv}_{\mc{G},C_{\xi_0}^n} \ar[d]_[@]{\sim}^{\varphi_{n}} \\
		 \big(\mr{Conv}_{\mc{G},C^{n-1}}\times \Gr_{\mc{G},C}\big)|_{C_{\xi_0}^n} \ar[r]^-{\sim}_-{\varphi_{n-1}\times \mr{id}} & \big(\Gr_{\mc{G},C}\times\cdots\times\Gr_{\mc{G},C}\times \Gr_{\mc{G},C}\big)|_{C_{\xi_0}^n}
	 }.
\end{equation}
\begin{definition}\label{324923}
	The convolution Schubert variety $\overline{\mr{Conv}}^{\la}_{\mc{G},C^n}$ is defined inductively as follows: when $n=1$, set $\overline{\Conv}_{\mc{G},C}^\lambda:=\Grb_{\mc{G},C}^\lambda$; when $n\geq 2$, define 
	\[\overline{\Conv}_{\mc{G},C^n}^{\la}:= \mb{E}|_{\overline{\Conv}_{\mc{G},C^{n-1}}^{ (\lambda_1,\ldots,\lambda_{n-1}) }\times C } \times^{L^+\mc{G}_C} \Grb_{\mc{G},C}^{\lambda_n}.\] 
\end{definition}
Now we would like to give another construction of the convolution Schubert variety. For each character $\nu$ of $T$, we define
\[F_i( \C_{\nu}):= \ms{O}_{C^{n+1}}\bigg(\sum_{k=1}^i\sum_{\gamma\in\Gamma}\langle \gamma(\lambda_k), \nu \rangle \Delta_{\gamma, k}\bigg).\]
This assignment extends to a $\mathbb{C}$-linear tensor functor $F_i:\mr{Rep}(G)\to \mr{Coh}^{\mr{lf}}(C^{n+1})$. By the same argument as in Section \ref{993431}, $F_i$ gives a locally trivial $(\Gamma,T)$-torsor $\mc{F}_i$ over $C^{n+1}$. Moreover, the natural maps $\beta_{\nu,i}: F_i(\C_\nu)|_{C^n\times\bar C-\Delta_{\bar P_i}} \simto F_{i-1} (\C_\nu)|_{C^n\times\bar C-\Delta_{\bar P_i}}$ for all $\nu\in X_*(T)$ induce an isomorhism $\beta_i:\mc{F}_i|_{C^n\times\bar C-\Delta_{\bar P_i}} \to \mc{F}_{i-1}|_{C^n\times\bar C-\Delta_{\bar P_i}} $. 

Set $ \mc{F}_{i,G}=\mc{F}_i\times^T G$. Then $\bar{\mc{F}}_{i}:=\pi_*(F_{i,G})^\Gamma$ is a $\mc{G}$-torsor over $C^n\times \bar C$, and $\beta_i$ induces an isomorohism $\bar\beta_i: \bar{\mc{F}}_{i}|_{C^n\times\bar C-\Delta_{\bar P_i}} \simeq\bar{\mc{F}}_{i-1}|_{C^n\times\bar C-\Delta_{\bar P_i}}$, where by convention $\bar{\mc{F}}_{0}$ is the trivial torsor. We construct a section $\tilde s_{\la}:C^n\to \mr{Conv}_{\mc{G},C^n}$ as follows:
\begin{equation}\label{282313}
	\tilde s_{\la}=\big(P_1,\ldots, P_n, (\bar{\mc{F}}_{i}, \bar\beta_i)_{i=1}^n\big).
\end{equation}
We similary define a morphism $\tilde{o}_{\la}:L^+\mc{G}_{C^n}\to \mr{Conv}_{\mc{G},C^n}$ as in Definition \ref{926470}. Denote by $\widetilde{\Conv}^{\la}_{C^n} $ the schematic image of $\tilde{o}_{\la}$. 
\begin{lem}\label{901319}
	For any partition $\xi=(I_1,\ldots,I_k)$ of $[n]$, there is an isomorphism $\Phi: \Conv_{\mc{G},C^n_\xi}\to \big(\prod_{i=1}^k\Conv_{\mc{G},C^{I_i}}\big)|_{C^n_\xi}$ satisfying the following commutative diagram
	\begin{equation}\label{998995}
		\xymatrix{
		\Conv_{\mc{G},C^n_\xi} \ar[d]_{\Phi}^[@]{\sim} \ar[r]^{\mf{m}} & \Gr_{\mc{G},C^n_\xi} \ar[d]_[@]{\sim}^{\eqref{447893}} \\
		\big(\prod_{i=1}^k\Conv_{\mc{G},C^{I_i}}\big)|_{C^n_\xi} \ar[r]^-{\prod \mf{m}} & \big(\prod_{i=1}^k \Gr_{\mc{G},C^{I_i}}\big)|_{C^n_\xi} }.
	\end{equation}
	Moreover, we have $\Phi\circ \tilde{s}_{\la}|_{C^n_\xi}=\prod \tilde{s}_{\la_{I_i}}|_{C^n_\xi}$. It follows that $\Phi$ induces an isomorphism
	\[\widetilde{\Conv}^{\la}_{\mc{G},C^n_\xi}\simeq \Big(\prod_{i=1}^k\widetilde{\Conv}^{\la_{I_i}}_{\mc{G},C^{I_i}}\Big)\big|_{C^n_\xi}.\]
\end{lem}
\begin{proof}
	We first define the map $\Phi$ as follows,
	\begin{align}\label{466981}
		\Phi:\Conv_{\mc{G},C^n_\xi} & \to \Big(\prod_{i=1}^k\Conv_{\mc{G},C^{I_i}}\Big)\big|_{C^n_\xi} \\
		\nonumber \big( \vec{p}, (\mc{F}_i, \beta_i)_{i=1}^{n}\big) & \mapsto {\textstyle\prod} \big( \vec{p}_{I_i},(\mc{F}'_j,\beta'_j)_{j\in I_i}\big)
	\end{align}
	where by Beauville-Laszlo's Lemma, for any $j\in I_i$ and the index $j'$ right before $j$ in the ordered set $I_i$, there exists a $\mc{G}$-torsor $\mc{F}'_j$ with isomorphisms $\alpha_j:\mc{F}'_j|_{\hat\Delta_{\bar p_j}}\to \mc{F}_j|_{\hat\Delta_{\bar p_j}}$ and $\beta_j': {\mc{F}'_j}|_{\bar C_R\setminus \Delta_{\bar p_j}}\to {\mc{F}'_{j'}}|_{\bar C_R\setminus \Delta_{\bar p_j}}$ such that $\beta'_j\circ\alpha^{-1}_j|_{\hat\Delta_{\bar p_j}^*}={\beta'}_{j'}^{-1}\circ\beta_1\circ\cdots\circ\beta_j|_{\hat\Delta_{\bar p_j}^*}$. One may check this is an isomorphism and satisfies the commutative diagram \eqref{998995}. 
	
	Let $\tilde s_{\la_I}:C^I\to \Conv_{\mc{G},C^I}$ be the section defined by \eqref{282313}, we shall show $ \textstyle \Phi\circ \tilde{s}_{\la}|_{C^n_\xi}= \prod \tilde{s}_{\la_{I_i}}|_{C^n_\xi}$ over $C^n_\xi$ for any partition $\xi$.
For the simplicity of illustration,	we will only prove a special case for $n=4$ and $\xi=(I,J)=(\{1,3\},\{2,4\})$, which indicates the argument in general. Recall the definition \eqref{282313} of $\bar{\mc{F}}_{i}:=\pi_*(\mc{F}_i\times^T G)^\Gamma$ and $\bar \eta_i$, where $\mc{F}_i$ is a $(\Gamma, T)$-torsor defined via the functor $F_i: \mr{Rep}(T)\to \mr{Coh}^{\mr{lf}}(C^{n+1})$ given by 
	$$F_i( \C_{\nu}) = \ms{O}_{C^{n+1}}\bigg(\sum_{j=1}^i\sum_{\gamma\in \Gamma} \langle \gamma(\lambda_j), \nu \rangle \Delta_{\gamma, j}\bigg),$$
	 and $\eta_i$ is the map induced from the natural isomorphisms $\eta_{\nu,i}:F(\C_\nu)|_{C^n\times C\setminus \bigcup_{j\leq i}\Delta_{ \gamma,j}} \simto F^0 (\C_\nu)|_{C^n\times C\setminus \bigcup_{j\leq i}\Delta_{ \gamma, j}} $, $\nu\in X^*(T)$. Denote by $D_{i,\nu}$ the divisor $\big(\sum_{\gamma\in \Gamma} \langle \gamma(\lambda_i), \nu \rangle \Delta_{\gamma, i}\big)\cap \big(C^n_\xi\times C\big)$ in $C^n_\xi\times C$. One may check that $D_{i,\nu}\cap D_{j,\nu}=\emptyset$ for any $i=1,3$, $j=2,4$. Let $U_{I}\subseteq C^n_\xi\times C$ be a neighborhood of $\big(\bigcup_{i\in I,\gamma\in\Gamma} \Delta_{\gamma, i}\big)\cap \big(C^n_\xi\times C\big)$ such that $U_I\cap (D_{2,\nu}\cup D_{4,\nu})=\emptyset$ for any $\nu$, and let $\mr{pr}_I:C^n_\xi\times C\to C^I\times C$ be the projection map. Define $U_J$ and $\mr{pr}_J$ similarly. Then 
	\begin{gather*}
		\ms{O}_{C^{n+1}}\bigg(\sum_{\gamma\in \Gamma} \langle \gamma(\lambda_1), \nu \rangle \Delta_{\gamma, 1}\bigg)\Big|_{U_I}\simeq \ms{O}_{U_I}(D_{1,\nu})\simeq \mr{pr}_I^*\Big(\ms{O}_{C^I\times C}\big( \mr{pr}_I(D_{1,\nu}) \big) \Big)\big|_{U_I},\\
		\ms{O}_{C^{n+1}}\bigg(\sum_{i=1}^3\sum_{\gamma\in \Gamma} \langle \gamma(\lambda_i), \nu \rangle \Delta_{\gamma, i}\bigg)\Big|_{U_I}\simeq \ms{O}_{U_I}(D_{1,\nu}\cup D_{3,\nu})\simeq \mr{pr}_I^*\Big(\ms{O}_{C^I\times C}\big( \mr{pr}_I(D_{1,\nu}\cup D_{3,\nu}) \big) \Big)\big|_{U_I}.
	\end{gather*}
	By the definition of the map \eqref{466981}, we have an equality of sections over $C^n_\xi$
	\[\Phi\circ \tilde s_{\la}|_{C^n_\xi} = \big((P_1,P_3,\bar{\mc{F}}_{1},\bar{\mc{F}}_{3}',\bar{\eta}_1,\bar{\eta}'_3),(P_2,P_4,\bar{\mc{F}}'_{2},\bar{\mc{F}}'_{4},\bar{\eta}'_2,\bar{\eta}'_4)\big)|_{C^n_\xi},\]
	where $\bar{\mc{F}}'_{3}=\pi_*(\mc{F}'_3\times^T G)^\Gamma$, and $\mc{F}'_3$ is the $(\Gamma,T)$-torsor constructed via the following functor $F'_3:\mr{Rep}(T)\to\mr{Coh}^{\mr{lf}}(C^I\times C)$
	\[F'_3: \C_{\nu}\mapsto \ms{O}_{C^I\times C}\big(\mr{pr}_I(D_{1,\nu}\cup D_{3,\nu})\big),\]
	$\eta'_3$ is the map induced form the natural maps 
	\[\eta'_{3,\nu}:\ms{O}_{C^I\times C\setminus \mr{pr}_I(D_{3,\nu})}\big(\mr{pr}_I(D_{1,\nu}\cup D_{3,\nu})\big) \to \ms{O}_{C^I\times C\setminus \mr{pr}_I(D_3)}(\mr{pr}_I(D_{1,\nu})),\]
	for all $\nu\in X^*(T)$.
	By definition \eqref{282313}, $(P_1,P_3,\bar{\mc{F}}_{1},\bar{\mc{F}}_{3}',\bar{\eta}_1,\bar{\eta}'_3)$ is exactly the section $\tilde s_{\la_I}:C^I\to \Conv_{\mc{G},C^I}$. Similarly, $(P_2,P_4,\bar{\mc{F}}'_{2},\bar{\mc{F}}'_{4},\bar{\eta}'_2,\bar{\eta}'_4)=\tilde s_{\la_J}$. This shows that 
	$\Phi\circ \tilde s_{\la}|_{C^n_\xi}=\tilde s_{\la_I}\times \tilde s_{\la_J}|_{C^n_\xi}$.
	Since $\Phi$ is $L^+\mc{G}_{C^n_\xi}$-equivariant, it follows that the map \eqref{466981} induces an isomorphism $\Phi: \widetilde{\Conv}^{\la}_{\mc{G},C^{n}_\xi} \simeq \big(\widetilde{\Conv}^{\la_I}_{\mc{G},C^{I}}\times \widetilde{\Conv}^{\la_J}_{\mc{G},C^{J}} \big)|_{C^{n}_\xi}$.
\end{proof}
Recall that there is an isomorphism \eqref{626689} over $C_{\xi_0}^n$:
\[ \varphi_n:\mr{Conv}_{\mc{G},C^n}|_{C_{\xi_0}^n} \to \big(\Gr_{\mc{G},C}\times\cdots\times \Gr_{\mc{G},C}\big)|_{C_{\xi_0}^n}. \]
By Lemma \ref{901319}, this map $\varphi$ induces an isomorphism
\begin{equation}\label{234926}
	\widetilde{\Conv}^{\la}_{\mc{G},C^n_{\xi_0}} \simeq \big(\Grb^{\lambda_{1}}_{\mc{G},C}\times\cdots\times \Grb^{\lambda_n}_{\mc{G},C}\big)|_{C^n_{\xi_0}}.
\end{equation}
\begin{prop}\label{234953}
	The morphism $f$ defined in \eqref{643298} induces an isomorphism 
	\begin{equation}\label{138990}
		\overline{ \Conv}_{\mc{G},C^n}^{\la}\simeq \widetilde{\Conv}^{\la}_{C^n}.
	\end{equation} 
\end{prop}
\begin{proof}
	We prove it by induction on $n$. When $n=1$, by definition $\tilde{s}_{\lambda}=s_{\lambda}$. Hence, $\widetilde{\Conv}^{\lambda}_{\G,C}=\Grb_{\G,C}^\lambda=\overline{\Conv}_{\G,C}^\lambda$. When $n\geq 2$, by induction we have $\overline{\Conv}_{\mc{G},C^{n-1}}^{\lambda_1,\ldots,\lambda_{n-1}}\simeq \widetilde{\Conv}^{\lambda_1,\ldots,\lambda_{n-1}}_{\G,C^n}$. Then, the map $\psi$ in \eqref{346788} induces an isomorphism 
	\begin{equation} \label{342960}
		\overline{\Conv}_{\mc{G},C^n_{\xi_0}}^{\la}\simeq \big(\widetilde{\Conv}^{\lambda_1,\ldots,\lambda_{n-1}}_{\G,C^n}\times \Grb_{\G,C}^{\lambda_n}\big)|_{C^n_{\xi_0}}.
	\end{equation}
	Combining with the isomorphism \eqref{234926}, we get the following diagram
	\begin{equation*}
		\xymatrixcolsep{3pc}
		\xymatrix{
			\overline{\Conv}_{\mc{G},C^n_{\xi_0}}^{\la} \ar[d]_{\eqref{342960}}^[@]{\sim} & \widetilde{\Conv}^{\la}_{\mc{G},C^n_{\xi_0}} \ar[d]_[@]{\sim}^{\eqref{234926}} \\
			\big(\widetilde{\Conv}^{\lambda_1,\ldots,\lambda_{n-1}}_{\G,C^n}\times \Grb_{\G,C}^{\lambda_n}\big)|_{C^n_{\xi_0}} \ar[r]^-{\sim}_-{\eqref{234926}\times \mr{id}} & \big(\Grb^{\lambda_{1}}_{\mc{G},C}\times\cdots\times \Grb^{\lambda_n}_{\mc{G},C}\big)|_{C^n_{\xi_0}}
		}.
	\end{equation*}
	This gives an isomorphism $\overline{\Conv}_{\mc{G},C^n_{\xi_0}}^{\la}\simeq \widetilde{\Conv}^{\la}_{\mc{G},C^n_{\xi_0}}$. By the commutativity of the diagram \eqref{129303}, this isomorphism is exactly the restriction of the map $f$ defined in \eqref{643298}. Since both of them are reduced and irreducible, we must have an isomorphism 	$\overline{\Conv}_{\mc{G},C^n}^{\la}\simeq \widetilde{\Conv}^{\la}_{\mc{G},C^n}$.
\end{proof}

\subsection{Flatness of BD Schubert varieties}\label{029728}
Let $G$ be a simple algebraic group of adjoint type over $\C$, and let $\sigma$ be the standard automorphism defined in Section \ref{133740}. Let $\mc{K}=\C((t))$ and $\mc{O}=\C[[t]]$. We define the action of $\sigma$ on $\mc{K}$ and $\mc{O}$ by $\sigma(t)=\epsilon^{-1} t$, where $\epsilon$ is a fixed primitive $m$-th root of unity. Denote by $\bar{\mc{K}}=\mc{K^\sigma}$ and $\bar{\mc{O}}=\mc{O}^\sigma$ the $\sigma$-fixed points in $\mc{K}$ and $\mc{O}$ respectively.

Let $\ms{G}$ be the $\sigma$-fixed point subgroup scheme $\mr{Res}_{\mc{O}/\bar{\mc{O}}} ( G_\mc{O})^\sigma$ of the Weil restriction $\mr{Res}_{\mc{O}/\bar{\mc{O}}} ( G_\mc{O}) $. Then $\ms{G}$ is a special parahoric group scheme over $\bar{\mc{O}}$ in the sense of Bruhat-Tits, see \cite[Section 2.2]{BH}. Let $L\ms{G}$ be the loop group scheme and $L^+\ms{G}$ be the jet group scheme with $L\ms{G}(R)=\ms{G}(R((t)))$ and $L^+\ms{G}(R)=\ms{G}(R[[t]])$, for any $\C$-algebra $R$. We define the affine grassmannian $\Gr_\ms{G}$ to be the fppf quotient $L\ms{G}/L^+\ms{G}$, and call it a twisted affine Grassmannian of $\ms{G}$.  Similarly, replacing $\ms{G}$ by $G$, we define the affine Grassmannian $\Gr_G$ of $G$ to be the fppf quotient $LG/L^+G$. In particular, $\Gr_{\ms{G}}(\C)=G(\mc{K})^\sigma/G(\mc{O})^\sigma$ and $\Gr_{G}(\C)=G(\mc{K})/G(\mc{O})$. We denote by $e_0$ the base point in $\Gr_{\ms{G}}(\C)$ and $\Gr_{G}(\C)$.

For any dominant coweight $\lambda\in X_*(T)$, we are associated to an element $t^\lambda\in T(\mc{K})$. Then, the closure of the orbit $G(\mc{O})\cdot t^\lambda e_0$ in $\Gr_G(\C)$ is called an affine Schubert variety, denoted by $\Grb_G^\lambda$. Let $n^\lambda$ be the norm of $t^\lambda$, i.e. $n^\lambda:=\prod_{i=0}^{m-1} \sigma^i(t^\lambda)\in T(\mc{K})^\sigma$. Let $\bar\lambda\in X_*(T)_{\sigma}$ be the image of $\lambda$ under the projection map $X_*(T)\to X_*(T)_\sigma$, and denote by $e_{\lb}$ the point $n^\lambda e_0\in \Gr_{\ms{G}}(\C)$. One can define a Schubert cell $\Gr_{\ms{G}}^{\lb}:=G(\mc{O})^\sigma e_{\lb}$. We define the Schubert variety $\Grb_\ms{G}^{\lb}$ to be the reduced closure of $\Gr_{\ms{G}}^{\lb}$ in $\Gr_\ms{G}$. The dimensions of $\Grb_G^\lambda$ and $\Grb_\ms{G}^{\lb}$ are given by 
\begin{equation}\label{eq_dim}
\dim \Grb_G^\lambda=\dim \Grb_\ms{G}^{\lb}=2\langle \lambda, \rho \rangle,
\end{equation}
where $\rho$ is the summation of fundamental weights of $G$, cf.\cite[Proposition 2.1.5]{Zhu17}\cite[Sectioin 2.3]{BH}.  Moreover, there is a level one line bundle $\ms{L}$ on $\Gr_{\ms{G}}$ such that 
\begin{equation}\label{eq_sch_dem}
	H^0\big(\Grb_{\ms{G}}^{\bar{\lambda}}, \ms{L}^c \big) ^\vee\simeq D^\sigma(c,\bar{\lambda}),
\end{equation}
as $\gs$-modules, where $D^\sigma(c,\bar{\lambda})$ is the twisted affine Demazure module defined in Section \ref{925386}, cf.\,\cite[Theorem 3.11]{BH}).

Let $\Gamma$ be the cyclic group generated by $\sigma$, and let $C$ be an irreducible smooth curve with a faithful $\Gamma$-action. From now on, we assume that the ramified points of $C$ are totally ramified, i.e.\,the stabilizer $\Gamma_p$ is equal to $\Gamma$ for any ramified point $p$. Moreover, when $C$ is the affine line $\A^1$, we require the $\Gamma$-action to be the `standard' one given by 
\begin{equation}\label{eq_standard}
	\sigma(p)=\epsilon p, \quad \text{ for any } p\in \A^1.  
\end{equation}
\begin{thm}\label{231044}
	With the assumptions above, the global Schubert variety $\Grb^\lambda_{\mc{G},C}$ is flat over $C$, and for any $p\in C$ the fiber $\Grb^\lambda_{\mc{G},p}$ is reduced. Moreover, 
	\begin{equation}\label{200864}
		\Grb^\lambda_{\mc{G},p}\simeq 
		\begin{cases}
		\Grb_{\ms{G}}^{\lb} & \text{for } p\in R;\\
		\Grb_{G}^{\lambda} & \text{for } p\not\in R,
		\end{cases}
	\end{equation}
	where $R$ denotes the ramification locus of the action of $\Gamma$ on $C$.
\end{thm}
\begin{proof}
	When $C=\A^1$, the theorem was proven in \cite[Theorem 3]{Zhu14}. Now, let $C$ be a general $\Gamma$-curve whose points are either unramified or totally ramified. Note that the theorem is local, we shall reduce it to the affine case. For any ramified point $x\in C$, there exists an open neighborhood $U$ of $x$ with a $\Gamma$-equivariant \'etale map $U\to \A^1$ sending $x$ to $0$. 
	By Corollary \ref{prop_etale_base}, we have a base change isomorphism
	\[\Gr_{\mc{G}_{\bar U}, U} \simeq \Gr_{\mc{G}_{\bar\A^1},\A^1}\times_{\A^1} U,\]
	and the left $L^+\mc{G}_U$-action on $\Gr_{\mc{G}_{\bar U}, U}$ commutes with the base change. Thus, 
	\[\Grb_{\mc{G}_{\bar U}, U}^\lambda \simeq \Grb^\lambda_{\mc{G}_{\bar\A^1},\A^1}\times_{\A^1} U.\] 
	Hence, the theorem holds over $U$. Similarly, one can show that the theorem also holds over $C\backslash R$. This completes the proof of the theorem.
\end{proof}
\begin{lem}\label{432849}
The convolution variety $ \overline{\mr{Conv}}_{\mc{G},C^n}^{\la} $ is flat over $C^n$. Moreover, the convolution morphism $\mf{m}$ induces a surjective morphism $\mf{m}: \overline{\mr{Conv}}_{\mc{G},C^n}^{\la}\to \Grb_{\mc{G},C^n}^{\la}$. 
\end{lem}
\begin{proof}
	Recall the Definition \ref{324923} that $ \overline{\mr{Conv}}_{\mc{G},C^n}^{\la} $ is a global Schubert variety when $n=1$ and is a sequence of fibrations of global Schubert varieties when $n\geq 2$. By Theorem \ref{231044}, $\overline{\mr{Conv}}_{\mc{G},C^n}^{\la} $ must be flat over $C^n$. 
	
	By Proposition \ref{234953}, $\overline{\mr{Conv}}_{\mc{G},C^n}^{\la} $ can be identified with the schematic image of the map $\tilde{o}_{\la}:L^+\mc{G}_{C^n}\to \mr{Conv}_{\mc{G},C^n}$. By the definition \eqref{818710} of convolution map $\mf{m}$ and the $L^+\mc{G}_{C^n}$-equivariance of $\mf{m}$, we have a commutative diagram
	\begin{equation}\label{015520}
		\xymatrix{
			L^+\mc{G}_{C^n} \ar[rd]_{o_{\la}} \ar[r]^-{\tilde{o}_{\la}} &\mr{Conv}_{\mc{G},C^n} \ar[d]^-{\mf{m}} \\
			& \Gr_{\mc{G},C^n}
		}.
	\end{equation}
	Since $\mf{m}$ is proper, it gives a surjective map $\mf{m}:\overline{\Conv}^{\la}_{\mc{G},C^n} \twoheadrightarrow \Grb^{\la}_{\mc{G},C^n}$. 
\end{proof}
We define the following subset in $C^n$:
\begin{equation}\label{348782}
	\Delta_{ [n]}=\{ \vec{p}\in C^n \mid\bar{p}_i= \bar{p}_j, \ \forall 1\leq i, j \leq n \}.
\end{equation}
\begin{prop}\label{930915}
	With the same assumptions in Theorem \ref{231044}, we have 
	\begin{enumerate}
		\item 
		For any point $\vec{p}=(p,\gamma_2(p),\ldots, \gamma_n (p) ) \in \Delta_{[n]}$ with $\gamma_i\in \Gamma$, we have 
		\[\Big(\Grb_{\mc{G},\vec{p}}^{\la} \Big)_{\mr{red}} \simeq \Grb_{\mc{G},p}^{\lambda_{\vec{p}}},\]
		where $\lambda_{\vec{p}}:=\lambda_1+\gamma_2^{-1}(\lambda_2)+\cdots+\gamma_n^{-1}(\lambda_n)$.
		\item 
		For any point $\vec{p}\in C^n$, we have
		\[R^i(\mf{m}_{\vec{p}})_*\Big(\ms{O}_{\mr{Conv}_{\mc{G},\vec{p}}^{\la} } \Big)=
		\begin{cases}
		\ms{O}_{\big(\Grb_{\mc{G},\vec{p}}^{\la}\big)_{\mr{red}}}, & i=0;\\
		0, & \forall i\geq 1.
		\end{cases}\]
	\end{enumerate}
\end{prop}
\begin{proof}
	(1) Taking the reduced fibers of both sides of the morphism $\mf{m}:\overline{\Conv}^{\la}_{\mc{G},C^n} \twoheadrightarrow \Grb^{\la}_{\mc{G},C^n}$ in Lemma \ref{432849}, we get a surjective map for any $\vec{p}\in C^n$,
	\[\mf{m}_{\vec{p}}: \overline{\Conv}^{\la}_{\mc{G},\vec{p}}=\Big(\overline{\Conv}^{\la}_{\mc{G},\vec{p}}\Big)_{\mr{red}} \twoheadrightarrow \Big(\Grb^{\la}_{\mc{G},\vec{p}}\Big)_{\mr{red}}.\]
	
	Now, let $\vec{p}=(p, \gamma_2(p), \ldots, \gamma_n(p))$ be any point in $\Delta_{[n]}$. To prove part (1), it suffices to show the image of $\overline{\Conv}^{\la}_{\mc{G},\vec{p}}$ under the convolution map $\mf{m}: \mr{Conv}_{\mc{G},C^n} \to \Gr_{\mc{G}, C^n}$ is $\Grb^{\lambda_{\vec{p}}}_{\mc{G},p}$. When $p\in C$ is unramified, the composition of morphisms \eqref{818710} \eqref{643298} restricts to the following map:
	\begin{align}
		\label{518909} \mb{E}|_{\vec{p}}\times^{L^+\mc{G}_{\gamma_n(p)}} \Gr_{\mc{G}, \gamma_n(p)} & \xrightarrow{\mathmakebox[1.5em]{\sim}} \mr{Conv}_{\mc{G},\vec{p}} \xrightarrow{\mathmakebox[1.5em]{\mf{m}}} \Gr_{\mc{G},\vec{p}} \\
		\nonumber \Big( \big( \vec{p}, (\mc{F}_i, \beta_i)_{i=1}^{n-1}, \alpha \big), (\mc{F},\beta) \Big) & \mapsto \big( \vec{p}, (\mc{F}_i, \beta_i)_{i=1}^{n}\big) \mapsto (\vec{p}, \mc{F}|_{\hat{\Delta}_{\bar p_n}}, \zeta\circ\beta|_{\hat{\Delta}^*_{\bar p_n}})
	\end{align}
	where $p_n:=\gamma_n(p)$, $\zeta=(\beta_1\circ\cdots\circ\beta_{n-1}\circ\alpha^{-1})|_{\hat{\Delta}^*_{\bar p_n}}\in G(\mc{K}_{p_n})$. It factors through the usual convolution map 
	\[G(\mc{K}_{p_n})\times^{G(\mc{O}_{p_n})}\Gr_{G,p_n}\to \Gr_{G,p_n}\simeq \Gr_{\mc{G},\vec{p}}.\] 
	By induction on $n$, for $\vec{q}:=(p,\gamma_2(p),\cdots,\gamma_{n-1}(p))\in C^{n-1}$, the image of $\overline{\Conv}_{\mc{G},\vec{q}}^{\lambda_1,\ldots,\lambda_{n-1}}$ under the convolution map is $ \Grb^{\lambda_{\vec{q}}}_{\mc{G},p}$. Thus for any $(\vec{q}, (\mc{F}_i,\beta_i)_i)\in \overline{\Conv}_{\mc{G},\vec{q}}^{\lambda_1,\ldots,\lambda_{n-1}}$, we have $(p,\mc{F}_{n-1},\beta_1\circ\cdots\circ\beta_{n-1})\in \Grb_{G,p}^{\lambda_{\vec{q}}}\simeq \Grb_{G,p_n}^{\gamma_n(\lambda_{\vec{q}})}$. Therefore, the following restriction of \eqref{518909} 
		\[\mb{E}|_{\overline{\Conv}_{\mc{G},\vec{q}}^{\lambda_1,\ldots,\lambda_{n-1}},p_n}\times^{L^+\mc{G}_{p_n}} \Grb^{\lambda_n}_{G, p_n}\to \Grb^{\gamma_n(\lambda_{\vec{q}})+\lambda_n}_{G, p_n},\]
	factors through $\mf{m}_{p_n}: \widetilde{\mr{Gr}}_{\mc{G},p_n}^{\gamma_n(\lambda_{\vec{q}})}\times^{L^+\mc{G}_{p_n}}\Grb^{\lambda_n}_{G, p_n}\to \Grb^{\gamma_n(\lambda_{\vec{q}})+\lambda_n}_{G, p_n}$, where $\widetilde{\mr{Gr}}_{\mc{G},p_n}^{\gamma_n(\lambda_{\vec{q}})}$ is the $L^+\mc{G}_p$-torsor over $\Grb^{\gamma_n(\lambda_{\vec{q}})}_{G, p_n}$. This implies that the image of $\overline{\Conv}^{\la}_{\mc{G},\vec{p}}$ under the convolution map is $\Grb^{\gamma_n(\lambda_{\vec{q}})+\lambda_n}_{G, p_n}\simeq \Grb^{\lambda_{\vec{p}}}_{G, p}\simeq \Grb^{\lambda_{\vec{p}}}_{\mc{G},p}$. The case when $p\in C$ is totally ramified can be proved similarly.
	
	For part (2), by Lemma \ref{901319} it suffices to consider the case when $\vec{p}\in \Delta_{[n]}$. We now the consider the morphism  
	$\mf{m}_{\vec{p}}: \overline{\mr{Conv}}_{\mc{G},\vec{p}}^{\la}\to \Grb_{\mc{G},\vec{p}}^{\la}$. 	From part (1) and \cite[Proposition 9.7]{PR08}, we know $(\mf{m}_{\vec{p}})_*\Big(\ms{O}_{\overline{\mr{Conv}}_{\mc{G},\vec{p}}^{\la} } \Big)=\ms{O}_{ \Grb^{\lambda_{\vec{p}}}_{\mc{G},p}}=\ms{O}_{\big(\Grb_{\mc{G},\vec{p}}^{\la}\big)_{\mr{red}}}$.

	For $i \geq 1$, we shall show $R^i(\mf{m}_{\vec{p}})_*\Big(\ms{O}_{\overline{\mr{Conv}}_{\mc{G},\vec{p}}^{\la} } \Big)=0$ by induction on $n$.	Clearly, this holds when $n=1$. Now, assume $n\geq 2$, and let $\vec{q}=(p,\gamma_2(p),\ldots,\gamma_{n-1}(p))\in C^{n-1}$. By induction, $R^i(\mf{m}_{\vec{q}})_*\Big(\ms{O}_{\overline{\mr{Conv}}_{\mc{G},\vec{q}}^{\la} } \Big)=0$. 
	Consider the following diagram
	\[\xymatrix{
		\mb{E}':=\mb{E}|_{\overline{\Conv}^{\lambda_1,\ldots, \lambda_{n-1}}_{\mc{G},\vec{q}}\times \{ p_n\}} \ar[d]_{\tilde{\mf{m}}_{\vec{q}}} \ar[r]^-{\mr{pr}} & \overline{\Conv}^{\lambda_1,\ldots, \lambda_{n-1}}_{\mc{G},\vec{q}} \ar[d]^{\mf{m}_{\vec{q}}} \\
		\widetilde{\mr{Gr}}_{\mc{G},p}^{\lambda_{\vec{q}}} \ar[r]^{\mr{pr}} & \Grb^{\lambda_{\vec{q}}}_{\mc{G},p} 
	} .\]
	By flat base change \cite[Proposition 9.3]{Hartshorne77}, we have $R^i(\tilde{\mf{m}}_{\vec{q}})_*\big(\ms{O}_{\mb{E}' } \big)=0$ for $i\geq 1$. 
	From part (1), there is a commutative diagram
	\[\xymatrix{
		\overline{\Conv}^{\la}_{\mc{G},\vec{p}}\simeq \mb{E}|_{\overline{\Conv}_{\mc{G},\vec{q}}^{\lambda_1,\ldots,\lambda_{n-1}}\times \{p_n\}}\times^{L^+\mc{G}_{p_n}} \Grb^{\lambda_n}_{\mc{G}, p_n}\ar[d]_f \ar[dr]^-{\mf{m}_{\vec{p}}} & \\
		W:=\widetilde{\mr{Gr}}_{\mc{G},p_n}^{\gamma_n(\lambda_{\vec{q}})}\times^{L^+\mc{G}_{p_n}} \Grb^{\lambda_n}_{\mc{G}, p_n} \ar[r]_-{\mf{m}_{p_n}} & \Grb^{\gamma_n(\lambda_{\vec{q}})+\lambda_n}_{\mc{G}, p_n}.
	} \]	
	Thus, $R^if_*\Big(\ms{O}_{\overline{\Conv}^{\la}_{\mc{G},\vec{p}}}\Big)=0$. Recall that $m_{p_n}$ is a partial Bott-Samelson resolution and we have $R^i({\mf{m}_{p_n}})_*(\ms{O}_W)=0$, cf.\cite[Proposition 9.7]{PR08}. Thus, $R^i(\mf{m}_{\vec{p}})_*\Big(\ms{O}_{\overline{\Conv}^{\la}_{\mc{G},\vec{p}}}\Big)=0$.
\end{proof}
\begin{thm}\label{368017}
	With the same assumptions as in Theorem \ref{231044}, the BD Schubert variety $\Grb_{\mc{G},C^n}^{\la}$ is flat over $C^n$. Moreover, $\Grb_{\mc{G},\vec{p}}^{\la}$ is reduced.
\end{thm}
\begin{proof}
Let $X=\overline{\mr{Conv}}_{\mc{G},C^n}^{\la}$ and $Y=\Grb_{\mc{G},C^n}^{\la}$.	By Lemma \ref{432849} and Proposition \ref{930915}, we can apply Corollary \ref{173291} to the morphism $\mf{m}:X\to Y$ and we get $\mf{m}_*(\ms{O}_X)$ is flat over $C^n$. To prove the first part of the theorem, it suffices to show $\mf{m}_*(\ms{O}_X)\simeq \ms{O}_Y$. Consider the natural map $\ms{O}_Y\to \mf{m}_*(\ms{O}_X)$. At every point $s=\vec{p}\in C^n$, the composition of the following maps 
	\begin{equation}\label{960342}
		\ms{O}_{Y_s}\to \mf{m}_*(\ms{O}_X)|_{Y_s}\xrightarrow[\text{part (1) of Corollary \ref{173291}}]{\sim} (\mf{m}_s)_*(\ms{O}_{X_s}) \xrightarrow[\text{Proposition \ref{930915}}]{\sim} \ms{O}_{(Y_s)_{\mr{red}}}
	\end{equation}
	is surjective. Thus, the morphism $\ms{O}_Y\to \mf{m}_*(\ms{O}_X)$ is surjective.

	Recall that $Y$ is the schematic image of $o_{\la}:L^+\mc{G}_{C^n}\to \Gr_{\mc{G},C^n}$. By Proposition \ref{234953}, $X$ is the schematic image of $\tilde{o}_{\la}:L^+\mc{G}_{C^n}\to \mr{Conv}_{\mc{G},C^n}$. 
	Thus, we have an inclusion $\ms{O}_X\hookrightarrow (\tilde{o}_{\la})_*\big(\ms{O}_{L^+\mc{G}_{C^n}}\big)$. By Proposition \ref{930915}, we have $R^1\mf{m}_*( \ms{O}_X)=0$. It follows that we have the following inclusion
	\[\mf{m}_*\ms{O}_X\hookrightarrow \big(\mf{m}_*\circ (\tilde{o}_{\la})_*\big) \big(\ms{O}_{L^+\mc{G}_{C^n}}\big)=(o_{\la})_*\big(\ms{O}_{L^+\mc{G}_{C^n}}\big). \]
	From the commutativity of the following diagram
	\[\xymatrix{
		\mf{m}_*\ms{O}_X \ar@{^{(}->}[r] & (o_{\la})_*\big(\ms{O}_{L^+\mc{G}_{C^n}}\big) \\
		\ms{O}_{Y} \ar@{->>}[u] \ar@{^{(}->}[ru]& 
	} ,\]
	we must have $\mf{m}_*\ms{O}_X\simeq \ms{O}_Y$. Thus, $\ms{O}_Y$ is flat over $C^n$. On the other hand, this also implies that the composition map \eqref{960342} is an isomorphism. In particular, it follows that $Y_s=\Grb_{\mc{G},\vec{p}}^{\la}$ is reduced.
\end{proof}

\begin{cor}\label{711570}
	With the same assumptions in Theorem \ref{231044}, we have
	\begin{enumerate}
		\item Let $\xi=(I_1,\ldots,I_k)$ be a partition of $[n]$. For any $\vec{p}\in C^n_\xi$, we have 
		\[\Grb_{\mc{G},\vec{p}}^{\la} \simeq \Grb_{\mc{G},\vec{p}_{I_1}}^{\la_{I_1}}\times \cdots\times \Grb_{\mc{G},\vec{p}_{I_k}}^{\la_{I_k}},\]
		where $\vec{p}_{I_j}\in C^{I_j}$ and $\la_{I_j}$ are defined in \eqref{053076} and \eqref{249296}.
		\item For any $\vec{p}=(p,\gamma_2(p),\ldots,\gamma_{n}(p))\in \Delta_{[n]}$, we have $\Grb_{\mc{G},\vec{p}}^{\la} \simeq \Grb_{\mc{G},p}^{\lambda_{\vec{p}}}$, 
		where $\lambda_{\vec{p}}:=\lambda_1+\gamma_2^{-1}(\lambda_2)+\cdots+\gamma_n^{-1}(\lambda_n)$.
	\end{enumerate}
\end{cor}
\begin{proof}
	Part (1) directly follows from Proposition \ref{684242}. Part (2) follows from Proposition \ref{930915} and Theorem \ref{368017}.
\end{proof}

\begin{lem}\label{215536}
The orbit $\Gr_{\mc{G}, {C^n}}^{\la}:=L^+\mc{G}_{C^n}\cdot o^{\la}$ is smooth over $C^n$.
\end{lem}
\begin{proof}
\!Observe that $\Gr_{\mc{G}, {C^n}}^{\la}$ is open in $\Grb_{\mc{G}, {C^n}}^{\la}$. By Theorem \ref{368017}, $\Gr_{\mc{G}, {C^n}}^{\la}$ is also flat over $C^n$. Since every fiber $\Gr_{\mc{G},\vec{p}}^{\la}$ is smooth, $\Gr_{\mc{G}, {C^n}}^{\la}$ is smooth over $C^n$, cf.\,\cite[Theorem 10.2]{Hartshorne77}.
\end{proof}
\begin{thm}\label{551564}
	The BD Schubert variety $\Grb^{\la}_{\mc{G},C^n}$ is normal.
\end{thm}
\begin{proof}
	%By the \'etale base change property of BD Grassmannian (cf.\,Corollary \ref{prop_etale_base}) and a similar argument in Theorem \ref{231044}, it suffices to prove the case when $C$ is the affine line $\A^1$. 
	
	When $n=1$, the theorem was proved in \cite[Section 3.3]{Zhu14}. For $n\geq2$, we prove it by induction on $n$. Given a non-trivial partition $\xi=(I,J)$ of $[n]$, by Proposition \ref{684242} there is a factorization over $C^n_{\xi}$
	\[\Grb^{\la}_{\mc{G}, C^n_{\xi}} \simeq \Big(\Grb^{\la_{I}}_{\mc{G}, C^I}\times \Grb^{\la_{J}}_{\mc{G}, C^J}\Big)\big|_{C^n_{\xi}}.\]
	By induction, $\Grb^{\la_{I}}_{\mc{G}, C^I}$ and $\Grb^{\la_{J}}_{\mc{G}, C^J}$ are normal, hence $\Grb^{\la}_{\mc{G}, C^n_{\xi}} $ is normal.  Note that the complement of the union of $C^n_\xi$ for all non-trivial partition $(I,J)$ in $C^n$ is $\Delta_{[n]}$. To show the normality around $\Delta_{[n]}$, in view of Corollary \ref{prop_etale_base} we can assume 	$C=\A^1$. For the rest of proof, we adapt an argument in \cite[Proposition 6.4]{Zhu14}. 
	
Let $X=\Grb^{\la}_{\mc{G},\A^n}$. For any $i\in I$, $j\in J$ and $0\leq k\leq m-1$, we regard $t:=z_i-\sigma^k (z_j)$ as a regular function on $X$. We shall show the local ring $\ms{O}_{X,x}$ is normal for any $x\in X$ such that $t(x)=0$. Denote by $(X_0)_{\mr{red}}$ the reduced fiber of $t:X\to \A^1$ at $0\in \A^1$. By Corollary \ref{711570}, $(X_0)_{\mr{red}}$ is isomorphic to a BD Schubert variety $\Grb^{\vec{\mu}}_{\mc{G},\A^{n-1}}$, where $\vec{\mu}$ consists of $\lambda_a$, $1\leq a\leq n$ with $a\not= i,j$, and $\lambda_i+\sigma^k(\lambda_j)$. Thus, $t\ms{O}_{X,x}$ has a unique minimal associated prime ideal $\mf{p}$. By induction, $(X_0)_{\mr{red}}$ is normal, hence $\ms{O}_{X,x}/\mf{p}$ is normal. By the dimension formula (\ref{eq_dim}), the complement of the orbit $\Gr_{\mc{G}, {C^n}}^{\la}$ in $X$ has codimension greater than or equal to $2$. Applying Lemma \ref{215536}, $X$ is regular in codimension one, which implies that the localization $(\ms{O}_{X,x})_\mf{p}$ at the minimal ideal $\mf{p}$ is regular. Thus, $\ms{O}_{X,x}$ is normal, since the three conditions of \cite[Lemma 9.12]{Hartshorne77} are satisfied. Since $i\in I$, $j\in J$ and $0\leq k\leq m-1$ are arbitrary, we conclude that $\Grb^{\la}_{\mc{G},\A^n}=X$ is normal.
\end{proof}
We conclude this subsection by showing that $\Gr_{\G,C^n}$ is a direct limit of $\Grb_{\G,C^n}^{\la}$ with respect to the standard partial order $\preceq$ on $(X_*(T)^+)^n$. The following Theorem is due to Wyatt Reeves \cite{Reeves23}.
\begin{thm}\label{498274}
	Let $S$ be a Noetherian scheme. Let $X \to Y$ be an ind-closed embedding of ind-schemes over  $S$. Suppose $X/S$ is ind-flat and $Y/S$ is ind-finite-type. Suppose $S$ admits a decomposition into a closed subscheme $Z$ and its open complement $U = S \setminus Z$ such that $X|_Z \to Y|_Z$ and $X|_U \to Y|_U$ are isomorphisms. Then $X \to Y$ is an isomorphism. \qed
\end{thm}

The following result was proved in \cite{Reeves23} in the untwisted case as a consequence of Theorem \ref{498274}.  The same method can be adapted in our twisted setting. 
\begin{prop}\label{423483}
The BD Grassmannian $\Gr_{\G,C^n}$ is a direct limit of $\Grb_{\G,C^n}^{\la}$ with respect to the standard partial order $\preceq$ on $(X_*(T)^+)^n$.
\end{prop}

\begin{proof}
	We prove by induction on $n$. When $n=1$, consider the ind-closed embedding 
	\begin{equation}\label{972364}
		\varinjlim_{\lambda}\Grb_{\G,C}^\lambda\to \Gr_{\G,C}.
	\end{equation}
	Let $R$ be the set of ramified points in $C$. Set $\mathring{C}=C\backslash R$. By Corollary \ref{prop_etale_base} (2), we have an isomorphism $\Gr_{\G, \mathring{C} }\simeq \Gr_{G, \mathring{C} }$. This isomorphism restricts to an isomorphism $\Grb^\lambda_{\G, \mathring{C} }\simeq \Grb^\lambda_{G, \mathring{C} }$. By \cite{Reeves23}, $\varinjlim\Grb_{G,\mathring{C}}^\lambda\to \Gr_{G,\mathring{C}}$ is an isomorphism. Thus, $\varinjlim\Grb_{\G,\mathring{C}}^\lambda\to \Gr_{\G,\mathring{C}}$ is also an isomorphism. For each $p\in R$,   $\varinjlim\Grb_{\G,p}^\lambda\to \Gr_{\G,p}$ is an isomorphism, cf.\,\cite[Proposition 9.9]{PR08}. Repeatedly applying Theorem \ref{498274} at ramified points, we conclude that \eqref{972364} is an isomorphism.	
	
	When $n\geq 2$, consider the following ind-closed embedding 
	\begin{equation}\label{879239}
		\varinjlim\Grb_{\G,C^n}^{\la}\to \Gr_{\G,C^n}.
	\end{equation}
	For any  nontrivial partition $\xi=(I,J)$  of $[n]$, by induction and the factorization maps \eqref{447893} \eqref{468789}, the  restriction of \eqref{879239} to $C^n_\xi$ is an isomorphism. Note that $\bigcup_\xi{C^n_\xi}=C^n\setminus \Gamma^n\cdot \Delta$, where $\Gamma^n$ acts on the diagonal $\Delta$ via the obvious action. We get an isomorphism 
	\[\varinjlim\Grb_{\G,C^n\setminus \Gamma^n\cdot \Delta}^{\la}\to \Gr_{\G,C^n\setminus \Gamma^n\cdot \Delta}.\]
	Note that $\Gr_{\G,\vec\gamma( \Delta)}\simeq \Gr_{\G,C}$. Moreover, by Corollary \ref{711570}, $\Grb_{\G,\vec\gamma( \Delta)}^{\la}\simeq\Grb_{\G,C}^{\lambda}$ for any $\vec\gamma=(\gamma_1,\ldots,\gamma_n)$, $\la=(\lambda_1,\ldots,\lambda_n)$, and $\lambda=\sum \gamma_i^{-1}(\lambda_i)$. By the case when $n=1$, \eqref{972364} induces an isomorphism $\varinjlim\Grb_{\G,\vec\gamma( \Delta)}^{\la}\to \Gr_{\G, \vec\gamma(\Delta)}$ for any $\vec\gamma\in\Gamma^n$.	 Repeatedly applying Theorem \ref{498274} at $\vec{\gamma}(\Delta)$, we conclude that \eqref{879239} is an isomorphism.	
\end{proof}
\begin{remark}
All results in this subsection remain true for any algebraically closed field $k$ of characteristic $p$, if $p\not \mid m$.  Moreover, if $\Gamma$ is trivial or generated by a diagram automorphism, all results still hold.
\end{remark}

\section{Line bundles on Beilinson-Drinfeld Grassmannian of $\G$}\label{717393}
In this section, we determine the rigidified Picard group of the BD Grassamnnian $\Gr_{\G,C^n}$ of $\G$ when $\G$ is generically simply-connected, along the way we construct the level one line bundle  $\mc{L}_{C^n}$ on $\Gr_{\G,C^n}$. We also establish the factorizable and $L^+\G_{C^n}$-equivariant structure on $\mc{L}_{C^n}$. This allow us to construct factorizable and equivariant level one line bundles on BD Schubert varieties of $\G$  when $C=\mathbb{A}^1$ and $\G$ is of adjoint type generically.
\subsection{Picard group of rigidified line bundles on BD Grassmannian of $\G$}\label{201752}
 We first make a digression to prove a relative version of seesaw principal, which is originally due to Mumford \cite[Corollary 6]{Mum85}. 
\begin{lem}[Seesaw Theorem]\label{086885}
	Let $\phi:X\to S$ and $\psi:Y\to S'$ be projective and flat morphisms of varieties over $\C$ such that, for any $s\in S$ and $s'\in S'$, the schematic fibers $X_{s}$ and $Y_{s'}$ are integral. Given an open subset $U\subset S\times S'$, let $L$ be a line bundle on $(X\times Y)|_U$ such that the restriction $L|_{X_s\times q}$ and $L|_{p\times Y_{s'}}$ are trivial for any $(s,q)\in (S\times Y)|_U$, $(p,s')\in (X\times S')|_U$. Then there is a line bundle $M$ on $U$ such that $L\simeq\pi^*M$, where $\pi=(\phi\times\psi )|_U: (X\times Y)|_U\to U$.
\end{lem}
\begin{proof}
	Consider the morphism $\pi_2=(\phi\times \mr{id})|_U: (X\times Y)|_U\to (S\times Y)|_U$. By assumption, for any $(s,q)\in (S\times Y)|_U$, the line bundle $L|_{X_s\times q}$ is trivial. Since $X$ is projective over $S$, we have $\dim\big( H^0(X_s\times q, L|_{X_s\times q})\big)=1$ for any $(s,q)\in (S\times Y)|_U$. It follows that $(\pi_2)_*L$ is a locally free sheaf of rank $1$ over $(S\times Y)|_U$. Let $L'=(\pi_2)_*L$. By adjunction, we have a morphism $(\pi_2)^*L'\to L$. This is an isomorphism since the map $H^0(X_s\times q, L)\to L_{(p,q)}$ is an isomorphism for any $(p,q)\in (X\times Y)|_U$.

	Similarly, consider the morphism $\pi_1= (\mr{id}\times \psi)|_U: (S\times Y)|_U\to U$ and note that $L'|_{s\times Y_{s'}}$ is trivial for any $(s,s')\in U$. By the same argument as above, $(\pi_1)_*L'$ is a locally free sheaf of rank $1$ over $U$, and $L'\simeq (\pi_1)^*M$ where $M=(\pi_1)_*L'$. Therefore, 
	\begin{equation*}
		\pushQED{\qed} L\simeq (\pi_2)^*L'\simeq (\pi_2)^*(\pi_1)^*M\simeq \pi^*M. \qedhere 
	\end{equation*}
\end{proof}

Let $\pi: X\to S$ be a projective and flat morphism of varieties such that every fiber is integral. Given a section $e:S\to X$, let $(L,\alpha)$ denote a rigidified line bundle on $X$, that is, a line bundle $L$ with an isomorphism $\alpha: \ms{O}_S \simeq e^*L$. We denote by $\Pic^e(X)$ the Picard groupoid of these rigidified line bundles on $X$ with respect to $e$.
 \begin{lem}\label{841770}
	Let $\pi:X\to S$ be a morphism as above. Let $(L,\alpha)$, $(L',\alpha')$ be rigidified line bundles in $\Pic^e(X)$. If $\psi:L\simeq L'$ is an isomorphism of line bundles on $X$, then there exists a unique isomorphism $\phi:(L,\alpha)\simeq (L',\alpha')$ in $\Pic^e(X)$, i.e.\,there is a unique isomorphism $\phi:L\simeq L'$ such that the following diagram commutes
	\begin{equation}\label{355605}
	\xymatrix{
		e^*L \ar[rr]^{e^*(\phi)} & & e^*L' \\
		&\ms{O}_S \ar[lu]^{\alpha} \ar[ru]_{\alpha'}&
	}.
	\end{equation}
\end{lem}
\begin{proof}
By a similar argument as in the proof of Lemma \ref{086885}, there exists an isomorphism $\pi^*\pi_* L\simeq L$. Applying $e^*$, we get $\pi_*L\simeq e^*L$. Thus, 
\begin{equation}\label{eq_iso_L}
	{\rm Hom}(L,L)\simeq {\rm Hom}(\pi^*\pi_*L,L)\simeq {\rm Hom}(\pi_*L, \pi_*L)\simeq {\rm Hom}(e^*L,e^*L). 
\end{equation}
If there are two isomorphisms $\phi_1, \phi_2:(L,\alpha)\simeq (L',\alpha')$, then $\phi_2^{-1}\circ\phi_1$ is an automorphism of $(L,\alpha)\in \Pic^e(X)$. In view of the isomorphism (\ref{eq_iso_L}), this must be an identity. This proves the uniqueness. 
	
Given an isomorphism $\psi: L\simeq L'$,	set $\beta=e^*(\psi)\circ \alpha: \ms{O}_S\simeq e^*L'$, which gives another rigidification of $L'$. By the isomorphism (\ref{eq_iso_L}), there exists a unique $\eta\in \mr{Aut}(L')$ such that $e^*(\eta)=\alpha'\circ \beta^{-1}$. Then, the isomorphism $\phi:= \eta\circ \psi:L\to L'$ satisfies the commutative diagram \eqref{355605}. Hence, it gives an isomorphism $(L,\alpha)\simeq (L',\alpha')$ in $\Pic^e(X)$.
\end{proof}
By Lemma \ref{841770}, the groupoid $\Pic^e(X)$ is discrete. We will use this lemma to glue rigidified line bundles. Let $\mr{Pic}^e(X)$ denote the group of isomorphism classes in $\Pic^e(X)$, and call it the rigidified Picard group of $X$ along $e$. Given an isomorphism class $[(\mc{L},\alpha)]\in \mr{Pic}^e(X )$, by abuse of notation, we simply denote it by $\mc{L}$. 
\medskip

From now on, we are in the setup of Section \ref{029728}. We further assume that the simple algebraic group $G$ is simply-connected. 

Let $\mr{Pic}^e(\Gr_{\mc{G},C^n})$ be the rigidified Picard group of $\Gr_{\mc{G},C^n}$ along the section $e:C^n\to \Gr_{\mc{G},C^n}$. Here, we abuse the notation $e$ for the section $e_{\vec{0}}$ over $C^n$. Given a rigidified line bundle $\mc{L}\in \Gr_{\mc{G},C^n}$, we denote by $\mc{L}|_{\Delta_{C^n}}$ the restriction of $\mc{L}$ to the BD Grassmannian $\Gr_{\mc{G},\Delta_{C^n}}$ over the diagonal.
By \cite[Proposition 4.1]{Zhu14}, the central charge of $\mc{L}|_{\vec{p}}$ for any $\vec{p}\in \Delta_{C^n}$ is constant. This gives a well-defined central charge map 
\begin{equation}\label{295279}
	c:\mr{Pic}^e(\Gr_{\mc{G},C^n}) \to \mb{Z} .
\end{equation} 
We shall show the map $c$ is an isomorphism.

 We first deal with the case when $C=\A^1$ with the standard $\Gamma$-action in (\ref{eq_standard}). Recall that when $C=\mathbb{P}^1$, $\mr{Pic}(\mr{Bun}_{\mc{G}})\simeq \mb{Z}$, where the positive generator will be denoted by $\mc{L}$. Consider the projection $\mr{pr}:\Gr_{\mc{G},(\mb{P}^1)^n}\to \mr{Bun}_{\mc{G}}$. Let 
\begin{equation}\label{eq_linebun_L}
	\mc{L}_{\A^n}:=(\mr{pr}^*\mc{L})|_{\Gr_{\mc{G},\A^n}}
\end{equation}
denote the restriction of the line bundle $\mr{pr}^*\mc{L}$ to $\Gr_{\mc{G},\A^n}$. Clearly, $\mc{L}_{\A^n}$ belongs to $\mr{Pic}^e(\Gr_{\mc{G},\A^n})$ since $e^*(\mc{L}_{\A^n})$ is a trivial bundle over $\A^n$. Note that the restriction $\mc{L}_{\A^n}|_{\Delta_{\A^n}}$ to the diagonal, regarded as a line bundle over $ \Gr_{\mc{G},\A^1}$, is isomorphic to $\mc{L}_{\A^1}$. Moreover, the restriction of the line bundle $\mc{L}_{\A^1}$ to each fiber is the ample generator in $\mr{Pic}(\Gr_{\mc{G},p})$ for any $p\in \A^1$, cf.\,\cite[Corollary 3.14]{BH}. Hence $c(\mc{L}_{\A^n})=1$, which implies that $c$ is a surjection. From now on, we call $\mc{L}_{\A^n}$ the level one line bundle over $\Gr_{\G,\A^n}$.
\begin{thm}\label{429430}
	Let $C$ be an irreducible smooth curve with a faithful $\Gamma$-action and assume that all ramified points in $C$ are totally ramified. Then the central charge map \eqref{295279} is an isomorphism 
		\[c:\mr{Pic}^e(\Gr_{\mc{G}, C})\simeq \mb{Z}.\]
\end{thm}
\begin{proof}
	Let $R:=\{p_1,\ldots, p_k\}$ be the set of ramified points in $C$. Set $\mathring{C}=C\backslash R$. By Corollary \ref{prop_etale_base}, $\Gr_{\G, \mathring{C} }\simeq \Gr_{G, \mathring{C} }$. Then, by \cite[Lemma 3.4.2]{Zhu17}, the following central charge map is an isomorphism 
	\begin{equation}\label{781234}
		\mr{Pic}^e(\Gr_{\mc{G}, \mathring{C} })\simeq \mb{Z}.
	\end{equation}
	We consider the following restriction map
	\begin{equation} \label{423790}
		\mr{Pic}^e(\Gr_{\mc{G},C})\to \mr{Pic}^e(\Gr_{\mc{G},\mathring{C}}).
	\end{equation}
	Then, the composition of \eqref{423790} and \eqref{781234} is the central charge map 
	\begin{equation}\label{349194}
		c: \mr{Pic}^e(\Gr_{\mc{G},C})\to \mathbb{Z}.
	\end{equation}
	
	We first show the restriction map \eqref{423790} is injective. Let $\mc{L}\in \mr{Pic}^e(\Gr_{\mc{G},C})$ such that the restriction $\mc{L}|_{\Gr_{\mc{G},\mathring{C} }}$ is trivial. We shall show $\mc{L}$ is also trivial.	By Proposition \ref{423483}, it suffices to show the restriction of $\mc{L}$ to $\Grb^{\lambda}_{\mc{G}, C}$ is trivial for all dominant coweight $\lambda$. We still denote the restriction by $\mc{L}$. By assumption, there is an isomorphism $\theta: \ms{O}_{C}\simeq e^*\mc{L}$, and $\mc{L}|_{\Gr_{\mc{G}, \mathring{C}}}$ is trivial. By Lemma \ref{841770}, there is a unique isomorphism 
	\[s: \ms{O}_{\Grb_{\G,\mathring{C}}^\lambda} \simeq \mc{L}|_{\Grb_{\G,\mathring{C}}^\lambda}\]
	such that $e^*(s)=\theta|_{\mathring{C}}$. Let $D_i$ denote the divisor $\Grb_{\G,p_i}^\lambda$. By Theorem \ref{551564}, there is a unique integer $n$ such that 
	\[s \in \Gamma(\Grb_{\mc{G},C\setminus\{p_1,\ldots,p_{k-1}\}}^\lambda,\mc{L}(nD_k))\setminus \Gamma(\Grb_{\mc{G},C\setminus\{p_1,\ldots,p_{k-1}\}}^\lambda,\mc{L}((n-1)D_k)). \] 
	Now, we regard $s$ as a regular section of $\mc{L}(nD_k)$. Then $e^*(s)=\theta|_{\mathring{C}}$ can be regarded a section of $(e^*\mc{L})|_{\mathring{C}}$ with pole order $n$ at $p_k\in C$. Since $\theta$ is regular at $p_k$, we must have 
	$n=0$, and hence $s$ induces an isomorphism 
	$$s:\ms{O}_{\Grb_{\G,C\setminus \{p_1,\ldots, p_{k-1}\} }^\lambda} \simeq \mc{L}|_{\Grb_{\G,C\setminus \{p_1,\ldots,p_{k-1}\}}^\lambda}.$$ 
	Repeating this process, we get a trivialization of $\mc{L}$. It follows that the restriction map (\ref{423790}) is injective. Hence, the central charge map \eqref{349194} is injective.
	
	We now prove the surjectivity. For each $p_i\in R$, there exists a $\Gamma$-stable open neighborhood $U_i$ of $p$ with a $\Gamma$-equivariant \'etale map $U_i\to \A^1$, where $\A^1$ carries the standard $\Gamma$-action as in (\ref{eq_standard}). Observe that $p_i$ is the only ramified point in $U_i$. By Corollary \ref{prop_etale_base} (1), there is a natural morphism $\pi_i:\Gr_{\mc{G}_{\bar{U}_i},U_i}\to\Gr_{\mc{G}_{\bar\A^1}, \A^1} $. Let $\mc{L}_{U_i}$ be the pull-back of the level one line bundle $\mc{L}_{\A^1}$ via the map $\pi_i$. By the injectivity of (\ref{349194}), the central charge map $c: \mr{Pic}^e(\Gr_{\mc{G}, U_i})\to \mathbb{Z}$ is injective. Since $c(\mc{L}_{U_i})=1$,  we must have $c: \mr{Pic}^e(\Gr_{\mc{G}, U_i})\simeq \mathbb{Z}$. Note that $\{\mathring{C}, U_1,\ldots ,U_r\}$ forms an open covering of $C$. Moreover, by Lemma \ref{841770} there are unique isomorphisms between the restrictions of $\mc{L}_{\mathring{C}}$ and $\mc{L}_{U_i}$ on their intersections, since their restrictions are the level one generator in $\mr{Pic}^e(\Gr_{\mc{G},U_i\setminus\{p_i\}})$ or $\mr{Pic}^e(\Gr_{\mc{G},U_i \cap U_j})$ for any $i,j$. Therefore, they can be glued to be a line bundle of central charge one in $\mr{Pic}^e(\Gr_{\mc{G},C}) $, which will be denoted by $\mc{L}_C$. It follows that the central charge map $c: \mr{Pic}^e(\Gr_{\mc{G},C})\to \mathbb{Z}$ is an isomorphism.
\end{proof}
In the proof of the injectivity of the map \eqref{423790}, we used a similar argument as in \cite{Zhu17}[Lemma 4.3.3]. In fact, using the same argument we have the following more general result, which will be used later.
\begin{prop}\label{984820}
	Let C be an irreducible smooth curve with a faithful $\Gamma$-action and assume that all ramified points in $C$ are totally ramified. Let $Z$ be a closed subvariety in $C^n$. Let $D$ be an irreducible closed subvariety of codimension one in $C^n\backslash Z $. Then, the following restriction map is an embedding
	\begin{equation*}
		\mr{Pic}^e(\Gr_{\G,C^n\setminus Z })\to\mr{Pic}^e(\Gr_{\G,C^n\setminus (Z\cup D)}).\qedhere
	\end{equation*}
\end{prop}

We define the action of the symmetric group $S_n$ and $\Gamma^n$ on $C^n$ by $s\cdot(p_1,\ldots,p_n):=(p_{s(1)},\ldots,p_{s(n)})$ and $(\gamma_1,\ldots,\gamma_n)\cdot (p_1,\ldots,p_n):=(\gamma_1p_1,\ldots,\gamma_np_n)$, where $s\in S_n$, $\gamma_i\in \Gamma$. Let $Z$ be a  closed subvariety in $C^n$ that is stable under the actions of $S_n$ and $\Gamma^n$. Recall that for the finest partition $\xi=(\{1\},\ldots,\{n\})$ of $[n]$, there is a factorization 
\begin{equation}
	\Gr_{\G,C^n_\xi\setminus Z}\simeq (\Gr_{\G,C}\times\cdots\times \Gr_{\G,C})|_{C^n_\xi\setminus Z}.
\end{equation}
Let $\mc{L}_1,\ldots,\mc{L}_n\in \mr{Pic}^e(\Gr_{\G,C})$. We define a map
\begin{equation}\label{420430}
	\mr{Pic}^e(\Gr_{\G,C})\times \cdots\times \mr{Pic}^e(\Gr_{\G,C})\to \mr{Pic}^e(\Gr_{\G,C^{n}_\xi\setminus Z})
\end{equation}
by $(\mc{L}_1,\ldots,\mc{L}_n)\mapsto (\mc{L}_1\boxtimes\ldots\boxtimes\mc{L}_n)|_{\Gr_{\G,C^n_\xi\setminus Z}}$.
\begin{prop}\label{122349}
	Let C be an irreducible smooth curve with a faithful $\Gamma$-action and assume that all ramified points in $C$ are totally ramified. Let $n\geq 2$ and $Z$ be a closed subvariety in $C^n$ which is stable under the actions of $S_n$ and $\Gamma^n$. We further assume that the diagonal $\Delta_{C^n}$ is contained in $C^n\setminus Z$. Let $\xi$ be the finest partition of $[n]$. Then, 
	\begin{enumerate}
		\item The map \eqref{420430} is an isomorphism, whose inverse induces the following isomorphism
		\begin{equation}\label{432760}
			c_\xi:\mr{Pic}^e(\Gr_{\G,C_\xi^n\setminus Z})\simeq\mb{Z}^n
		\end{equation}
given by $\mc{L}\mapsto (c_1,\ldots,c_n)$, where $c_i=c\big(\mc{L}|_{(e_1,\ldots,e_{i-1})\times \Gr_{\G,q_i}\times(e_{i+1},\ldots,e_n)}\big)$ and $e_j$ is the base point in $\Gr_{\G,q_j}$ for any $(q_1,\ldots,q_n)\in C^n_\xi\setminus Z$. We will call $c_\xi$ the multi-central charge map. 
		\item The restriction map $\mr{Pic}^e(\Gr_{\G,C^n\setminus Z})\hookrightarrow\mr{Pic}^e(\Gr_{\G,C_\xi^n\setminus Z})$ is an embedding. 
	%	$$\iota:\mr{Pic}^e(\Gr_{\G,C^n\setminus Z})\hookrightarrow\mr{Pic}^e(\Gr_{\G,C_\xi^n\setminus Z})\simeq \mb{Z}^n.$$
		Moreover, for any $\mc{L}\in \mr{Pic}^e(\Gr_{\G,C^n\setminus Z})$, if  $c(\mc{L}|_{\Gr_{\G,\Delta_{C^n}}})=a$, then $c_\xi( \mc{L}|_{\Gr_{\G,C_\xi^n\setminus Z} }  )=(a,a,\ldots,a)$. As a consequence, the central charge map \eqref{295279} is injective.
	\end{enumerate} 
\end{prop}
\begin{proof}
	When n=2, $\xi=(\{1\},\{2\})$ is the finest partition. We fix an arbitrary point $(q_1,q_2)\in C_\xi^2\setminus Z$ and denote $C_{q_1}:=\{p\in C\mid (q_1,p)\in C^2_\xi\setminus Z\}$, $C_{q_2}:=\{p\in C\mid (p,q_2)\in C^2_\xi\setminus Z\}$. Clearly, $C_{q_i}$ is open in $C$. If $C_{q_1}$ is empty, then $(q_1,p)\not \in C^2_\xi\setminus Z$ for all $p\in C$. It follows that  $(q_1,p)\in Z$ for all $p\in C\setminus \Gamma\cdot \{p\}$.  Since $Z$ is closed, we conclude that $\{q_1\}\times C\subseteq Z$. Hence, $(q_1,q_2)\in Z$, which is a contradiction. Therefore, $C_{q_1}$ is non-empty. Similarly, $C_{q_2}$ is also non-empty. Let $\pi_1:C^2\setminus Z\to C$ be the map given by $(p_1,p_2)\mapsto p_1$. Then $\pi_1$ must be surjective. Otherwise, there is a point $p\in C$ such that $\pi_1^{-1}(p)=\big(\{p\}\times C\big)\setminus Z$ is empty. Then, $p\times C\subseteq Z$ contradicts with $\Delta_{C^2}\subseteq Z$.
	
	Let $e_{q_i}\in \Gr_{\G,q_i}$ be the base point. Consider the following map 
	\begin{equation}\label{428922}
		\mr{Pic}^e((\Gr_{\G,C}\times \Gr_{\G,C})|_{C_\xi^2\setminus Z})\to \mr{Pic}^e(\Gr_{\G,C_{q_2}})\times\mr{Pic}^e(\Gr_{\G,C_{q_1}}) 
	\end{equation}
	given by $\mc{L}\mapsto (\mc{L}_1,\mc{L}_2)$, where $\mc{L}_1:=\mc{L}|_{\Gr_{\G,C_{q_2}}\times e_{q_2}}$ and $\mc{L}_2:=\mc{L}|_{e_{q_1}\times \Gr_{\G,C_{q_1}}}$ are regarded as line bundles on $\Gr_{\G,C_{q_i}}$ respectively. We claim \eqref{428922} is an isomorphism.
	
	We first prove the injectivity. By Theorem \ref{429430}, we have  isomorphisms
	\begin{equation}\label{324910}
		\mr{Pic}^e(\Gr_{\G,C})\simeq \mr{Pic}^e(\Gr_{\G,C_{q_i}})\simeq \mb{Z}.
	\end{equation}
	For any $\mc{L}\in \mr{Pic}^e((\Gr_{\G,C}\times \Gr_{\G,C})|_{C_\xi^2\setminus Z}))$, by the isomorphism \eqref{324910}, $\mc{L}_1$ and $\mc{L}_2$ uniquely extend to rigidified line bundles on $\Gr_{\G,C}$, which will be still denoted by $\mc{L}_1$ and $\mc{L}_2$. Set 
	$$\mc{L}':=\mc{L}^{-1}\otimes(\mc{L}_1\boxtimes\mc{L}_2)|_{C_\xi^2\setminus Z}.$$
	By Lemma \ref{086885}, $\mc{L}'$ is the pullback of some line bundle $M$ on $C_\xi^2\setminus Z$ via the projection $\pi:(\Gr_{\G,C}\times\Gr_{\G,C})|_{C_\xi^2\setminus Z}\to C_\xi^2\setminus Z$. Note that $M\simeq e^*\pi^*M\simeq e^*\mc{L}'$ is trivial. Then, $\mc{L}'\simeq \pi^*M$ is trivial. Thus, $\mc{L}\simeq (\mc{L}_1\boxtimes\mc{L}_2)|_{C_\xi^2\setminus Z}$. This proves the injectivity of \eqref{428922}.
	
	Given any $(\mc{L}_1,\mc{L}_2)\in  \mr{Pic}^e(\Gr_{\G,C_{q_2}})\times\mr{Pic}^e(\Gr_{\G,C_{q_1}}) $, we still denote their extensions to $\Gr_{\G,C}$ by $(\mc{L}_1, \mc{L}_2)$. Clearly, the line bundle $(\mc{L}_1\boxtimes\mc{L}_2)|_{C_\xi^2\setminus Z}\in \mr{Pic}^e((\Gr_{\G,C}\times \Gr_{\G,C})|_{C_\xi^2\setminus Z})$ is mapping to $(\mc{L}_1,\mc{L}_2)$ via the map \eqref{428922}. This shows that \eqref{428922} is also surjective.
	
	In summary, we have an isomorphism 
	\begin{equation}\label{421434}
		\mr{Pic}^e(\Gr_{\G,C_\xi^2\setminus Z})\simeq  \mr{Pic}^e(\Gr_{\G,C})\times\mr{Pic}^e(\Gr_{\G,C}).
	\end{equation}
	From construction, this map is exactly the inverse map of \eqref{420430}, and it gives rise to the multi-central charge map \eqref{432760}. This completes the proof of part (1) for $n=2$. 
	
	\medskip
	By Proposition \ref{984820},  the following composition map is injective,
	$$\iota:\mr{Pic}^e(\Gr_{\G, C^2\setminus Z})\to \mr{Pic}^e(\Gr_{\G,C_\xi^2\setminus Z})\xrightarrow{c_\xi} \mb{Z}\times \mb{Z}.$$
	Given $\mc{L}\in \mr{Pic}^e(\Gr_{\G, C^2\setminus Z})$ such that $\iota(\mc{L})=(a,b)$, i.e.\,$c(\mc{L}|_{\Gr_{\G,C_{q_2}}\times e_{q_2}})=a$ and  $c(\mc{L}|_{ e_{q_1}\times \Gr_{\G,C_{q_1}}})=b$ for any $(q_1,q_2)\in C^2\backslash Z$. Suppose that
	the central charge of $\mc{L}|_{\Gr_{\G,(q_1,q_1)}}$ is $d$. We follow an argument in \cite[Lemma 3.4.3]{Zhu17} to show $b=d=a$.
	Consider the following convolution and projection morphisms $$\mf{m}: \mr{Conv}_{\G, C^2\setminus Z}\to \Gr_{\G, C^2\setminus Z}, \quad  \mr{pr}:\mr{Conv}_{\G, C^2\setminus Z}\to \Gr_{\G, C} .$$
	Note that $\mr{pr}^{-1}(e_{q_1})\simeq \Gr_{\G,C_1}$, where $C_1:=C_{q_1}\cup\Gamma\cdot q_1$. We regard $\mc{L}':=(\mf{m}^*\mc{L})|_{\mr{pr}^{-1}(e_q)}$ as a line bundle on $\Gr_{\G,C_1}$. We denote by $\phi$ the composition map $\mr{pr}^{-1}(e_{q_1})\cap \mr{Conv}_{\G, (q_1,q_1)} \simeq \Gr_{\G,q_1}\simeq \Gr_{\G,(q_1,q_1)}$. Then, $\phi$ is exactly the restriction of the local convolution map $\mf{m}_{(q_1,q_1)}:\Conv_{\G,(q_1,q_1)}\to \Gr_{\G,(q_1,q_1)}$.  Thus, the central charge of $\mc{L}'|_{q_1}\simeq \phi^*(\mc{L}|_{(q_1,q_1)})$ is equal to the central charge of $\mc{L}_{(q_1,q_1)}$ which is $d$.
	
	On the other hand, for any $p\in C_{q_1}$, we have an isomorphism $\mr{Conv}_{\G,(q_1,p)}\xrightarrow{\mf{m}_{(q_1,p)}} \Gr_{\G,(q_1,p)}\simeq \Gr_{\G,q_1}\times \Gr_{\G,p}$. This isomorphism restricts to an isomorphism $\psi:\mr{pr}^{-1}(e_{q_1})\cap \mr{Conv}_{\G, (q_1,p)} \simeq e_{q_1}\times \Gr_{\G,p}$. Thus, the central charge of $\mc{L}'|_p\simeq \psi^*(\mc{L}|_{e_{q_1}\times\Gr_{\G,p}})$ is $b$. Since $\mc{L}'$ is a line bundle on $\Gr_{\G,C_1}$, the central charge of $\mc{L}'|_x$ is constant along $x\in C_1$. It follows that $b=d$. Since, $\Gr_{\G, C^2\setminus Z}$ is $S_2$-symmetric, we also have $a=d$. This completes the proof for $n=2$.
	\medskip
	
	When $n\geq 3$, let $\xi$ be the finest partition of $[n]$. Fix an arbitrary point $(q_1,\ldots, q_n)\in C_\xi\setminus Z$. Denote by $e_i$ the base point in $\Gr_{\G,q_i}$. By induction on $n$ and using the same argument as in proving the isomorphism \eqref{421434}, one can prove that \eqref{420430} is an isomorphism and its inverse map gives rise to the multi-central charge map \eqref{432760}.
	Moreover, there is a commutative diagram
	\begin{equation}\label{432982}
		\xymatrix{
		\mr{Pic}^e(\Gr_{\G,C^n_\xi\setminus Z})\ar[r]^-{} \ar[d] \ar[rd]^-{c_\xi}& \mr{Pic}^e(\Gr_{\G,C^{n-1}_{\xi'}\setminus Z_n}\times e_n)\times \mr{Pic}^e(x\times\Gr_{\G,C\setminus Z_x}) \ar[d]^{c_{\xi'}\times c} \\
		\mr{Pic}^e(\Gr_{\G,C\setminus Z_y}\times y)\times \mr{Pic}^e(e_1\times\Gr_{\G,C^{n-1}_{\xi'}\setminus Z_1}) \ar[r]_-{c\times c_{\xi'}}& \mb{Z}^n
	}.
	\end{equation}
	where $\xi'$ is the minimal partition of $[n-1]$, $x=(e_1,\ldots, e_{n-1})$, $y=(e_2,\ldots, e_{n})$, $Z_x=\{\,  p\in C\,|\,(q_1,\ldots,q_{n-1},p)\not \in C^n_\xi\setminus Z \, \}$,  $Z_y=\{\, p\in C\,|\,(p,q_2,\ldots,q_n)\not \in C^n_\xi\setminus Z  \,\}$, $Z_n=\{\vec{p}\in C^{n-1}\,|\,(\vec{p},q_n)\not\in C^n_\xi\setminus Z\}$, and $Z_1=\{\vec{p}\in C^{n-1}\,|\,(q_1,\vec{p})\not\in C^n_\xi\setminus Z\}$.
	By Proposition \ref{984820},  the following composition map is injective,
	$$\iota:\mr{Pic}^e(\Gr_{\G, C^n\setminus Z})\to \mr{Pic}^e(\Gr_{\G,C^n_\xi\setminus Z})\xrightarrow{c_\xi} \mb{Z}^n.$$
	Let $\mc{L}\in \mr{Pic}^e(\Gr_{\G, C^2\setminus Z})$. Suppose $\iota(\mc{L})=(a_1,\ldots,a_n)$. By induction on $n$ and the commutativity of \eqref{432982}, we have $a_1=a_2=\cdots=a_n=c(\mc{L})$.
\end{proof}

\begin{lem}\label{423913}
	Let $X$ be an irreducible smooth $\Gamma$-curve consisting of  exactly one ramified point $p$. Suppose $f:X\to \A^1$ is a $\Gamma$-equivariant \'etale morphism sending $p$ to $0$, where the $\Gamma$-action on $\A^1$ is defined by \eqref{eq_standard}.  Let $n\geq2$ and  $Z=\{(x_1,\ldots,x_n)\in X^n \,|\, \bar x_i\neq\bar x_j, \overline{f(x_i)}=\overline{ f(x_j)} \text{ for some } i,j\} $. There is an isomorphism 
	$$c:\mr{Pic}^e(\Gr_{\G_{\bar X},X^n\setminus Z})\simeq\mb{Z}.$$
\end{lem}
\begin{proof}
	We first show that $Z$ is a closed subvariety in $X^n$. When $n=2$, $Z=\{(x_1,x_2)\in X^2 \,|\, \bar x_1\neq\bar x_2, \overline{f(x_1)}=\overline{ f(x_2)} \} $. Note that $Z= (X\times_{\mathbb{A}^1} X) \backslash \Delta_X$. Since $f$ is \'etale,  the diagonal morphism $\Delta: X\to X\times_{\mathbb{A}^1} X$ is an open embedding, cf.\,\cite[\href{https://stacks.math.columbia.edu/tag/06CR}{Tag 06CR}]{stacks-project} and \cite[\href{https://stacks.math.columbia.edu/tag/05W1}{Tag 05W1}]{stacks-project}.   Thus, $Z$ is a closed subvariety of $X\times X$.  
%	Note that the pre-image $Y:=(f\times f)^{-1}(\Delta_{\A^2})$ is a closed subscheme of $X^2$. Then, $Z=Y\setminus \Delta_{X^2}$ can be identified with $X\times_{\A^1} X\setminus X$, where $X\hookrightarrow X\times_{\A^1} X$ is the diagonal morphism. Since $f$ is \'etale, $X\hookrightarrow X\times_{\A^1} X$ is an open immersion. Thus, $X\times_{\A^1} X\setminus X$ is closed in $X\times_{\A^1} X$. It follows that $Z$ is closed in $Y$. Hence, $Z$ is closed in $X^2$. 
	When $n\geq 3$, let $Z_{i,j}=\{(x_1,\ldots,x_n)\in X^n \,|\, \bar x_i\neq\bar x_j, \overline{f(x_i)}=\overline{ f(x_j)}\}$. Then, $Z_{i,j}\simeq X^{n-2}\times Z'$ is closed in $X^n$, where $Z':=\{(x_1,x_2)\in X^2 \,|\, \bar x_1\neq\bar x_2, \overline{f(x_1)}=\overline{ f(x_2)} \}$.  Thus, 	$Z=\bigcup_{i,j} Z_{i,j}$ is closed in $X^n$.
	
	Given any $x\in X$, in view of Proposition \ref{122349}, the map  $\mr{Pic}^e(\Gr_{X^n\setminus Z})\hookrightarrow \mb{Z}$ sending $\mc{L}$ to the central charge of $\mc{L}|_{(x,\ldots,x)}$ is an embedding.
	On the other hand, by  part (1) of Corollary \ref{prop_etale_base}, we have an isomorphism
	$$\Gr_{\G_{\bar X}, X^n\setminus Z}\simeq \Gr_{\G_{\bar\A^1}, \A^n}\times_{\A^n}(X^n\setminus Z).$$
	Consider the natural morphism $\phi:\Gr_{\G_{\bar X}, X^n\setminus Z}\to \Gr_{\G_{\bar\A^1}, \A^n}$. Pulling back $\mc{L}_{\A^n}$ via $\phi$, we get a line bundle $\phi^*\mc{L}_{\A^n}$ on $\Gr_{\G_{\bar X}, X^n\setminus Z}$. Note that $(\phi^*\mc{L}_{\A^n})|_{(x,\ldots,x)}$ has the same central charge as $(\mc{L}_{\A^n})|_{(f(x),\ldots,f(x))}$, which is $1$. It follows that $\mr{Pic}^e(\Gr_{\G_{\bar X},X^n\setminus Z})\hookrightarrow \mb{Z}$ is also surjective.
\end{proof}
\begin{prop}\label{284344}
	\begin{enumerate}
		\item Let $X$ be an irreducible smooth curve. Let $G$ be a simply-connected simple algebraic group. Then, the central charge map $c:\mr{Pic}^e(\Gr_{G,X^n})\to \mb{Z}$ is an isomorphism.
		\item Let $X$ be an irreducible smooth curve with a faithful $\Gamma$-action and assume that all points in $X$ are unramified. Then, the central charge map $c:\mr{Pic}^e(\Gr_{\G,X^n})\to \mb{Z}$ is an isomorphism.
	\end{enumerate}
\end{prop}
\begin{proof}
	Part (1) follows from \cite[Corollary 3.4.4.]{Zhu17}. It can also be shown by applying   Proposition \ref{122349} to the case when $\Gamma=\{e\}$.

	For part (2), by Corollary \ref{prop_etale_base} (2), we have an isomorphism $\alpha:\Gr_{\G, X^n\setminus Z }\simeq \Gr_{G, X^n\setminus Z }$ for $Z=\{\vec{p}\in X^n\mid p_i=\gamma\cdot p_j \text{ for some } 1\leq i,j\leq n, \gamma\neq e\}$. By part (1), the central charge map $c:\mr{Pic}^e(\Gr_{G, X^n })\hookrightarrow \mr{Pic}^e(\Gr_{G, X^n\setminus Z }) \to \mb{Z}$ is an isomorphism. Let $L_{X^n\setminus Z}\in \mr{Pic}^e(\Gr_{G,X^n\setminus Z})$ denote the restriction of the level one line bundle $L_{X^n}$ constructed in part (1). Pulling back $L_{X^n\setminus Z}$ along $\alpha$, we get a rigidified line bundle on $\Gr_{\G,X^n\setminus Z}$, denoted by $\mc{L}_{X^n\setminus Z}$. Let $\vec\gamma\in \Gamma^n$.  Consider the isomorphism $\vec\gamma:X^n\to X^n$ given by $(p_1,\ldots,p_n)\mapsto (\gamma_1p_1,\ldots,\gamma_np_n)$. It induces an isomorphism $\vec\gamma:\Gr_{\G,X^n\setminus Z}\to \Gr_{\G,X^n\setminus \vec\gamma(Z)}$. Pushing forward $\mc{L}_{X^n\setminus Z}$ along $\vec\gamma$, we get a rigidified line bundle $\mc{L}_{X^n\setminus \vec\gamma(Z)}\in \mr{Pic}^e(\Gr_{\G,X^n\setminus \vec\gamma(Z)})$. 
	
	Note that $\vec\gamma(Z)=\{\vec{p}\in X^n\mid \text{there exist } 1\leq i,j\leq n \text{ such that } p_i=\gamma\cdot p_j \text{ for some }  \gamma\neq \gamma_i\gamma_j^{-1}\}$. We have $\bigcap_{\vec\gamma} \vec\gamma(Z)=\emptyset$. Hence, $\{X^n\setminus \vec\gamma(Z)\}_{\vec\gamma}$ is a open cover of $X^n$. By Propostition \ref{984820}, the restiction map $\mr{Pic}^e(\Gr_{\G,X^n\setminus (\vec\gamma(Z)\cup\vec\gamma'(Z))})\to  \mr{Pic}^e(\Gr_{\G,X^n_{\xi}})$ is injective, where $\xi$ is the finest partition of $[n]$. By Proposition \ref{122349}, this induces an embedding 
	\begin{equation}\label{932484}
		\mr{Pic}^e(\Gr_{\G,X^n\setminus (\vec\gamma(Z)\cup\vec\gamma'(Z))})\hookrightarrow \mb{Z}^n
	\end{equation}
	for any $\vec\gamma,\vec\gamma'\in \Gamma^n$. From construction, the images of the restrictions of line bundles $\mc{L}_{X^n\setminus \vec\gamma(Z)}$ and $\mc{L}_{X^n\setminus \vec\gamma(Z)}$ to $\Gr_{\G,X^n\setminus (\vec\gamma(Z)\cup\vec\gamma'(Z))}$ under the map \eqref{932484} are $(1,1,\ldots,1)$. By Lemma \ref{841770}, we conclude that the restrictions of $\mc{L}_{X^n\setminus \vec\gamma(Z)}$ and $\mc{L}_{X^n\setminus \vec\gamma'(Z)}$ to the intersection $\Gr_{\G,X^n\setminus (\vec\gamma(Z)\cup\vec\gamma'(Z))}$ are isomorphic. Hence, we can glue these line bundles $\{\mc{L}_{X^n\setminus \vec\gamma(Z)}\}_{\vec\gamma}$. This produces a line bundle on $\Gr_{\G, X^n}$, denoted by $\mc{L}_{X^n}$. From construction, $c(\mc{L}_{X^n})=1$. Therefore, $c$ is an isomorphism.	
\end{proof}
\begin{thm}\label{797528}
	Let $C$ be an irreducible smooth curve with a faithful $\Gamma$-action and assume that all ramified points in $C$ are totally ramified.
	\begin{enumerate}
		\item For $n\geq 1$, the central charge map \eqref{295279} is an isomorphism 
		\[c:\mr{Pic}^e(\Gr_{\mc{G}, C^n})\simeq \mb{Z}.\]
		\item For $n\geq 2$, let $\xi=(I,J)$ be a partition of $[n]$. Let $\mc{L}$, $\mc{L}_I$, $\mc{L}_J$ be rigidified line bundles over $\Gr_{\G,C^n}$, $\Gr_{\G,C^I}$, $\Gr_{\G,C^J}$ respectively. Suppose that they have the same central charge. Then, we have an isomorphism of line bundles on $\Gr_{\mc{G}, C^n_\xi}$
		\[\mc{L}|_{C^n_{\xi}}\simeq (\mc{L}_{I}\boxtimes\mc{L}_{J})|_{C^n_{\xi}}.\]
	\end{enumerate}
\end{thm}

\begin{proof}
	The case when $n=1$ has been proved in Theorem \ref{429430}. 
	
	For $n\geq 2$, the central charge map $c$ is injective by Proposition \ref{122349}. 	Let $\xi=(I_1,\ldots,I_k)$ be a nontrivial partition. Since $|I_j|<n$, by induction on $n$, we have an isomorphism $c:\mr{Pic}^e(\Gr_{\G,C^{I_j}})\simeq \mb{Z}$. Let $\mc{L}_{C^{I_j}}$ be the rigidified line bundle on $\Gr_{\G,C^{I_j}}$ such that $c(\mc{L}_{C^{I_j}})=1$. Recall that we have a factorization $\Gr_{\G,C^n_\xi\setminus Z}\simeq (\Gr_{\G,C^{I_1}}\times\cdots\times \Gr_{\G,C^{I_k}})|_{C^n_\xi\setminus Z}$. Set $\mc{L}_{C^n_\xi}:=(\mc{L}_{C^{I_1}}\boxtimes\cdots\boxtimes \mc{L}_{C^{I_k}})|_{\Gr_{\G,C^n_\xi}}$. 
	
	Let $R:=\{p_1,\ldots, p_k\}$ be the set of ramified points in $C$. Set $\mathring{C}=C\backslash R$. By Proposition \ref{284344}, $\mr{Pic}^e(\Gr_{\G,(\mathring{C})^n})\simeq \mb{Z}$. Let $\mc{L}_{\mathring{C}^n}$ denote the level one line bundle on $\Gr_{\G,(\mathring{C})^n}$ constructed in Proposition \ref{284344}. On the other hand, for each $p_i$, there is a neighborhood $U_i$ of $p_i$ with a $\Gamma$-equivariant  \'etale morphism $f_i:U_i\to \A^n$. Let 
	\begin{equation}\label{773922}
		Z_i:=\{(x_1,\ldots,x_n)\in U_i^n \,|\,\bar x_a\neq \bar x_b, \overline{f_i(x_a)}= \overline{f_i(x_b)} \text{ for some } 1\leq a<b\leq n\} 
	\end{equation}
	be a closed subvariety in $U_i^n$. By Lemma \ref{423913}, pulling back $\mc{L}_{\A^n}$ via the natural map $\Gr_{\mc{G}_{\bar{U}_i},(U_i)^n\setminus Z_i}\to\Gr_{\mc{G}_{\bar\A^1}, \A^n}$, we get a line bundle $\mc{L}_{(U_i)^n\setminus Z_i}\in \mr{Pic}^e(\Gr_{(U_i)^n\setminus Z_i})$ whose central charge is $1$. Note that $\big\{ C^n_\xi, (U_i)^n\setminus Z_i, (\mathring{C})^n\big\}_{\xi, i}$ is an open covering of $C^n$, and pairwise intersections are of the form $X^n_\xi$, $X^n_\xi\setminus Z_X$, or $X^n\setminus Z_X$ for some $\Gamma$-stable open subset $X$ in $C$ and some closed subset $Z_X$ in $X^n$. In view of Proposition \ref{122349} and Lemma \ref{423913}, the restrictions of rigidified line bundles $\{\mc{L}_{C^n_\xi}, \mc{L}_{(U_i)^n\setminus Z_i}, \mc{L}_{\mathring{C}^n}\}_{\xi,i,}$ are isomorphic on the pairwise intersections. Thus, we can glue these line bundles and get a factorizable rigidified line bundle $\mc{L}_{C^n}$ whose central charge is $1$. Hence, $c$ is also surjective. Part (2) follows from the fact that $\mc{L}_{C^n}|_{\Gr_{\G,C^n_\xi}}\simeq (\mc{L}_{C^{I_1}}\boxtimes\cdots\boxtimes \mc{L}_{C^{I_k}})|_{\Gr_{\G,C^n_\xi}}$.
\end{proof}
We remark that the proof of Theorem \ref{797528} also works when $\Gamma$ is a trivial group. We call the line bundle $\mc{L}_{C^n}$ constructed in Theorem \ref{797528} the level one line bundle on $\Gr_{\G,C^n}$. 

\subsection{$L^+\mc{G}_{C^n}$-equivariant structure on the line bundle $\mc{L}_{C^n}$} 
 In this subsection, we will show that when $G$ is simply-connected, there is a unique $L^+\mc{G}_{C^n}$-equivariant structure on the level one line bundle $\mc{L}_{C^n}$, where $L^+\mc{G}_{C^n}$ is the jet group scheme defined in Definition \ref{def_BD_jet}.

Let $\mc{L}$ be a line bundle on $\Gr_{\mc{G}, C^n}$ (here we think of line bundles from the point of view as in Appendix \ref{appendC}). An $L^+\mc{G}_{C^n}$-equivariant structure on $\mc{L}$ is a collection of isomorphisms $\{\phi_{g,x}:\mc{L}_x\simeq \mc{L}_{gx} \mid g\in L^+\mc{G}_{C^n}(S), x\in \Gr_{\mc{G}, {C}^n}(S) \}$ for any $\mathbb{C}$-scheme $S$,  satisfying the cocycle conditions $\phi_{g',gx}\circ\phi_{g,x}=\phi_{g'g, x}$. We define a group scheme
\begin{equation}
	\label{5.2_ext}
	\widehat{L^+\mc{G}}_{C^n}(S):=\left\{ \big(g,\{\phi_{g,x}\} \big) \,\middle|\, g\in L^+\mc{G}_{C^n}, \phi_{g,x}: \mc{L}_x\simeq \mc{L}_{gx}, x\in \Gr_{\mc{G}, {C}^n}(S)\right\}, 
\end{equation}
where the multiplication is defined by $(g',\{\phi_{g',x}\}) \cdot (g,\{\psi_{g,x}\})=(g'g, \{\phi_{g',gx}\circ \psi_{g,x}\}) $. One can easily check the associativity. By definition, $\mc{L}$ has an $L^+\mc{G}_{C^n}$-equivariant structure if and only if the following central extension has a splitting $\beta:L^+\mc{G}_{C^n} \to \widehat{L^+\mc{G}}_{C^n} $
\begin{equation}\label{899726}
	K\to \widehat{L^+\mc{G}}_{C^n}\xrightarrow{\mr{pr}} L^+\mc{G}_{C^n},
\end{equation}
where $K=\{\phi\mid \phi:\mc{L}\simeq\mc{L}\} $ is a group scheme over $C^n$, and $\mr{pr}: \widehat{L^+\mc{G}}_{C^n}\to L^+\mc{G}_{C^n}$ is the projection map. Conversely, if the extension \eqref{899726} has a splitting $\beta:L^+\mc{G}_{C^n} \to \widehat{L^+\mc{G}}_{C^n} $, say $\beta(g)=(g,\{\phi_{g,x}\})$, then it gives a collection of isomorphisms $\{\phi_{g,x}:\mc{L}_x\simeq \mc{L}_{gx}\mid g\in L^+\mc{G}_{C^n}(S), x\in \Gr_{\mc{G}, {C}^n}(S) \}$. Since $\beta$ is a group homomorphism, we have $(g',\{\phi_{g',x}\}) \cdot (g,\{\phi_{g,x}\})=(g'g, \{\phi_{g'g,x}\}) $. Hence $\phi_{g',gx}\circ \phi_{g,x}=\phi_{g'g, x}$, which implies that $\mc{L}$ is $L^+\mc{G}_{C^n}$-equivariant. Therefore, the extension \eqref{899726} has a splitting if and only if $\mc{L}$ is $L^+\mc{G}_{C^n}$-equivariant.
\begin{lem}\label{770893}
	If the extension \eqref{899726} has a splitting $\beta:L^+\mc{G}_{C^n} \to \widehat{L^+\mc{G}}_{C^n} $, then it is unique.
\end{lem}
\begin{proof}
	Let $\beta_1$, $\beta_2$ be two splittings, which give a group homomorphism $\phi:L^+\mc{G}_{C^n}\to K$ sending $g\in L^+\mc{G}_{C^n}$ to $\beta_1(g)\cdot\beta_2(g)^{-1}$. For any $\vec{p}\in C^n$, we have an isomorphism $L^+\mc{G}_{\vec{p}}\simeq \prod_{j=1}^k L^+\mc{G}_{q_j}$, where $\bigsqcup_j \Gamma\cdot q_j=\bigcup_i \Gamma\cdot p_i$. Over the fiber at $\vec{p}$, the morphism $\phi_{\vec{p}}:L^+\mc{G}_{\vec{p}}\to K_{\vec{p}}$ is trivial since the character of each $L^+\mc{G}_{q_j}$ is trivial, (when $q_j$ is unramified, the character of $L^+\mc{G}_{q_j}\simeq G(\mc{O})$ is trivial, cf.\,\cite[Corollary 9.1.3]{Sorger99}; when $q_j$ is ramified, there is a quotient map $L^+\mc{G}_{q_j}\to G^\sigma$ whose kernel is pro-unipotent, hence the character of $L^+\mc{G}_{q_j}$ is trivial since $G^\sigma$ is simple, cf.\,\cite[Table 2.3]{BH}.) Thus $\phi$ is trivial.
\end{proof}
From now on, we take $C$ to be $\mb{A}^1$ with the standard $\Gamma$-action as in (\ref{eq_standard}).
\begin{prop}\label{882893}
	Let $\mc{L}$ be a line bundle over $\Gr_{\mc{G}, \A^n}$. Given an integer $l>0$, if the $l$-th power $\mc{L}^l$ of the line bundle $\mc{L}$ has an $L^+\mc{G}_{\A^n}$-equivariant structure, then  so does $\mc{L}$.
\end{prop}

Before proving this proposition, we first make some preparations. By Proposition \ref{423483}, to show a line bundle $\mc{L}$ on $\Gr_{\mc{G},\A^n}$ has an $L^+\mc{G}_{\A^n}$-equivariant structure, it suffices to consider the restriction of $\mc{L}$ to each BD Schubert variety $\Grb^{\la}_{\mc{G},\A^n}$.

For any integer $k$, we define a group scheme $L_k^+\mc{G}_{\bar{\A}^n}$ as follows, which is of finite type over $\mathbb{C}$, for each $\C$-algebra $R$, 
\begin{equation}\label{eq_jet}
L_k^+\mc{G}_{\bar{\A}^n}(R):=G\bigg(\frac{R[t,z_1,\ldots,z_n]}{\big(\prod_{i=1}^n \prod_{\gamma\in \Gamma}(t- \gamma\cdot z_i)^k\big)}\bigg)^\Gamma.\end{equation}
Set $L^+_k\G_{\A^n}=L^+_k\G_{\bar{\A}^n}\times_{\bar{\A}^n }\A^n $. Clearly, $L^+\mc{G}_{\A^n}=\varprojlim L_k^+\mc{G}_{\A^n} $. When $\xi$ is the finest partition of $[n]$, we have a factorization
\[L_k^+\mc{G}_{\A_\xi^n}\simeq \big(L_k^+\mc{G}_{\A^1}\times\cdots\times L_k^+\mc{G}_{\A^1}\big)|_{\A^n_\xi}.\]
\begin{lem}\label{223442}
	The action of $L^+\mc{G}_{\A^n}$ on $\Grb^{\la}_{\mc{G},\A^n}$ factors through the action of $L^+_k\mc{G}_{\A^n}$ for sufficiently large integer $k$.
\end{lem}
\begin{proof}
	Let $\mathring{\A}=\A^1\setminus \{0\}$. When $n=1$, we have $\Grb^{\lambda}_{\mc{G},\mathring{\A}^1}\simeq \Grb^\lambda_{G}\times \mathring{\A}^1$ and $L^+\mc{G}_{\mathring{\A}^1}\simeq G_{\mc{O}}\times \mathring{\A}^1$, cf.\,\cite[Lemma 3.4.2]{Zhu17}. For any integer $k$, we have $L_k^+\mc{G}_{\mathring{\A}^1}\simeq G_{\mc{O}_k}\times\mathring{\A}^1$, where $\mc{O}_k:=\C[t]/(t^k)$. It is well-known that the $G_\mc{O}$-action on $\Grb_G^\lambda$ factors through $G_{\mc{O}_k}$ for some $k>0$. Thus, the $L^+\mc{G}_{\mathring{\A}^1}$-action on the Schubert variety $\Grb^\lambda_{\mc{G},\mathring{\A}^1}$ factors through the action of $L^+_{k}\mc{G}_{\mathring{\A}^1}$ for this $k$. 
	Denote by $K_{\A^1}$ the kernel of the morphism $L^+\mc{G}_{\A^1}\to L^+_{k}\mc{G}_{{\A^1}}$. Then it suffices to show $K_{\A^1}$ acts trivially on the Schubert variety $\Grb^\lambda_{\mc{G},{\A}^1}$. This is equivalent to that the composition of morphisms $K_{\A^1}\simeq K_{\A^1}\times_{\A^1} {\A^1} \xrightarrow{s^{\lambda}} K_{\A^1}\times _{\A^1}\Gr_{\mc{G},{\A^1}}\xrightarrow{\mr{act}} \Gr_{\mc{G},{\A^1}}$ agrees with the constant morphism $K_{\A^1} \to {\A^1}\xrightarrow{s^\lambda} \Gr_{\mc{G},{\A^1}}$. This is true over $\mathring{\A}^1$. Hence they must agree over ${\A^1}$, since $K_{\A^1}$ is irreducible.

	When $n\geq2$, let $\xi=(\{1\},\ldots,\{n\})$ be the finest partition of $[n]$. We have 
	\[\Grb^{\la}_{\mc{G},{\A}^n_\xi}\simeq \Big(\Grb^{\lambda_1}_{\mc{G},\A^1}\times\dots\times \Grb^{\lambda_n}_{\mc{G},\A^1}\Big)\big|_{\A^n_\xi}.\] 
	Thus, the $L^+\mc{G}_{\A_\xi^n}$-action on the Schubert variety $\Grb^\lambda_{\mc{G},{\A}_\xi^n}$ factors through the action of $L^+_k\mc{G}_{\A_\xi^n}$ for some $k>0$. By the same argument as before, for this $k$, the $L^+\mc{G}_{\A^n}$-action on the Schubert variety $\Grb^\lambda_{\mc{G}, {\A}^n}$ factors through the action of $L^+_k\mc{G}_{\A^n}$.
\end{proof}
\begin{lem}\label{109830}
	Let $X$ and $Y$ be smooth schemes over $S$, where $X,Y$ and $S$ are all schemes of finite type over $\mathbb{C}$. Suppose that $f:X\to Y$ is a morphism over $S$ such that $f_s:X_s\to Y_s$ is an isomorphism on the fiber for each $s\in S$, then $f:X\to Y$ is an isomorphism.
\end{lem}
\begin{proof}
	The proof is taken from \cite{Dracula}. By assumption $X$ and $Y$ are flat over $S$ and $f_s$ is an isomorphism for each $s\in S$, then  $f$ is a bijection, unramified, flat morphism. In particular, $f$ is \'etale and injective on the $\C$ points of $X$ and $Y$. It follows that $f$ is an open immersion and also a bijection. Hence $f$ is an isomorphism.
\end{proof}
\begin{lem}\label{683027}
 Let $\phi: G\to H$ be a morphism of connected affine algebraic groups over $\mathbb{C}$, such that the induced morphism $d\phi: \mr{Lie}\,(G)\to \mr{Lie}\,(H)$ of their Lie algebras is an isomorphism. If $H$ is simply-connected, then $\phi$ is an isomorphism. 
\end{lem}
\begin{proof}
	cf.\,\cite[Exercise 5.3.5]{Springer98}
\end{proof}

\begin{proof}[\textbf{Proof of Proposition \ref{882893}}]
	It suffices to prove this proposition for each variety $\Grb^{\la}_{\mc{G},\A^n}$. By Lemma \ref{223442}, the $L^+\mc{G}_{\A^n}$-action factors through the action of the group scheme $L_k^+\mc{G}_{\A^n}$ for some $k$ (depending on $\la$). We still denote by $\mc{L}$ its restriction to $\Grb^{\la}_{\mc{G},\A^n}$.

	By assumption, $\mc{L}^l$ has an $L_k^+\mc{G}_{\A^n}$-equivariant structure on $\Grb^{\la}_{\mc{G},\A^n}$, i.e.\,there are isomorphisms $\Phi_{g,x}:\mc{L}^l_{x}\simeq \mc{L}^l_{gx}$ such that $\Phi_{g',gx}\circ\Phi_{g,x}=\Phi_{g'g, x}$ for any $g, g'\in L_k^+\mc{G}_{\A^n}$, $x\in \Grb^{\la}_{\mc{G},\A^n}$.
	We define a groups scheme $\overline{L_k^+\mc{G}}_{\A^n}$ as follows, for each $\C$-algebra $R$,
	\begin{equation}\label{653942}
		\overline{L_k^+\mc{G}}_{\A^n}(R):=\left\{ \big(g,\{\phi_{g,x}\} \big) \,\middle|\,
		\begin{aligned}
		&g\in L_k^+\mc{G}_{\A^n}(R), \phi_{g,x}: \mc{L}_x\simeq \mc{L}_{gx} \text{ such that } \\
		&\phi_{g,x}^l=\Phi_{g,x} \text{ for } x\in \Grb^{\la}_{\mc{G},\A^n}(R)
		\end{aligned}
		\right\}.
	\end{equation}
	There is a central extension
	\begin{equation}\label{001553}
		1_{\A^n}\to Z_l \to \overline{L_k^+\mc{G}}_{\A^n}\xrightarrow{p} {L_k^+\mc{G}}_{\A^n}\to 1_{\A^n},
	\end{equation}
	where $Z_l=\{\phi:\mc{L}\to\mc{L}\mid \phi^l=\mr{id}:\mc{L}^l\to \mc{L}^l\}$ is a finite group scheme over $\A^n$. To prove our proposition, it suffices to show that this central extension splits.

	When $n=1$, set $\mathring{\A}^1=\A^1\setminus\{0\}$. There is a splitting over $\mathring{\A}^1$
	\[\xymatrix{
		1\ar[r] &Z_l|_{\mathring{\A}^1} \ar[r] & \overline{L_k^+\mc{G}}_{\mathring{\A}^1} \ar[r]_{p} &{L_k^+\mc{G}}_{\mathring{\A}^1} \ar@<0.5ex>@/_1pc/[l]_{\beta} \ar[r] &1
	},\]
	since $\Grb^{\lambda}_{\mc{G},\mathring{\A}^1}\simeq \Grb^\lambda_{G}\times \mathring{\A}^1$ and $L^+\mc{G}_{\mathring{\A}^1}\simeq G_{\mc{O}}\times \mathring{\A}^1$. 
	Let $L_k^+\mc{G}_{\A^1}^\dag$ be the schematic image of $\beta:L_k^+\mc{G}_{\mathring{\A}^1}\to \overline{L_k^+\mc{G}}_{\A^1}$ (we remove the non-identity components over $0$ if they occur). By \cite[Proposition 9.8]{Hartshorne77}, $L_k^+\mc{G}^\dag_{\A^1}$ is a flat group scheme over ${\A^1}$. By Cartier's Theorem and \cite[Theorem 10.2]{Hartshorne77}, $L_k^+\mc{G}^\dag_{\A^1}$ is smooth over ${\A^1}$. We claim that the natural morphism $p: L_k^+\mc{G}^\dag_{\A^1}\rightarrow L_k^+\mc{G}_{\A^1}$ is an isomorphism. By Lemma \ref{109830}, it suffices to check the isomorphism $L_k^+\mc{G}^\dagger_{x} \simeq L_k^+\mc{G}_{x}$ of fibers over any $x\in {\A^1}$. It is true over $\mathring{\A}^1$, since there is splitting $\beta$ over $\mathring{\A}^1$. When $x=0$, we have 
	\begin{equation}\label{eq_fiber_0}
		1\to \mu_l \to \overline{L_k^+\mc{G}}_{0}\to {L_k^+\mc{G}}_{0}\to 1 ,
	\end{equation}
	where $\mu_\ell$ denote the group of $\ell$-th root of unity. Since $L_k^+\mc{G}^\dag_{\A^1}$ and $L_k^+\mc{G}_{\A^1}$ are both smooth group schemes over ${\A^1}$ and isomorphic over $\mathring{\A}^1$, their fibers at $0$ have the same dimension. By (\ref{eq_fiber_0}), $\mr{Lie}(\overline{L_k^+\mc{G}}_{0})\simeq \mr{Lie}({L_k^+\mc{G}}_{0})$. By dimension consideration, the injective map 
	\[ \mr{Lie}(L_k^+\mc{G}^\dag_{0})\hookrightarrow \mr{Lie}(\overline{L_k^+\mc{G}}_{0})\simeq \mr{Lie}(L_k^+\mc{G}_{0})\] 
	must be an isomorphism. Moreover, the kernel of the quotient map $L_k^+\mc{G}_{0}\to G^\sigma$ is unipotent. By \cite[Lemma 6.1]{HK22}, $G^\sigma$ is simply-connected. Thus, $L_k^+\mc{G}_{0}$ is also simply-connected. In view of Lemma \ref{683027}, the morphism $L_k^+\mc{G}^\dag_{0}\to L_k^+\mc{G}_{0}$ is an isomorphism. Finally, by Lemma \ref{109830}, the morphism $p:L_k^+\mc{G}^\dag_{\A^1}\to L_k^+\mc{G}_{\A^1}$ is also an isomorphism. This gives a splitting of the central extension \eqref{001553}.
	
	\medskip
	
	When $n= 2$, there exists an isomorphism 
	\begin{equation}\label{329962}
		\Gr_{\mc{G},{\A^1}\times \mathring{\A}^1}\simeq \Gr_{\mc{G}, {\A^1}\times \{1\}}\times \mathring{\A}^1
	\end{equation}
	given by $(a,b, \mc{F}, \beta)\mapsto \big( (b^{-1}a,1, (\bar\rho_{b^{-1}})_*\mc{F}, (\bar\rho_{b^{-1}})_* \beta), b\big)$, where $\bar\rho_{b^{-1}}:\bar {\A}^1\to \bar{\A}^1$ is given by $\bar{c}\mapsto \overline{b^{-1} c}$. This isomorphism is $\mathbb{G}_m$-equivariant, where on the left the action is induced from the simultaneous dilation on ${\A^1}\times \mathring{\A}^1$ and on the right the action is induced from the dilation on $\mathring{\A}^1$. 
	
	Let $U={\A^1}\setminus \{1,\sigma (1),\ldots, \sigma^{m-1}(1)\}$. By the factorization property, we have $\Grb^{\lambda_1,\lambda_2}_{\mc{G}, U\times \{1\}}\simeq \Grb^{\lambda_1}_{\mc{G},U}\times \Grb_G^{\lambda_2}$. Thus, we have a splitting over $U\times\{1\}$,
	\begin{equation}\label{splitting_C'}
	 \xymatrix{
		1\ar[r] &Z_l|_{U} \ar[r] & \overline{L_k^+\mc{G}}_{U\times\{1\}} \ar[r]_{p} &{L_k^+\mc{G}}_{U\times\{1\}} \ar@<0.5ex>@/_1pc/[l]_{\beta} \ar[r] &1 }.
	\end{equation}
	By the same argument as in the case when $n=1$, we can extend the splitting $\beta$ to a splitting $\alpha$ over $\A^1\times \{1\}$.	Via the isomorphism \eqref{329962}, the $\mathbb{G}_m$-dilation produces a splitting $\alpha_1: L_{k}^+\mc{G}_{\A^1\times\mathring{\A}^1} \to \overline{L_{k}^+\mc{G}}_{\A^1\times \mathring{\A}^1}$ from $\alpha$. By a similar argument, there exists a splitting $\alpha_2: L_k^+\mc{G}_{\mathring{\A}^1\times \A^1} \to \overline{L_k^+\mc{G}}_{ \mathring{\A}^1\times \A^1}$. By the uniqueness of splitting, $\alpha_1$ and $\alpha_2$ agree on $\mathring{\A}^1\times \mathring{\A}^1$. Thus, we get a splitting over $(\A^1\times \mathring{\A}^1)\cup (\mathring{\A}^1\times \A^1)$. By Hartogs' Lemma, this splitting extends to $\A^1\times \A^1$. 

	\medskip
	
	When $n>2$, we prove it by induction on $n$. For any nontrivial partition $\xi=(I_1,\ldots, I_j)$ of $[n]$, we have factorizations $\Gr_{\mc{G},\A_{\xi}^n}\simeq (\Gr_{\mc{G},\A^{I_1}}\times\cdots\times \Gr_{\mc{G},\A^{I_j}})|_{\A_{\xi}^n}$ and $L_k^+\G_{\A_{\xi}^n}\simeq (L^+_k\G_{\A^{I_1}}\times\cdots\times L^+_k\G_{\A^{I_j}} )|_{\A_{\xi}^n}$. By induction, there is a splitting 
	\[\xymatrix{
		1\ar[r] & Z_l|_{\A_{\xi}^n} \ar[r] & \overline{L_k^+\mc{G}}_{\A_{\xi}^n} \ar[r]_{p} & {L_k^+\mc{G}}_{\A_{\xi}^n} \ar@<0.5ex>@/_1pc/[l]_{\beta_{\xi}} \ar[r] &1
	}.\]
	Moreover, $\beta_\xi$ and $\beta_{\xi'}$ agree on $\A^n_{\xi}\cap \A^n_{\xi'}$ by the uniqueness of splitting. Hence, we can glue these morphisms $\{\beta_\xi\}_\xi$ and obtain a splitting $\beta: L_k^+\mc{G}_{\bigcup \A_{\xi}^n} \to \overline{L_k^+\mc{G}}_{ \bigcup \A_{\xi}^n}$ over $\bigcup_\xi \A_{\xi}^n$. Note that the complement of $\bigcup_{\xi} \A^n_{\xi}$ is $\Delta_{[n]}$, which has codimension $\geq 2$. By Hartogs' Lemma, $\beta$ extends to a unique morphism $L^+\mc{G}_{\A^n} \to \overline{L^+\mc{G}}_{ \A^n}$.
\end{proof}

\begin{remark}
	When $\Gamma$ is taken to be the trivial group, we have $\Gr_{\G,C^n}=\Gr_{G,C^n}$. Clearly the proofs of Lemma \ref{770893} and Proposition \ref{882893} are still valid.
\end{remark}

Following \cite{Faltings03}, we can construct a determinant line bundle $\mc{L}_{\mr{det}}$ on $\Gr_{\mc{G}, {C}^n}$, which has a natural $L^+\mc{G}_{C^n}$-action. The adjoint representation $\mr{Ad}:\mc{G}\to \mr{GL}(\mc{V}_0)$ induces a morphism $\phi:\Gr_{\mc{G},C^n}\to \Gr_{\mr{GL}(\mc{V}_0),C^n}$. For each morphism $\mr{Spec}(R)\to \Gr_{\mr{GL}(\mc{V}_0),C^n}$, i.e.\,a triple $\big(\vec{p}=(p_1,\ldots, p_n), \mc{V}, \beta:\mc{V}|_{\bar C_R\setminus\cup\Gamma_{\bar p_i}} \simeq \mc{V}|_{\bar C_R\setminus\cup\Gamma_{\bar p_i}}\big)$, where $p_i\in C_R$, and $\mc{V}$ is a vector bundle on $\bar C_R$, there exists $N>0$ such that 
$\mc{V}_0(-N(\sum \Gamma_{\bar p_i}))\subseteq \mc{V} \subseteq \mc{V}_0(N(\sum \Gamma_{\bar p_i}))$. We associate a line bundle 
\begin{equation}\label{823479}
	\textstyle\bigwedge^{\mr{top}}(\mc{V}_0(N(\sum \Gamma_{\bar p_i})/\mc{V})\otimes 
	\bigwedge^{\mr{top}}(\mc{V}_0(N(\sum \Gamma_{\bar p_i})/\mc{V}_0)^{-1}
\end{equation} 
on $\mr{Spec}(R)$. This assignment gives a line bundle $\mc{L}_{\mr{det}}$ on $ \Gr_{\mr{GL}(\mc{V}_0),C^n}$. Then the pullback $\phi^*\mc{L}_{\mr{det}}$ is a line bundle on $\Gr_{\mc{G},C^n}$ with a natural $L^+\mc{G}_{C^n}$-action, which will be still denoted by $\mc{L}_{\mr{det}}$. 
\begin{thm}\label{206120}
	Let $\Gamma$ be either trivial or generated by $\sigma$. Let $C$ be an irreducible smooth curve with a faithful $\Gamma$-action such  that all ramified points in $C$ are totally ramified.  Then, there is an $L^+\mc{G}_{C^n}$-equivariant structure on the level one line bundle $\mc{L}_{C^n}$ over $\Gr_{\mc{G}, {C}^n}$.
\end{thm} 
\begin{proof}
	We first prove the case when $\Gamma$ is trivial and $\G$ is the constant group scheme. Let $\{U_i\}$ be a cover of $C$ with \'etale morphisms $U_i\to \A^1$. By Proposition \ref{329321}, we have the following isomorphism for some $Z_i$ closed in $U_i^n$,
	\begin{equation}\label{582873}
		\Gr_{G, U_i^n\setminus Z_i}\simeq \Gr_{G,\A^n}\times_{\A^n} (U_i^n\setminus Z_i).
	\end{equation}
	Note that the determinant line bundle $\mc{L}_{\mr{det}}$ on $\Gr_{G,\A^n}$ automatically has a trivialization along the section $e$. By Theorem \ref{284344}(1), we have $\mc{L}_{\mr{det}}\simeq L_{\A^n}^l$ for some $l>0$, where $L_{\A^n}$ denotes the level one line bundle on $\Gr_{G,\A^n}$. Applying Proposition \ref{882893} to the case when $\Gamma$ is trivial, we get an $L^+G_{\A^n}$-equivariant structure on $L_{\A^n}$. Since the isomorphism  \eqref{582873} commutes with the $L^+G_{U_i^n\setminus Z_i}$ action, the level one line bundle $\mc{L}_{U_i^n\setminus Z_i}:=\phi^*(L_{\A^n})$ on $\Gr_{G, U_i^n\setminus Z_i}$ is equipped with an $L^+G_{U_i^n\setminus Z_i}$-equivariant structure, where $\phi:\Gr_{G, U_i^n\setminus Z_i}\to \Gr_{G,\A^n}$ is the natural morphism induced from \eqref{582873}. When $n=1$,  $Z_i$ is empty, the level one line bundle $\mc{L}_C$ on $\Gr_{G,C}$ can be recovered by gluing these rigidified line bundles $\mc{L}_{U_i}$. Hence, $\mc{L}_C$ also has an $L^+G_{C}$-equivariant structure  by the uniqueness of splitting, cf.\,Lemma \ref{770893}. When $n\geq 2$, for any non-trivial partition $\xi=(I_1,\ldots,I_k)$ of $[n]$, the restriction of the line bundle $\mc{L}_{C^n_\xi}\simeq (\mc{L}_{C^{I_1}}\boxtimes\cdots\boxtimes\mc{L}_{C^{I_k}})|_{C^n_\xi}$ has an ${L}^+G_{C^n_\xi}$-equivariant structure by induction. Note that $\{U_i^n\setminus Z_i, C^n_\xi\}_{i,\xi}$ is a cover of $C^n$ and the level one line bundle $\mc{L}_{C^n}$ can be recovered by gluing the rigidified line bundles $\{\mc{L}_{(U_i)^n\setminus Z_i}, \mc{L}_{C^n_\xi}\}_{i,\xi}$. Again, by the uniqueness of splitting, $\mc{L}_{C^n}$ admits an $L^+G_{C^n}$-equivariant structure.
	
	\medskip
	
	When $\Gamma$ is the cyclic group generated by the standard automorphism $\sigma$, let $\{\mathring{C},U_i \}_{i=1}^r$ be the covering of $C$ constructed in Theorem \ref{797528}. Recall Proposition \ref{284344} for the construction of the level one line bundle $\mc{L}_{(\mathring{C})^n}$ on $\Gr_{\G,(\mathring{C})^n}$. By the case when $\Gamma$ is trivial and the uniqueness of splitting, $\mc{L}_{(\mathring{C})^n}$ has an $L^+\G_{(\mathring{C})^n}$-equivariant structure. For each $i$, by Corollary \ref{prop_etale_base}, we have a base change map
	\begin{equation}
		\Gr_{\G, U_i^n\setminus Z_i}\simeq \Gr_{\G,\A^n}\times_{\A^n} (U_i^n\setminus Z_i),
	\end{equation}
	where $Z_i$ is define in \eqref{773922}. By the same argument as in the case when $\Gamma$ is trivial, the level one line bundle 
	$\mc{L}_{(U_i)^n\setminus Z_i}$ on $\Gr_{\G, U_i^n\setminus Z_i}$ has an ${L}^+\mc{G}_{(U_i)^n\setminus Z_i}$-equivariant structure  for $1\leq i\leq r$. We prove the theorem by induction on $n$. When $n=1$, the level one line bundle $\mc{L}_{C}$ is constructed by gluing the rigidified line bundles $\{\mc{L}_{\mathring{C}},\mc{L}_{U_i}\}_{i=1}^r$. Thus, $\mc{L}_C$ is equipped with an $L^+\mc{G}_{C}$-equivariant structure naturally by the uniqueness of splitting, cf.\,Lemma \ref{770893}. When $n\geq 2$, for any non-trivial partition $\xi=(I_1,\ldots,I_k)$ of $[n]$, the restriction of the line bundle $\mc{L}_{C^n_\xi}\simeq (\mc{L}_{C^{I_1}}\boxtimes\cdots\boxtimes\mc{L}_{C^{I_k}} )|_{C^n_\xi}$ has an ${L}^+\mc{G}_{C^n_\xi}$-equivariant structure by induction. As in the proof of Theorem \ref{797528}, the level one line bundle $\mc{L}_{C^n}$ can constructed by gluing the rigidified line bundles $\{\mc{L}_{(\mathring{C})^n},\mc{L}_{(U_i)^n\setminus Z_i}, \mc{L}_{C^n_\xi}\}_{i,\xi}$. Again, by the uniqueness of splitting, $\mc{L}_{C^n}$ is equipped with an $L^+\mc{G}_{C^n}$-equivariant structure.
\end{proof}
\subsection{Level one line bundle on BD Grassmannians of adjoint type}\label{528475}
In this subsection, we assume that $G$ is of adjoint type. Let $G'$ be the simply-connected cover of $G$. Let $T'$ be the maximal torus of $G'$ which is the preimage of $T$ via the projection $G'\to G$. We have the natural embedding $X_*(T')\subseteq X_*(T)$. 

Recall that the components of affine Grassmannian $\Gr_{G}$ can be parameterized by the set $X_*(T)/\check{Q}$. Let $\A^1$ be the affine line with a coordinate $z$ and the standard $\Gamma$-action as in (\ref{eq_standard}).  Set $\mathring{\A}^1=\mb{A}^1-0$. Then,  $\Gr_{\mc{G},\mathring{\A}^1}\simeq\Gr_G\times \mathring{\A}^1$, cf.\,\cite[Lemma 3.4.2]{Zhu17}. Thus, the components of $\Gr_{\mc{G},\mathring{\A}^1}$ can be parameterized by the elements of $X_*(T)/\check{Q}$. For any $\kappa\in X_*(T)/\check{Q}$, denote by $\Gr_{\mc{G},\mathring{\A}^1} [\kappa]$ the component corresponding to $\kappa$. Let $\Gr_{\mc{G},\mathbb{A}^1}[\kappa]$ denote the closure of $\Gr_{\mc{G},\mathring{\A}^1}[\kappa]$ in $\Gr_{\mc{G},\A^1}$, and we call it the $\kappa$-component of $\Gr_{\mc{G},\A^1}$.

More generally, for $n\geq 2$, let $\xi$ be the finest partition of $[n]$, then there is a factorization over the open subset $\mathring{\A}^n_\xi:=\{\vec{p}\in (\mathring{\A}^1)^n\mid \bar{p}_i\neq\bar{p}_j \ \forall i\neq j\}$ of $\A^n$
\[\Gr_{\G,\mathring{\A}^n_\xi}\simeq (\Gr_{\G,\mathring{\A}^1}\times\cdots\times \Gr_{\G,\mathring{\A}^1})|_{\mathring{\A}^n_\xi}.\]
For any $\vec{\kappa}=(\kappa_1,\ldots,\kappa_n)\in \big(X_*(T)/\check{Q}\big)^n$, let $\Gr_{\G,\mathring{\A}_\xi^n}[\vec{\kappa}]:=(\prod\Gr_{\G,\mathring{\A}^1}[\kappa_i])|_{\mathring{\A}^n_\xi}$. Define $\Gr_{\G,\A^n}[\vec{\kappa}]$ to be the closure of $\Gr_{\G,\mathring{\A}_\xi^n}[\vec{\kappa}]$ in $\Gr_{\G,\A^n}$, and again call it the $\vec{\kappa}$-component of $\Gr_{\G,\A^n}$. When $\vec{\kappa}=\vec{0}$, we denote it by $\Gr_{\G,\A^n}[\vec{o}]$.

Let $\mc{G}'$ be the group scheme $\mr{Res}_{C/\bar{C}}(G'_C)^\Gamma$. There is a natural morphism $\mc{G}'\to \mc{G}$.
Let $\iota:\Gr_{\G',\A^n}\to \Gr_{\G,\A^n}$ be the morphism given by $(\vec{p},\mc{F},\beta)\mapsto (\vec{p}, \bar{\mc{F}},\bar\beta)$, where $\bar{\mc{F}}=\mc{F}\times^{\mc{G}'}\mc{G}$ and $\bar\beta$ is the trivialization induced from $\beta$.

\begin{lem}
	\label{lem_emb}
	The morphism $\iota :\Gr_{\G',\A^n}\to \Gr_{\G,\A^n}$ is a closed embedding.
\end{lem} 
\begin{proof}
From $\iota$, we naturally get a morphism $\iota_{\la}:\Grb^{\la}_{\G',\A^n}\to \Grb^{\la}_{\G,\A^n}$ for any $\la \in (X_*(T')^+)^n$. Over each point $\vec{p}\in \A^n$, the morphism $\iota_{\la,\vec{p} }: \Grb^{\la}_{\G',\vec{p}}\to \Grb^{\la}_{\G,\vec{p}}$ is an isomorphism, cf.\,Corollary \ref{711570}. By Lemma \ref{109830}, $\iota_{\la}$ must be an isomorphism. Taking the direct limit of $\iota_{\la}$ with respect to the standard partial order $\preceq $, by Proposition \ref{423483} we recover the map $\iota$. Thus, $\iota$ is a closed embedding. 
\end{proof}
From Lemma \ref{lem_emb}, the morphism $\iota$ restricts to an isomorphism $\Gr_{\G',\mathring{\A}_\xi^n}\simeq \Gr_{\G,\mathring{\A}_\xi^n}[\vec{o}]$. Thus, we get an isomorphism 
\[\Gr_{\G',\A^n}\simeq \Gr_{\G,\A^n}[\vec{o}].\]
For each $1\leq i\leq n$, let $z_i$ be the $i$-th coordinates of $\A^n$. For any $\vec\mu=(\mu_1,\ldots,\mu_n)\in (X_*(T))^n$, we define 
\[n_{\vec\mu}:=\prod_{i=1}^n\prod_{\gamma\in \Gamma}( t-\gamma^{-1}\cdot z_i)^{\gamma(\mu_i)}\] 
to be a section $\A^n\to L\mc{G}_{\A^n}$, where $t$ is the standard coordinate of $C$. Point-wisely, it sends $\vec{p}$ to $\prod_{i,\gamma}( t-\gamma(p_i))^{\gamma(\mu_i)}\in L\mc{G}_{\vec{p}}$. Let $\kappa_i$ be the image of $\mu_i$ in $X_*(T)/\check{Q}$, then $\Grb_{\G,\A^n}^{\vec\mu}\subseteq \Gr_{\G,\A^n}[ \vec{\kappa}]$, and the translation by $n_{\vec\mu}$ gives an isomorphism
\begin{equation}
	\label{eq_T_tran}
	T_{\vec{\mu}} :\Gr_{\mc{G},\A^n}[\vec{o}] \simeq \Gr_{\mc{G},\A^n}[\vec{\kappa}].
\end{equation}
Let $\mc{L}'_{\A^n}$ be the level one line bundle over $\Gr_{\G',\A^n}$, which is constructed in Section \ref{201752}. It gives rise to a line bundle $\mc{L}_{\A^n}$ over $\Gr_{\mc{G},\A^n}[\vec{o}]$ via the isomorphism $\Gr_{\mc{G}',\A^n} \simeq \Gr_{\mc{G},\A^n}[\vec{o}]$. Denote by $\mc{L}_{\A^n,\vec\mu}:=(T_{\vec\mu})_*\mc{L}_{\A^n}$ the push forward of $\mc{L}_{\A^n}$ via $T_{\vec\mu}$. 
\begin{lem}
	Given any $\vec{\mu}, \vec\nu\in \big(X_*(T)^+\big)^n$ whose images in $\big(X_*(T)/\check{Q}\big)^n$ are $\vec{\kappa}$, there is an isomorphism $\mc{L}_{\A^n,\vec\nu}\simeq \mc{L}_{\A^n,\vec{\mu}}$ of line bundles over $\Gr_{\mc{G},\A^n,\vec{\kappa}}$.
\end{lem}
\begin{proof}
	Consider the isomorphism $T_{\vec\nu-\vec{\mu}}:\Gr_{\G,\A^n}[\vec{o}]\to \Gr_{\G,\A^n}[\vec{o}]$. Regard $(T_{\vec\nu-\vec{\mu}})_*\mc{L}_{\A^n}$ and $ \mc{L}_{\A^n}$ as line bundles over $\Gr_{\G',\A^n}$. Then they have rigidified structures automatically since any line bundle over $\A^n$ is trivializable. Moreover, they have the same central charge at any point. By Theorem \ref{797528}, $(T_{\vec\nu-\vec{\mu}})_*\mc{L}_{\A^n}\simeq \mc{L}_{\A^n}$. Pushing forward via $T_{\vec{\mu}}$, we get an isomorphism $\mc{L}_{\A^n,\vec\nu}\simeq \mc{L}_{\A^n,\vec{\mu}}$. 
\end{proof}
From this lemma, the line bundle $(T_{\vec\mu})_*\mc{L}_{\A^n}$ over $\Gr_{\mc{G},\A^n}[\vec{\kappa}]$ only depends on $\vec{\kappa}$. We shall denote it by $\mc{L}_{\A^n,\vec{\kappa}}$. 
\begin{prop}\label{324995}
	For any $\vec{\kappa}=(\kappa_1,\ldots,\kappa_n)\in \big(X_*(T)/\check{Q}\big)^n$, set $\kappa:=\sum \kappa_i$. Then
	\begin{enumerate}
		\item There is an isomorphism of line bundles over $\Gr_{\G,\Delta_{\A^n}}[\vec{\kappa}]$
		\[\mc{L}_{\A^n,\vec{\kappa}}|_{\Delta_{\A^n}}\simeq \mc{L}_{\A^1,\kappa}.\]
		\item Given a partition $\xi=(I,J)$ of $[n]$,  there is an isomorphism of line bundles over $\Gr_{\G, C_\xi^n}[\vec{\kappa}]$
		\[\mc{L}_{\A^n,\vec{\kappa}}|_{\A^n_\xi}\simeq \big(\mc{L}_{\A^I,\vec{\kappa}_I}\boxtimes \mc{L}_{\A^J,\vec{\kappa}_J}\big)|_{\A^n_\xi}.\]
	\end{enumerate} 
\end{prop}
\begin{proof}
	(1) By Theorem \ref{797528}, the restriction of $\mc{L}'_{\A^n}|_{\Delta_{\A^n}}$ to the diagonal $\Gr_{\G',\Delta_{\A^n}}\simeq \Gr_{\G',\A^1}$ is isomorphic to $\mc{L}'_{\A^1}$. Thus, $\mc{L}_{\A^n}|_{\Delta_{\A^n}}\simeq \mc{L}_{\A^1}$ as line bundles over $ \Gr_{\G,\Delta_{\A^n}}[\vec{o}]\simeq \Gr_{\G,\A^1}[o]$. For each $i$, choose $\mu_i\in X_*(T)$ such that its image in $X_*(T)/\check{Q}$ is $\kappa_i$. Set $\vec\mu=(\mu_1,\ldots,\mu_n)$ and $\mu=\sum\mu_i$. Restricting the section $n_{\vec\mu}$ to the diagonal, we get a map $\A^1\simeq \Delta_{\A^n}\xrightarrow{n_{\vec\mu}} L\mc{G}_{\Delta_{\A^n}}\simeq L\mc{G}_{\A^1}$ sending $p$ to $\prod_\gamma(t-\gamma(p))^{\gamma(\sum\mu_i)}$. Let $\mu=\sum \mu_i\in X_*(T)$, then this composition map coincides with the section $n_{\mu}:\A^1\to L\mc{G}_{\A^1}$. Therefore, 
	\[\mc{L}_{\A^n,\vec{\kappa}}|_{\Delta_{\A^n}}\simeq\big((T_{\vec\mu})_*\mc{L}_{\A^n}\big)|_{\Delta_{\A^n}}\simeq (T_{\mu})_*\mc{L}_{\A^1}\simeq \mc{L}_{\A^1,\kappa}\] 
	as line bundles over $\Gr_{\G,\Delta_{\A^n}}[\vec{\kappa}]$.
	
	(2) It follows from the factorization properties of $\mc{L}'_{\A^n}$ and $n_{\vec\mu}$.
\end{proof}
\begin{prop}\label{235859}
	For any $\vec{\kappa}\in \big(X_*(T)/\check{Q}\big)^n$, there is a unique $L^+\mc{G}'_{\A^n}$-equivariant structure on $\mc{L}_{\A^n,\vec{\kappa}}$.
\end{prop}
\begin{proof}
	In view of Lemma \ref{770893}, the uniqueness is clear. We now prove the existence. 
	Choose $\vec{\mu}\in (X_*(T))^n$ such that its image in $(X_*(T)/\check{Q})^n$ is $\vec{\kappa}$. Recall that $\mc{L}_{\A^n,\vec{\kappa}}$ is the push forward of $\mc{L}_{\A^n}$ via $T_{\vec{\mu}}:\Gr_{\mc{G},\A^n}[\vec{o}] \simeq \Gr_{\mc{G},\A^n}[\vec{\kappa}]$. Thus, an $L^+\mc{G}'_{\A^n}$-equivariant structure on $\mc{L}_{\A^n,\vec{\kappa}}$ is equivalent to an $\mr{Ad}_{n_{-\vec{\mu}}}(L^+\mc{G}'_{\A^n})$-equivariant structure on the line bundle $\mc{L}_{\A^n}$ over $\Gr_{\G,\A^n}[\vec{o}]$. By the definition of $\mc{L}_{\A^n}$, this is equivalent to an $\mr{Ad}_{n_{-\vec{\mu}}}(L^+\mc{G}'_{\A^n})$-equivariant structure on the level one line bundle $\mc{L}'_{\A^n}$ over $\Gr_{\G',\A^n}$.
	
	Consider the morphism $\phi:\Gr_{\G,\A^n}\to \Gr_{\mr{GL}(\mc{V}_0)}$ induced from the adjoint representation $\mc{G}\to \mr{GL}(\mc{V}_0)$. One can define a determinant line bundle $\mc{L}_{\mr{det}}$ over $\Gr_{\mr{GL}(\mc{V}_0)}$ by \eqref{823479}. Then the pullback $\phi^*\mc{L}_{\mr{det}}$ is a line bundle over $\Gr_{\G,\A^n}$ with a natural $L^+\mc{G}_{\A^n}$-equivariant structure, which will still be denoted by $\mc{L}_{\mr{det}}$. Restricting to the $\vec{\kappa}$-component, we get an $L^+\mc{G}'_{\A^n}$-equivariant structure on $\mc{L}_{\mr{det}}|_{\Gr_{\G,\A^n}[\vec{\kappa}]}$. This implies that $\mc{L}_{\mr{det}}|_{\Gr_{\G,\A^n}[\vec{o}]}$ has an $\mr{Ad}_{n_{-\vec{\mu}}}(L^+\mc{G}'_{\A^n})$-equivariant structure. 
	Via the identification $\Gr_{\G',\A^n}\simeq \Gr_{\G,\A^n}[\vec{o}]$, by Theorem \ref{797528} we have $\mc{L}_{\mr{det}}|_{\Gr_{\G,\A^n}[\vec{o}]}\simeq (\mc{L}'_{\A^n})^l$ for some $l$. Thus $(\mc{L}'_{\A^n})^l$ is $\mr{Ad}_{n_{-\vec{\mu}}}(L^+\mc{G}'_{\A^n})$-equivariant.

	For any $\la\in (X_*(T))^n$ whose image in $(X_*(T)/\check{Q})^n$ is $\vec{\kappa}$, by the same proof of Lemma \ref{223442}, one can show the $L^+\G'_{\A^n}$-action on $\Grb^{\la}_{\G,\A^n}\subseteq \Gr_{\G,\A^n}[\vec{\kappa}]$ factors through a finite type group scheme $L_k^+\G'_{\A^n}$. Let $_{\vec{\mu}}L^+\mc{G}'_{\A^n}$ denote $\mr{Ad}_{n_{-\vec{\mu}}}(L^+\mc{G}'_{\A^n})$, and regard $_{\vec{\mu}}\Grb^{\la}_{\G',\A^n}:=T_{\vec{\mu}}^{-1}\big(\Grb^{\la}_{\G,\A^n}\big)$ as a subscheme of $\Gr_{\G',\A^n}$. Then the $_{\vec{\mu}}L^+\mc{G}'_{\A^n}$-action on $_{\vec{\mu}}\Grb^{\la}_{\G',\A^n}$ factors through a finite type group scheme $_{\vec{\mu}}L_{k}^+\mc{G}'_{\A^n}$. Define $_{\vec{\mu}}\overline{L_k^+\mc{G}'}_{\A^n}$ similarly as in \eqref{653942}, by the same argument as in Proposition \ref{882893}, one can show the following central extension splits
	\begin{equation}
		1\to Z_l \to _{\vec{\mu}}\overline{L_k^+\mc{G}'}_{\A^n}\xrightarrow{p} _{\vec{\mu}}L_k^+\mc{G}'_{\A^n}\to 1.
	\end{equation}
	It follows that the line bundle $\mc{L}'_{\A^n}$ over $\Gr_{\G',\A^n}$ is $\mr{Ad}_{n_{-\vec{\mu}}}(L^+\mc{G}'_{\A^n})$-equivariant. Equivalently, 
	the line bundle $\mc{L}_{\A^n,\vec{\kappa}}$ is $L^+\mc{G}'_{\A^n}$-equivariant. 
\end{proof}
\begin{definition}\label{436248}
	For any BD Schubert variety $\Grb^{\la}_{\G,\A^n}$, it is contained in a unique $\vec{\kappa}$-component $\Gr_{\G,\A^n}[\vec{\kappa}]$. With no confusion, we will denote by $\mathcal{L}_{\A^n}$ the restriction of $\mc{L}_{\A^n,\vec{\kappa}}$ on $\Grb^{\la}_{\G,\A^n}$, and call it the level one line bundle over $\Grb^{\la}_{\G,\A^n}$. 
\end{definition}
\begin{cor}\label{432749}
	The space $H^0\big(\Grb_{\mc{G}, \A^n}^{\la}, \mc{L}_{\A^n}\big) $ is a $\gs$-module. 
\end{cor}
\begin{proof}
 By Proposition \ref{235859}, the group scheme $L^+\G'_{\A^n}$ acts on $H^0\big(\Grb_{\mc{G}, \A^n}^{\la}, \mc{L}_{\A^n}\big) $. This induces an action of ${\rm Lie} (L^+\G’_{\A^n})$ on $H^0\big(\Grb_{\mc{G}, \A^n}^{\la}, \mc{L}_{\A^n}\big) $. Note that $L^+\G’_{\A^n}\simeq \varprojlim L_k^+\G’_{\A^n}$, where $L_k^+\G'_{\A^n}$ is defined in (\ref{eq_jet}). Thus, 
 $$ {\rm Lie } (L^+ \G'_{\A^n})\simeq \varprojlim_{k} \mathfrak{g} \bigg( \mathbb{C}[t, z_1,\cdots, z_n] /\prod_{i=1}^n \prod_{\gamma\in \Gamma}(t- \gamma\cdot z_i)^k \bigg)^\Gamma \otimes_{\mathbb{C}[\bar{\A}^n ] }\mathbb{C}[\A^n]  .$$
There is a natural morphism $\gs\to {\rm Lie } (L^+\G'_{\A^n})$ induced from the natural inclusion $\mathbb{C}[t]\to \mathbb{C}[t,z_1,\cdots, z_n]$. Therefore, $H^0\big(\Grb_{\mc{G}, \A^n}^{\la}, \mc{L}_{\A^n}\big)$ is a $\gs$-module. 
\end{proof}

\section{Global Demazure modules and BD Schubert varieties}\label{753118}
In this section, we first equip a $\mathbb{G}_m$-equivariant structure on level one line bundles on $\Gr_{\G,\A^n}$. Then, we prove a Borel-Weil type theorem on BD Schubert varieties for twisted global Demazure modules. 

\subsection{$\mathbb{G}_m$-equivariance on the level one line bundle $\mc{L}_{\A^n}$}\label{324893}
Let $\sigma$ be a standard automorphism $G$ of order $m$ as in Section \ref{133740}, and let $\Gamma$ be the cyclic group generated by $\sigma$. Denote by $\C[t]$ the coordinate ring of the affine space $\A^1$. Recall the standard $\Gamma$-action on $\A^1$ given by $\sigma(t)=\epsilon^{-1}t$. Let $\pi:\A^1\to \bar \A^1$ denote the quotient map sending $a$ to $\bar{a}$. We can identify $\bar \A^1$ with $\A^1$. Under this identification, $\pi$ can be regarded as the map $\A^1\to \A^1$ given by $a\mapsto a^m$. Similarly, we can define the $\Gamma$-action on $\mb{G}_m$ and the quotient space $\bar{\mb{G}}_m$.

Let $\mc{G}=\mr{Res}_{\A^1/\bar{\A}^1}\big( G\times \A^1\big) ^\Gamma$ be the parahoric Bruhat-Tits group scheme over $\bar{\A}^1$ defined as in Section \ref{BD_Gr}.  For any scheme $S$ over $\mathbb{C}$,  Set $\A^1_S:=\A^1\times S$ and $\bar \A^1_S:=\bar\A^1\times S$. We denote by $\mc{G}_S$ the group scheme $\mc{G}\times_{\bar\A^1} \bar\A^1_S$ over $\bar\A^1_S$.
For any $h\in \mb{G}_m(S)$, let $\rho_h: \A^1_S \to \A^1_S$ be the map sending $(a,s)$ to $(h(s)a,s)$. It induces a map $\bar{\rho}_h:\bar{\A}_S^1\to \bar{\A}_S^1$ sending $(\bar{a},s)$ to $(\bar{h}(s) \bar{a},s)$, where $\bar{h}$ is image of $h$ via $\mathbb{G}_m(S)\to \bar{\mathbb{G}}_m(S)$. 

\begin{lem}\label{452674}
For any $h\in \mb{G}_m(S)$, there is an isomorphism of group schemes over $\bar \A^1_S$
	\[\bar{\rho}_h^*(\mc{G}_S)\simeq \mc{G}_S.\] 
	As a consequence, the pullback $(\bar{\rho}_h)^* \mc{F}$ of a $\mc{G}$-torsor $\mc{F}$ over $\bar\A^1\times S$ for any scheme $S$ over $\mb{C}$, via $\bar{\rho}_h$ is also a $\mc{G}$-torsor.
\end{lem}
\begin{proof}
	For any scheme $T$ over $\bar \A_S^1$, denote by $(\bar{\rho}_h)_*T$ its pull-forward via $\bar{\rho}_h$, i.e.\,$(\bar{\rho}_h)_*T$ is regarded as the scheme over $\bar{\A}_S^1$ via the composition $T\to \bar{\A}_S^1\xrightarrow{\bar{\rho}_h} \bar{\A}_S^1$. There is a $\Gamma$-equivariant isomorphism $f:T\times_{\bar \A_S^1}\A_S	^1\simeq ((\bar{\rho}_h)_*T)\times _{\bar \A_S^1}\A_S^1$ of schemes, given by $(t,a)\mapsto (t, h^{m}(a))$. This induces an isomorphism of groups
	\[ (\bar{\rho}_h^* \mc{G}_S)(T)= \mc{G}_S\big((\bar{\rho}_h)_*T\big)=G\Big(\big((\bar{\rho}_h)_*T\big)\times _{\bar \A_S^1}\A_S^1\Big)^\Gamma\xrightarrow{f} G\big(T\times_{\bar \A_S^1}\A_S^1\big)^{\Gamma}=\mc{G}_S(T). \]
	Thus, $\bar{\rho}_h^*(\mc{G}_S)\simeq \mc{G}_S$. 
	
	For any $\mc{G}_S$-torsor $\mc{F}$ over $\bar \A^1_S$, the pull-back $\bar{\rho}_h^*\mc{F}$ is an $\bar{\rho}_h^*\mc{G}_S$-torsor, and hence is a $\mc{G}_S$-torsor via isomorphism $ \bar{\rho}_h^*\mc{G}_S\simeq \G_S$. 

\end{proof}

From the proof, one can see that this proposition also holds for $\bar{\mb{P}}^1$. Let $\mr{Bun}_{\mc{G}}$ be the moduli stack of $\mc{G}$-torsors on $\bar{\mb{P}}^1$. For any scheme $S$ over $\mb{C}$ and $h\in \mb{G}_m(S)$, we define a map $\phi_h: \mr{Bun}_{\mc{G}}(S)\to \mr{Bun}_{\mc{G}}(S)$, sending $\mc{F}$ to $\bar{\rho}_h^*\mc{F}$. This gives an $\mb{G}_m$-action on $\mr{Bun}_{\mc{G}}$.

By Lemma \ref{452674}, for any $h\in \mb{G}_m(S)$, one can define the following map 
\begin{align}\label{432859}
	\psi_h:\Gr_{\G,\A^1}(S) & \to \Gr_{\G,\A^1}(S) \\
	\nonumber (p,\mc{F},\beta) & \mapsto (\,h\cdot p, \bar{\rho}_h^*\mc{F}, \bar{\rho}_h^*\beta\,).
\end{align}
Recall that $\Gr_{\G,0}=\Gr_{\ms{G}}$, the restriction of $\psi_h$ to the $0$-fiber induces a map $\Gr_{\ms{G}}(S)\to \Gr_{\ms{G}}(S)$, which will still be denoted by $\psi_h$. This gives a $\mb{G}_m$-action on $\Gr_{\ms{G}}$. In fact, one can check this action agrees with the action induced from the natural $\mb{G}_m$-action on $\mc{O}=\C[[t]]$ sending $t$ to $a^{-1}t$ for any $a\in \mb{G}_m(\C)$.

Now, we assume $G$ is simply connected. By the uniformization theorem \cite{Heinloth10}, we have $\mr{Bun}_{\mc{G}}\simeq G[t^{-1}]^\sigma\backslash \Gr_{\ms{G}}$ as algebraic stacks. Then, we have the following commutative diagram
\begin{equation*}
	\xymatrix{
		\Gr_{\ms{G}}(S) \ar[d] \ar[r]^{\psi_h} & \Gr_{\ms{G}}(S) \ar[d] \\
		\mr{Bun}_{\mc{G}}(S) \ar[r]^{\phi_h} & \mr{Bun}_{\mc{G}}(S)
	}.
\end{equation*}
Let $\ms{L}$ be the level one line bundle on $\Gr_{\ms{G}}$. From \cite[Theorem 3.13]{BH}, the line bundle has a natural $G[t^{-1}]^\sigma\rtimes \mathbb{G}_m$-equivariant structure. The $G[t^{-1}]^\sigma$-equivariant structure on the line bundle $\ms{L}$ descends to the ample generator $\mc{L}$ of $\mr{Pic}(\mr{Bun}_{\mc{G}})$. Moreover, Lemma \ref{708537} implies that $\mathbb{G}_m$-action on $\ms{L}$ also descends to a $\mathbb{G}_m$-action on $\mc{L}$.

Let $\Gr_{\mc{G}, \A^n}$ be the Beilinson-Drinfeld Grassmannian over $\A^n$. For $h\in \mb{G}_m$, we define a map $\Gr_{\mc{G}, \A^n}\to \Gr_{\mc{G}, \A^n}$ sending $(a_1,\ldots,a_n,\mc{F},\beta)$ to $(ha_1,\ldots,ha_n, (h^m)^*\mc{F},(h^m)^*\beta)$. It gives a $\mb{G}_m$-action on $\Gr_{\mc{G},\A^n}$. Recall that the level one line bundle $\mc{L}_{\A^n}$ on $\Gr_{\mc{G},\A^n}$ is the pull-back of $\mc{L}$ via the projection $\Gr_{\mc{G},\A^n}\to \mr{Bun}_{\mc{G}}$. Then the $\mb{G}_m$-action on $\mc{L}$ induces a $\mb{G}_m$-action on $\mc{L}_{\A^n}$. One can check this action is compatible with the $\mb{G}_m$-action on $\Gr_{\mc{G}, \A^n}$. In summary, we have the following result.
\begin{lem}\label{374239}
	When $G$ is simply connected, there exists a $\mb{G}_m$-equivariant structure on the level one line bundle $\mc{L}_{\A^n}$ over $\Gr_{\G,\A^n}$, which restricts to the natural $\mb{G}_m$-equivariant structure on the line bundle $\mc{L}_{\vec{p}}$ over $\Grb^{\la}_{\G,\vec{p}}$ for any $\vec{p}\in \A^n $. \qed
\end{lem}

Let $G$ be adjoint, and let $G'$ be its simply connected cover. With the same setup as in Section \ref{528475}, for any BD Schubert variety $\Grb^{\la}_{\G,\A^n}$, it is contained in a unique $\vec{\kappa}$-component $\Gr_{\G,\A^n}[\vec{\kappa}]$. Recall Definition \ref{436248}, the level one line bundle $\mc{L}_{\A^n}$ over $\Grb^{\la}_{\G,\A^n}$ is defined to be the restriction of $\mc{L}_{\A^n,\vec\kappa}$ to $\Grb^{\la}_{\G,\A^n}$.
\begin{prop}\label{543573}
	When $G$ is adjoint, there exists a $\mb{G}_m$-equivariant structure on the level one line bundle $\mc{L}_{\A^n}$ over $\Grb^{\la}_{\G,\A^n}$, which restricts to the natural $\mb{G}_m$-equivariant structure on the line bundle $\mc{L}_{\vec{p}}$ over $\Grb^{\la}_{\G,\vec{p}}$ for any $\vec{p}\in \A^n $. 
\end{prop}
\begin{proof}
	For any $h\in \mb{G}_m$, we define $\psi_{h}:\Gr_{\G,\A^n}\to \Gr_{\G,\A^n}$ and $\psi'_{h}:\Gr_{\G',\A^n}\to \Gr_{\G',\A^n}$ similarly as in \eqref{432859}. Clearly, $\psi_{h}$ preserves the $\kappa$-component. Hence, it restricts to an morphism $\psi_{h}:\Gr_{\G,\A^n}[\vec\kappa]\to \Gr_{\G,\A^n}[\vec\kappa]$. One can easily check that the following diagram commutes,	
	\begin{equation*}
		\xymatrix{
			\Gr_{\G',\A^n} \ar[d]_{\psi'_{h}} \ar[r]^-{T_{\la}} & \Gr_{\G,\A^n}[\vec\kappa] \ar[d]_{\psi_{h}} \\
			\Gr_{\G',\A^n} \ar[r]^-{T_{\la}} & \Gr_{\G,\A^n}[\vec\kappa]
		},
	\end{equation*}
	where $T_{\vec{\lambda}}$ is defined in (\ref{eq_T_tran}). Let $\mc{L}'_{\A^n}$ be the level one line bundle over $\Gr_{\G',\A^n} $. Then, the $\mb{G}_m$-equivariant structure on $\mc{L}'_{\A^n}$ induces a $\mb{G}_m$-equivariant structure on the line bundle $\mc{L}_{\A^n,\vec\kappa}:=(T_{\la})_*\mc{L}'_{\A^n}$ over $\Gr_{\G,\A^n}[\vec\kappa]$. Restricting to $\Grb^{\la}_{\G,\A^n}$, we get a $\mb{G}_m$-equivariant structure on the line bundle $\mc{L}_{\A^n}$. The second statement of the proposition directly follows from Lemma \ref{374239}.
\end{proof}

\begin{thm}\label{243693}
	With the same setup as in Proposition \ref{543573}, the space $H^0(\Grb^{\la}_{\G,\A^n},\mc{L}_{\A^n})$ is a graded free $\C[\A^n]$-module.
\end{thm}
\begin{proof}
	By Theorem \ref{257230}, Corollary \ref{711570} and the isomorphism (\ref{eq_sch_dem}), the function $\varphi(\vec{p}):=\dim_{ \C}H^0\big(\Grb^{\la}_{\G,\vec{p}},\mc{L}_{\vec{p}}\big)$ is a constant function over $\vec{p}\in \A^n$. Thus, $H^0(\Grb^{\la}_{\G,\A^n},\mc{L}_{\A^n})$ is a projective $\C[\A^n]$-module. 
	
	Moreover, the $\mb{G}_m$-equivariant structure on $\mc{L}_{\A^n}$ endows $H^0(\Grb^{\la}_{\G,\A^n},\mc{L}_{\A^n})$ a graded structure, which is compatible with the natural grading of the polynomial ring $\C[\A^n]$. Using the graded Nakayama lemma (cf.\,\cite[\href{https://stacks.math.columbia.edu/tag/0EKB}{Tag 0EKB}]{stacks-project}), $H^0(\Grb^{\la}_{\G,\A^n},\mc{L}_{\A^n})$ is a graded free $\C[\A^n]$-module.
\end{proof}

\subsection{Borel-Weil type theorem for twisted global Demazure modules }
With the same setup as in Proposition \ref{543573}, by Theorem \ref{243693}, the space $H^0(\Gr^{\la}_{\G,\A^n},\mc{L}_{\A^n})$ is a graded free $\C[\A^n]$-module. Denote by $H^0(\Gr^{\la}_{\G,\A^n},\mc{L}_{\A^n})^\vee$ its dual as a $\C[\A^n]$-module. Then, $H^0(\Gr^{\la}_{\G,\A^n},\mc{L}_{\A^n})^\vee$ is still a graded free $\C[\A^n]$-module, and is a $\gs$-module by Corollary \ref{432749}. 

In the following context, we will simply denote by $\mc{L}$ the level one line bundle $\mc{L}_{\A^n}$ over $\Grb^{\la}_{\G,\A^n}$. By the factorization of Schubert varieties in Theorem \ref{684242} and the factorization of line bundles in Proposition \ref{324995}, we have an isomorphism
\[f_\xi:H^0\big( \Grb^{\la}_{\mc{G},\A^n_{\xi}}, \mc{L}^c\big)^\vee \simeq \Big(H^0\big( \Grb^{\la_{I_1}}_{\mc{G},\A^{I_1}}, \mc{L}^c_{\A^{I_1}}\big)^\vee\otimes \cdots\otimes H^0\big( \Grb^{\la_{I_k}}_{\mc{G},\A^{I_k}}, \mc{L}^c_{\A^{I_k}}\big)^\vee\Big) \otimes_{\C[\A^n]} \C[\A^n_\xi],\]
By Theorem \ref{616072}, the following $\gs$-module 
\begin{equation}\label{234923}
	\Ds(c,\la)_{\A^n}:=\Ds(c,\la)\otimes_{\As(c,\la)}\C[\A^n]
\end{equation}
is a graded and free over $\C[\A^n]$. Finally, we have the following theorem. 

\begin{thm}\label{055073}
	For any $n>0$, there exists an isomorphism of $(\gs, \C[\A^n])$-bimodules:
	\begin{equation}
	\label{eq_final_thm}
	\Phi_{\A^n}: \Ds(c,\la)_{\A^n}\simeq H^0\big( \Grb_{\mc{G},\A^n}^{\la}, \mc{L}^c\big)^\vee.\end{equation}
	Moreover, for any nontrivial partition $\xi=(I_1,\ldots, I_k)$ of $[n]$, the following diagram commutes
	\begin{equation}\label{eq_thm_diag}
		\xymatrixcolsep{4pc}
		\xymatrix{
			 \Ds(c,\la)_{\A^n}\otimes_{\C[\A^n]} \C[\A^n_\xi] \ar@{^{(}->}[d]_{\psi_\xi} \ar[r]^-{\Phi_{\A^n}|_{\A^n_\xi}}& H^0\big( \Grb_{\mc{G},\A_\xi^n}^{\la}, \mc{L}^c\big)^\vee \ar[d]_-[@]{\sim}^{f_\xi} \\
			 \Big(\bigotimes_j\Ds(c,{\la}_{I_j})_{\A^{I_j}} \Big)\otimes_{\C[\A^n]} \C[\A^n_{\xi}] \ar[r]^-{\sim}_-{(\prod\Phi_j)|_{\A^n_\xi}} & \Big(\bigotimes_j H^0\Big( \Grb^{\la_{I_j}}_{\mc{G},\A^{I_j}}, \mc{L}^c_{\A^{I_j}}\Big)^\vee\Big) \otimes_{\C[\A^n]} \C[\A^n_\xi]
			},
	\end{equation}
where $\la_{I_j}$ is defined in \eqref{249296}, $\Phi_j:=\Phi_{\A^{I_j}}$, and the map $\psi_\xi$ is defined in (\ref{eq_psi}) in Proposition \ref{348932}. As a consequence, $\psi_\xi$ is an isomorphism. 
\end{thm}
\begin{proof}
	When $n=1$, set $\mathring{\A}^1=\A^1\setminus\{0\}$. We have an isomorphism $\Grb^\lambda_{\mc{G},\mathring{\A}^1}\simeq \Grb^{\lambda}_{G}\times \mathring{\A}^1$, cf.\,\cite[Lemma 3.4.2]{Zhu17}. Thus, there is an $\gs$-isomorphism
	\[ H^0\big( \Grb^\lambda_{\mc{G},\mathring{\A}^1}, \mc{L}^c\big)^\vee \simeq D(c, \lambda)\otimes_{\C} \C[\mathring{\A}^1],\]
 where the $\gs$-action on two sides are given by Corollary \ref{432749} and equation \eqref{554908} respectively. 
	Composing with the map $H^0\big( \Grb^\lambda_{\mc{G},\A^1}, \mc{L}^c\big)^\vee \hookrightarrow H^0\big( \Grb^\lambda_{\mc{G},\mathring{\A}^1}, \mc{L}^c\big)^\vee$, we get an embedding of $\gs$-modules
	\begin{equation}\label{340883}
		H^0\big( \Grb^\lambda_{\mc{G},\A^1}, \mc{L}^c\big)^\vee \hookrightarrow D(c, \lambda)\otimes_{\C} \C[\mathring{\A}^1].
	\end{equation}
	Recall Proposition \ref{348932} that there is an embedding of $\gs$-modules
	\begin{equation}\label{233497}
		 \Ds(c,\lambda)_{\A^1} \hookrightarrow  D(c, \lambda)\otimes_{\C} \C[\mathring{\A}^1].
	\end{equation}
	Taking the lowest weight parts of these modules in \eqref{340883} and \eqref{233497} with respect to the $\mathfrak{h}^\sigma$-action, we get the following embeddings:
	\begin{gather*}
		\phi:H^0\big( \Grb^\lambda_{\mc{G},\A^1}, \mc{L}^c\big)^\vee_{-c\iota(\lambda)} \hookrightarrow D(c, \lambda)_{-c\iota(\lambda)}\otimes_{\C} \C[\mathring{\A}^1] ,\\
		\psi: (\Ds(c,\lambda)_{\A^1})_{-c\iota(\lambda)} \hookrightarrow D(c, \lambda)_{-c\iota(\lambda)}\otimes_{\C} \C[\mathring{\A}^1].
	\end{gather*}
	We shall show $\mr{Im}(\phi)=\mr{Im}(\psi)$. Note that $\mr{Im}(\phi)$ and $\mr{Im}(\psi)$ are free rank one $\C[z]$-modules, and $ D(c, \lambda)_{c\iota(\lambda)} \otimes \C[\mathring{\A}^1]$ is a free rank one $\C[z,z^{-1}]$-module. We must have 
	\[ \mr{Im}(\phi)=z^a\mr{Im}(\psi).\]
	We claim $a=0$. By Theorem \ref{243693} and Theorem \ref{616072}, $H^0\big( \Grb_{\mc{G},\A^1}^{\lambda}, \mc{L}^c\big)^\vee$ and $\Ds(c,\lambda)_{\A^1}$ are graded free $\C[\A^1]$-modules. By Theorem \ref{398063} and Corollary \ref{711570}, we can compute their characters via the fibers at $0\in \A^1$:
	\begin{equation}\label{524378}
		\mr{ch}\big( H^0\big( \Grb_{\mc{G},\A^1}^{\lambda}, \mc{L}^c\big)^\vee\big)=\mr{ch}\big(D^\sigma(c,\lambda)\big)\cdot (1-q)^{-1}=\mr{ch}\big(\Ds(c,\lambda)_{\A^1}\big).
	\end{equation}
	In particular, on the lowest weight part, we have 
	\[\mr{ch}\big( H^0\big( \Grb_{\mc{G},\A^1}^{\lambda}, \mc{L}^c\big)^\vee_{-c\iota(\lambda)}\big)=\mr{ch}\big((\Ds(c,\lambda)_{\A^1})_{-c\iota(\lambda)}\big).\]
	Hence $\mr{ch}\big(\mr{Im}(\phi)\big)=\mr{ch}\big(\mr{Im}(\psi)\big)$. It follows that $a=0$. In other words, $\mr{Im}(\phi)=\mr{Im}(\psi)$, and furthermore 
	\[H^0\big( \Grb_{\mc{G},\A^1}^{\lambda}, \mc{L}^c\big)^\vee_{-c\iota(\lambda)} \simeq (\Ds(c,\lambda)_{\A^1})_{-c\iota(\lambda)}.\]
	Consider the following embedding of $\gs$-modules
	\begin{align*}
		\Ds(c,\lambda)_{\A^1}= \U(\gs)\bu (\Ds(c,\lambda)_{\A^1})_{-c\iota(\lambda)} &\simeq \U(\gs) \cdot H^0\big( \Grb_{\mc{G},\A^1}^{\lambda}, \mc{L}^c\big)^\vee_{-c\iota(\lambda)}\\
		&\hookrightarrow H^0\big( \Grb_{\mc{G},\A^1}^{\lambda}, \mc{L}^c\big)^\vee. 
	\end{align*}
	By the equality \eqref{524378} of characters, this embedding must be an isomorphism, which will be denoted by $\Phi_{\A^1}$.	
	\medskip
	
	When $n = 2$, let $\xi=(\{1\},\{2\})$ be the nontrivial partition. From the case when $n=1$, we have an isomorphism $(f_\xi)^{-1}\circ(\Phi_{\A^1}\times\Phi_{\A^1})|_{\A^2_\xi}$ of $\C[\A^2_\xi]$-modules
	\begin{equation}\label{423849}
		\big(\Ds(c,{\lambda}_{1})_{\A^1} \otimes \Ds(c,{\lambda}_{2})_{\A^1} \big)\otimes_{\C[\A^2]} \C[\A^2_{\xi}] \simeq H^0\big( \Grb^{\la}_{\mc{G},\A^2_{\xi}}, \mc{L}^c\big)^\vee.
	\end{equation}
	Combining with Proposition \ref{348932}, we get an embedding
	\begin{equation}\label{529673}
		\Ds(c,\la)_{\A^2} \hookrightarrow H^0\big( \Grb^{\la}_{\mc{G},\A^2_{\xi}}, \mc{L}^c\big)^\vee.
	\end{equation}
	Consider the restriction map
	\begin{equation}\label{327489}
		H^0\big( \Grb^{\la}_{\mc{G},\A^2}, \mc{L}^c\big)^\vee\hookrightarrow H^0\big( \Grb^{\la}_{\mc{G},\A^2_{\xi}}, \mc{L}^c\big)^\vee.
	\end{equation}
	Taking the lowest weight parts of these modules in \eqref{327489} and \eqref{529673} with respect to the $\mf{g}^\sigma$-action, we get the following embeddings:
	\begin{gather*}
		\phi:H^0\big( \Grb^{\la}_{\mc{G},\A^2}, \mc{L}^c\big)^\vee_{-c\iota(\lambda)}\hookrightarrow H^0\big( \Grb^{\la}_{\mc{G},\A^2_{\xi}}, \mc{L}^c\big)_{-c\iota(\lambda)}^\vee,\\
		\psi: (\Ds(c,\la)_{\A^2})_{-c\iota(\lambda)} \hookrightarrow H^0\big( \Grb^{\la}_{\mc{G},\A^2_{\xi}}, \mc{L}^c\big)_{-c\iota(\lambda)}^\vee.
	\end{gather*}
	We shall show $\mr{Im}(\phi)=\mr{Im}(\psi)$. Note that $\mr{Im}(\phi)$ and $\mr{Im}(\psi)$ are free rank one $\C[z_1,z_2]$-modules, and $ H^0\big( \Grb^{\la}_{\mc{G},\A^2_{\xi}}, \mc{L}^c\big)_{-c\iota(\lambda)}^\vee$
	is a free rank one $\C[z_1,z_2]_{\mf{a}}$-modules, where $\mf{a}$ is the ideal generated by $z_1- \epsilon^i z_2$, with $i=0,\cdots, m-1$. We must have
	\begin{equation}\label{eq_im_eq}
		\mr{Im}(\phi)=\prod_{i=0}^{m-1} (z_1-\epsilon^i z_2)^{a_i} \mr{Im}(\psi).
	\end{equation}
	We claim $a_i=0$ for any $i$. Consider the $\Gamma$-action on $\A^2$ defined by $\sigma* (x,y):=(x,\epsilon(y))$. This action extends to an action on $\Gr_{\G, \A^2 }$. From the construction of $\mc{L}$ (cf.\,(\ref{eq_linebun_L}) and Section \ref{528475}), there is a natural $\Gamma$-equivariant structure on $\mc{L}$. When $n=1$, it is clear that $\Phi_{\A^1}$ is equivariant under the standard action of $\Gamma$ defined in (\ref{eq_standard}). Hence, the morphism \eqref{423849} is $\Gamma$-equivariant. Then the induced map \eqref{529673} and hence $\psi$ is $\Gamma$-equivariant. 
	From (\ref{eq_im_eq}), the $\Gamma$-equivariance of $\phi$ and $\psi$ imply that 
	\[\textstyle\prod_i (z_1-\epsilon^{i-1} z_2)^{a_i}=\prod_i (z_1-\epsilon^i z_2)^{a_i}.\]
	It follows that $a_i=a_{i-1}$ for any $i$. Since $\mr{ch}\big(\mr{Im}(\phi)\big)=\mr{ch}\big(\mr{Im}(\psi)\big)$, we must have $\sum a_i=0$. Therefore, $a_i=0$ and $\mr{Im}(\phi)=\mr{Im}(\psi)$. By the same argument as in the case $n=1$, this leads to an isomorphism $\Phi_{\A^2}:\Ds(c,\la)_{\A^2}\simeq H^0\big( \Grb_{\mc{G},\A^2}^{\la}, \mc{L}^c\big)^\vee$. From the construction, $\Phi_{\A^2}$ satisfies diagram \eqref{eq_thm_diag}.
	
	\medskip
	
	When $n > 2$, we prove it by induction. Let $\xi=(I_1,I_2, \ldots, I_k)$ be any nontrivial partition, i.e. $\xi\neq ([n])$. 
	By induction, we get an embedding 
	\[\textstyle\Phi_{\xi}:= f_\xi^{-1}\circ(\prod\Phi_\alpha)|_{\A^n_\xi}\circ\psi_\xi: \Ds(c,\la)_{\A^n_{\xi}}\hookrightarrow H^0\big( \Grb^{\la}_{\mc{G},\A^n_{\xi}}, \mc{L}^c\big)^\vee. \]
	Let $\xi'=(I_1',\ldots, I_{k'}')$ be another nontrivial partition. Consider a new partition as follows (with the lexicographical order), 
	\[\xi'':=(I_\alpha\cap I'_\beta\mid1\leq\alpha\leq k,1\leq \beta\leq k').\]
	It is easy to check $\A^n_{\xi''}=\A^n_{\xi}\cap \A^n_{\xi'}$. For any $1\leq \alpha\leq k$, we denote by $\xi''_\alpha$ the partition $(I_\alpha\cap I'_{\beta}\mid 1\leq \beta\leq k')$ of $I_\alpha$. By induction, we have $\Phi_{\alpha}|_{\A^{I_\alpha}_{\xi''_\alpha}}=f_{\xi_\alpha}^{-1}\circ(\prod_{\beta}\Phi_{\alpha,\beta})|_{\A^{I_\alpha}_{\xi''_\alpha}}\circ\psi_{\xi_\alpha}$ over $\A^{I_\alpha}_{\xi''_\alpha}$ for any $1\leq \alpha \leq k$. It follows that
	\begin{align*}
	\textstyle \Phi_{\xi}|_{\A^n_{\xi''}} &\textstyle =f_\xi^{-1}\circ (\prod_\alpha \Phi_{\alpha})|_{\A^n_{\xi''}} \circ\psi_\xi \\
	& \textstyle =f_\xi^{-1}\circ 
	(\prod_\alpha f_{\xi''_\alpha}^{-1})|_{\A^n_{\xi''}} \circ
	(\prod_{\alpha,\beta}\Phi_{\alpha,\beta})|_{\A^n_{\xi''}} \circ 
	(\prod_\alpha\psi_{\xi''_\alpha} )|_{\A^n_{\xi''}}\circ\psi_{\xi}\\
	& \textstyle = f_{\xi''}^{-1} \circ
	(\prod_{\alpha,\beta}\Phi_{\alpha,\beta})|_{\A^n_{\xi''}} \circ 
	\psi_{\xi''}\\
	& =\Phi_{\xi''}.
	\end{align*}
	Similarly $\Phi_{\xi'}|_{\A^n_{\xi''}}=\Phi_{\xi''}$. Thus,
	$\Phi_\xi$ and $\Phi_{\xi'}$ agree on $\A^n_{\xi}\cap \A^n_{\xi'}$. Hence, we can glue these $\{\Phi_\xi\}_\xi$ together to get an embedding over $\bigcup_\xi \A^n_\xi=\A^n- \Delta_{[n]}$, i.e.
	\[\Phi: \Ds(c,\la)_{\A^n} \otimes_{\C[\A^n]} \C[\A^n-\Delta_{[n]} ] \hookrightarrow H^0\big( \Grb^{\la}_{\mc{G},\A^n-\Delta_{[n]}}, \mc{L}^c\big)^\vee,\]
	where $\Delta_{[n]}$ is defined in \eqref{348782}. Since $\Delta_{[n]}$ has dimension $1$, by Hartogs' Lemma, $\Phi$ extends to an embedding over $\A^n$,
	\[\Phi_{\A^n}: \Ds(c,\la)_{\A^n} \hookrightarrow H^0\big( \Grb^{\la}_{\mc{G},\A^n}, \mc{L}^c\big)^\vee.\]
	Note that both sides are graded free $\C[\A^n]$-modules, and have the same graded $T^\sigma$-character as the same argument as in the cases for $n=1,2$. Thus,  $\Phi_{\A^n}$ must be an isomorphism.
\end{proof}
\begin{remark}\label{rem_difference}
\begin{enumerate}
\item The proof of this theorem is similar to the method used in \cite[Section 5.1]{DFF21}. It is worth noting that, in the twisted case, it is already nontrivial when $n=1$, and when $n=2$ the argument is more subtle. Moreover, in our argument, in order to prove the isomorphism (\ref{eq_final_thm}), the commutativity of the diagram (\ref{eq_thm_diag}) needs to be involved in the induction. 
\item In \cite{DFF21}, a version of Theorem \ref{055073} is proved for symmetric BD Schubert varieties in the untwisted case.  In our paper, we only work with the twisted BD Schubert varieties over $\A^n$, where the coordinates of points are ordered. The main reason is that the algebra $\As(c,\vec{\lambda})$ is more complicated than the untwisted case, see Remark \ref{rem_hw_alg}. In particular, when $n=1$, $H^0\big( \Grb^{\lambda}_{\mc{G},\A^1}, \mc{L}^c\big))$ is naturally an $\mathbb{C}[z]$-module, while $\Ds(c,\lambda)$ is a $\mathbb{C}[z^r]$-module for some integer $r\geq 1$, where $r$ depends on whether $\lambda$ is $\sigma$-invariant, see Example \ref{ex_algebra}.
\end{enumerate}
\end{remark}
The following corollary immediately follows from Theorem \ref{055073}.
\begin{cor} \label{367230}
	Let $\la=(\lambda_1,\ldots,\lambda_n)\in (X_*(T)^+)^n$. 
	\begin{enumerate}
		\item Let $\xi=(I_1,\ldots, I_k)$ be a partition on $[n]$. For any point $\vec{p}\in\A^n_\xi$, there is an isomorphism of $\gs$-modules as follows,
		\[ \Ds(c,\la)_{\vec{p}}=\bigotimes_{j=1}^k \Ds(c,\la_{I_j})_{\vec{p}_{I_j}},\]
		where $\vec{p}_{I_j}$ and $\la_{I_j}$ are defined in \eqref{053076} and \eqref{249296}.
		\item For any $\vec{p}=(p,\gamma_2(p),\ldots, \gamma_n(p))\in \A^n$ such that $p\neq 0$, there is an isomorphism of $\gs$-modules 
		\[ \Ds(c,\la)_{\vec{p}}=D(c,\lambda_{\vec{p}})_{p},\]
		where $\lambda_{\vec{p}}:=\lambda_1+ \gamma^{-1}_2(\lambda_2)+\cdots+\gamma_n^{-1}(\lambda_n)$ and $D(c,\lambda_{\vec{p}})_{p}$ denotes the module $D(c,\lambda_{\vec{p}})$ with a new $\gs$-action given by \eqref{342834}. 
	\end{enumerate}
\end{cor}
\begin{proof}
	Part (1) follows from the fact that the $\gs$-morphism $\psi_\xi$ defined in Proposition \ref{348932} is an isomorphism. Part (2) follows from Corollary \ref{711570} and Theorem \ref{055073}.
\end{proof}
\begin{remark}\label{remark_fiber}
Note that when $\vec{p}=(0,\cdots, 0)\in \A^n$, the fiber $ \Ds(c,\la)_{\vec{p}}$ is already determined in Theorem \ref{398063}, and this fact is used in the proof of Theorem \ref{055073}. This fact together with Corollary \ref{367230} determines the fiber $ \Ds(c,\la)_{\vec{p}}$ at any point $\vec{p}\in \A^n$. 
\end{remark}

\appendix
\section{Tannakian interpretation for $(\Gamma, G)$-torsors}
\label{appendixA}
Let $G$ be an affine algebraic group over an algebraically closed field $k$. 
Recall a classical result that a $G$-torsors is equivalent to a faithful exact tensor functor from $\mr{Rep}(G)$ to the category of locally free sheaves, cf.\,\cite{Nori76,Broshi13}. In this appendix, we generalize this equivalence to an equivariant setting.

Let $\Gamma$ be a finite group acting on a $k$-scheme $Y$ and acting on an affine algebraic group $G$ over $k$. 

\begin{definition}\label{653962}
	A $(\Gamma,G)$-torsor is a scheme $\mc{F}$ which is faithfully flat and affine over $Y$, and equipped with a $\Gamma$-action and a right $G$-action such that
	\begin{enumerate}
		\item The map $\mc{F}\to Y$ is both $\Gamma$-equivariant and $G$-equivariant.
		\item The $G$-action on $\mc{F}$ is compatible with the $\Gamma$-action, i.e.\,the $G$-action map $\mc{F}\times G\to \mc{F}$ is $\Gamma$-equivariant, where $\Gamma$-action on $\mc{F}\times G$ is given by $\gamma(p,g)=(\gamma p,\gamma g)$ for $\gamma\in \Gamma$.
		\item The map $\mc{F}\times G\to \mc{F}\times_Y\mc{F}$ sending $(p,g)$ to $(p,pg)$ is an isomorphism.
	\end{enumerate}
Given $(\Gamma, G)$-torsors $\mc{F}$ and $\mc{F}'$, a $(\Gamma, G)$-equivariant map $f:\mc{F}\to \mc{F'}$ is a morphism of $G$-torsors which is also 
 $\Gamma$-equivariant. 
\end{definition}
We denote by $\mathbf{Bun}_{\Gamma,G}(Y)$ the category of $(\Gamma,G)$-torsors over $Y$ in which the morphisms are $(\Gamma, G)$-equivariant maps. 
\begin{definition}\label{325042}
	Let $\mf{C}$, $\mf{D}$ be two tensor categories, and let $F, F':\mf{C}\to \mf{D}$ be two tensor functors. A morphism $f: F\to F'$ of tensor functors is a natural transformation from $F$ to $F'$ such that, for any finite families $(V_i)_{i\in I}$ of objects in $\mf{C}$, the diagram
	\begin{equation}\label{405833}
		\xymatrix{
		\bigotimes_{i\in I} F(V_i) \ar[d]_{\bigotimes_{i\in I}f_{V_i}}^[@]{\sim} \ar[r]^{\sim} & F \big(\bigotimes_{i\in I} V_i \big) \ar[d]_[@]{\sim}^{f_{\bigotimes_{i\in I} V_i}} \\
		\bigotimes_{i\in I} F'(V_i) \ar[r]^{\sim} & F' \big(\bigotimes_{i\in I} V_i \big)
		}
	\end{equation}
	commutes. (See \cite[Definition 1.12]{DM18})
\end{definition}
\begin{remark}
	This definition is equivalent to the commutativity of the diagram (\ref{405833})  for $I=\{1,2\}$ and $I=\emptyset$.
\end{remark}
Denote by $\mr{Rep}(G)$ the category of finite dimensional algebraic representations of $G$ over $k$, and by $\mr{Coh}^{\mr{lf}}(Y)$ the category of locally free sheaves of finite rank on $Y$. For each $\gamma\in \Gamma$, denote by $\gamma_G$ (resp.\,$\gamma_Y$) the automorphism on $G$ (resp.\,$Y$) associated to $\gamma$. Let $\gamma_Y^*:\mr{Coh}^{\mr{lf}}(Y)\to \mr{Coh}^{\mr{lf}}(Y)$ be the pullback functors induced from $\gamma_Y$. We define a functor $\gamma_G^*:\mr{Rep}(G)\to \mr{Rep}(G)$ as follows, for any $(V,\rho) \in \mr{Rep}(G) $, 
\[\gamma_G^*(V, \rho):=(V, \rho\circ \gamma_G).\] The natural tensor structure on the category $\mr{Rep}(G)$ (resp.\,$\mr{Coh}^{\mr{lf}}(Y)$) is preserved by $\gamma^*_G$ (resp.\,$\gamma_Y^*$).
\begin{definition}\label{964424}
	We say a tensor functor $F:\mr{Rep}(G)\to\mr{Coh}^{\mr{lf}}(Y)$ is $\Gamma$-equivariant if it is equipped with a collection of isomorphisms between tensor functors $\{\theta_\gamma: F \gamma_G^*\to\gamma_Y^* F\}_{\gamma\in \Gamma}$ such that, for any $\eta,\gamma\in \Gamma$, the following diagram 
	\begin{equation}\label{230892}
		\xymatrix{
		F\eta_G^*\gamma_G^* \ar[d]_{F\delta}^[@]{\sim} \ar[r]^{\sim}_{\theta_{\eta}} & \eta_Y^*F\gamma_G^* \ar[r]^{\sim}_{\eta^*\theta_{\gamma}} & \eta_Y^*\gamma_Y^*F \ar[d]_[@]{\sim}^{\xi} \\
		F(\gamma\eta)_G^* \ar[rr]^{\sim}_{\theta_{\gamma\eta}}& & (\gamma\eta)_Y^*F }
	\end{equation}
	commutes, where $\delta: \eta_G^*\gamma_G^*\simeq (\gamma\eta)_G^*$ and $\xi: \eta_Y^*\gamma_Y^*\simeq (\gamma\eta)_Y^*$ are natural isomorphisms between these functors. In fact, the commutativity of this diagram implies that $\theta_e=\mr{id}: F=Fe^*\to e^* F=F$.
\end{definition}
Let $\mc{F}$ be a $G$-torsor over $Y$. For each object $V$ in $\mr{Rep}(G)$, we denote by $\ms{F}_V$ the sheaf of sections of the vector bundle $\mc{F}\times^G V$ over $Y$. Then, $\ms{F}_V$ is locally free and it is well-known that
\[\mr{Spec}\,S(\check{\ms{F}}_V)\simeq \mc{F}\times^G V,\] 
where $S(\check{\ms{F}}_V)$ denotes the symmetric algebra of the dual sheaf $\check{\ms{F}}_V$ of $\ms{F}_V$, cf.\,\cite[Exercise 5.18]{Hartshorne77}. There is a natural functor 
$\Phi(\mc{F}): \mr{Rep}(G)\to\mr{Coh}^{\mr{lf}}(Y)$ associated to $\mc{F}$, which sends any representation $V$ to the locally free sheaf $\ms{F}_V$ on $Y$.
\begin{lem}\label{001938}
	Let $\mc{F}$ be a $(\Gamma, G)$-torsor over $Y$. The functor $\Phi(\mc{F}): \mr{Rep}(G) \to \mr{Coh}^{\mr{lf}}(Y)$ is a faithful exact tensor functor and it is $\Gamma$-equivariant.
\end{lem}
\begin{proof}
		By \cite[Lemma 4.1]{Broshi13}, $F$ is a faithful exact tensor functor. Thus, it suffices to show $F$ satisfies Definition \ref{964424}. Let $\theta_\gamma$ be the natural isomorphism between the vector bundles $\mc{F}\times^G \gamma_G^*V$ and $\gamma_Y^*(\mc{F}\times^G V)$ over $Y$ sending $(p,v)$ to $(\gamma p,v)$. It is easy to check the collection of isomorphisms $\{\theta_\gamma\}_{\gamma\in \Gamma}$ satisfies \ref{964424}.
\end{proof}
\begin{definition}\label{176768}
	Let $F$, $F'$ be two $\Gamma$-equivariant tensor functors from $\mr{Rep}(G)$ to $\mr{Coh}^{\mr{lf}}(Y)$. A morphism from $F$ to $F'$ as $\Gamma$-equivariant tensor functors is a morphism $f:F\to F'$ of tensor functors such that for any $\gamma\in \Gamma$, the following diagram commutes 
	\begin{equation}\label{302075}
		\xymatrix{
		F\gamma_G^* \ar[d]_{\theta_{\gamma}}^[@]{\sim} \ar[r]^{f\circ \mr{id}} & F'\gamma_G^* \ar[d]_[@]{\sim}^{\theta'_{\gamma}} \\
		\gamma_Y^*F \ar[r]^{\mr{id} \circ f} & \gamma_Y^*F' }
	\end{equation}
	Denote by $\mathbf{TF}_{\Gamma}(G,Y)$ the category of tensor functors from $\mr{Rep}(G)$ to $\mr{Coh}^{\mr{lf}}(Y)$ that are faithful, exact and $\Gamma$-equivariant. The morphisms between these objects are defined to be morphisms between $\Gamma$-equivariant tensor functors.
\end{definition} 
Recall that given a $(\Gamma,G)$-torsor $\mc{F}$ over $Y$, for each object $V$ in $\mr{Rep}(G)$, one can associate a locally free sheaf $\ms{F}_V$. By Lemma \ref{001938}, the functor $\Phi(\mc{F}): \mr{Rep}(G)\to\mr{Coh}^{\mr{lf}}(Y)$ is an object in $\mathbf{TF}_{\Gamma}(G,Y)$. Let $f:\mc{F}\to \mc{F}'$ be a morphism of $(\Gamma,G)$-torsors. 
We naturally get a morphism $\Phi(f): \Phi (\mc{F})\to \Phi (\mc{F}')$ of $\Gamma$-equivariant tensor functors. 
As a summary, we have the following result.
\begin{prop}\label{401313}
	$\Phi$ is a functor from $\mathbf{Bun}_{\Gamma,G}(Y)$ to $\mathbf{TF}_{\Gamma}(G,Y)$.  \qed
\end{prop}
Conversely, for any object $F$ in $\mathbf{TF}_{\Gamma}(G,Y)$, we would like to construct a $(\Gamma,G)$-torsor corresponding to $F$. 
\begin{definition}
	An affine $G$-scheme over $Y$ is a scheme affine and flat over $Y$ with a left $G$-action. Moreover, an affine $(\Gamma,G)$-scheme over $Y$ is an affine $G$-scheme $f:X\to Y$ over $Y$ equipped with a $\Gamma$-action such that 
	\begin{enumerate}
		\item $f:X\to Y$ is $\Gamma$-equivariant.
		\item The $G$-action map $G\times X\to X$ is $\Gamma$-equivariant.
	\end{enumerate}
\end{definition} 
Denote by $\mr{Rep}^{\infty}(G)$ the category of locally finite dimensional representations of $G$, and by $\mr{QCoh}(Y)$ the category of quasi-coherent sheaves on $Y$. Let $F:\mr{Rep}(G)\to\mr{Coh}^{\mr{lf}}(Y)$ be a faithful exact tensor functor. One can extend it uniquely to a faithful exact tensor functor $\mr{Rep}^{\infty}(G)\to \mr{QCoh}(Y)$ which will still be denoted by $F$, see \cite[Lemma 4.4]{Broshi13}. For an affine $G$-scheme $X$ over $k$, we define a scheme over $Y$
\[\tilde F(X):=\mr{Spec}\, F(k[X]),\]
where $k[X]$ denotes the coordinate ring of $X$. 
It was shown in \cite[Lemma 4.5]{Broshi13} that $\tilde F(X)$ is flat and affine over $Y$. Thus, $\tilde F$ gives to a functor from the category of affine $G$-schemes over $k$ to the category of schemes affine and flat over $Y$. Given any representation $V$ in ${\rm Rep}(G)$, it can be regarded as a $G$-scheme. Then $\tilde{F}(V)=\tilde{F(V)}$, where $\tilde{F(V)}$ represents the affine bundle associated to the locally free sheaf $F(V)$.

Let $X$ be an affine $G$-scheme over $k$ and for any $\gamma\in \Gamma$, we define $\gamma_G^*X$ to be the scheme $X$ with a left $G$-action twisted by the morphism $\gamma_G:G\to G$, and define $\gamma_Y^* \tilde FX$ to be the pullback of the scheme $\tilde FX$ via the morphism $\gamma_Y:Y\to Y$. Then we have $\gamma_G^*X=\mr{Spec}(\gamma_G^*(k[X]))$ and $\gamma_Y^* \tilde FX=\mr{Spec}(\gamma_Y^*F(k[X]))$. If $F:\mr{Rep}(G)\to\mr{Coh}^{\mr{lf}}(Y)$ is $\Gamma$-equivariant, then the extended functor $F: \mr{Rep}^{\infty}(G)\to \mr{QCoh}(Y)$ is also $\Gamma$-equivariant, i.e.\,there is a collection of natural isomorphisms between tensor functors $\{\theta_\gamma: F\gamma_G^*\to \gamma_Y^*F\}_{\gamma\in\Gamma}$ satisfying the diagram \eqref{230892} in Definition \ref{964424}. For each $\gamma\in \Gamma$, $\theta_\gamma$ induces an isomorphism $\tilde\theta_{\gamma,X}: \tilde F\gamma_G^*X \simeq \gamma_Y^*\tilde FX$ of schemes which is functorial in $X$.
\begin{lem}\label{931735}
	Let $F:\mr{Rep}(G)\to\mr{Coh}^{\mr{lf}}(Y)$ be a $\Gamma$-equivariant faithful and exact tensor functor. The induced functor $\tilde F$ satisfies the following properties: 
	\begin{enumerate}
		\item $\tilde F$ respects products and if $X$ has a trivial $G$-action then $\tilde F(X)=Y\times X$ with the obvious $G$-action.
		\item For any $\gamma\in \Gamma$ and any two affine $G$-schemes $X_1,X_2$ over $k$, the following diagram commutes
		\[\xymatrix{
		\tilde F\gamma_G^*(X_1\times X_2) \ar[d]_{\tilde\theta_{\gamma,X_1\times X_2}}^[@]{\sim} \ar[r]^-{\sim} & \tilde F\gamma_G^*(X_1)\times_Y \tilde F\gamma_G^*(X_2) \ar[d]_[@]{\sim}^{(\tilde\theta_{\gamma,X_1},\tilde\theta_{\gamma, X_2})} \\
		\gamma_Y^*\tilde F(X_1\times X_2) \ar[r]^-{\sim}& \gamma_Y^*\tilde F(X_1)\times_Y \gamma_Y^*\tilde F(X_2) 
		}.\]
		\item For any $\eta,\gamma\in \Gamma$, the following diagram 
		\[
		\xymatrix{
		\tilde F\eta_G^*\gamma_G^* \ar[d]_{\tilde F\delta}^[@]{\sim} \ar[r]^{\sim}_{\tilde \theta_{\eta}} & \eta_Y^*\tilde F\gamma_G^* \ar[r]^{\sim}_{\eta^*\tilde \theta_{\gamma}} & \eta_Y^*\gamma_Y^*\tilde F \ar[d]_[@]{\sim}^{\xi} \\
		\tilde F(\gamma\eta)_G^* \ar[rr]^{\sim}_{\tilde \theta_{\gamma\eta}}& & (\gamma\eta)_Y^*\tilde F }
		\]
		commutes, where $\delta: \eta_G^*\gamma_G^*\simeq (\gamma\eta)_G^*$ and $\xi: \eta_Y^*\gamma_Y^*\simeq (\gamma\eta)_Y^*$ are the natural isomorphisms between functors.
	\end{enumerate}
\end{lem}
\begin{proof}
	Part (1) was established in \cite[Lemma 4.5]{Broshi13}. Part (2) follows from that fact that $\theta_{\gamma}$ is a morphism of tensor functors. Part (3) follows from the commutativity \eqref{230892} for $\{\theta_{\gamma}\}_{\gamma\in\Gamma}$.
	\end{proof}
\begin{definition}
	An affine $\Gamma$-scheme $X$ over $Y$ is a scheme affine and flat over $Y$ with a $\Gamma$-action such that the map $X\to Y$ is $\Gamma$-equivariant.
\end{definition}
\begin{cor}
\label{cor_gamma_Y}
	With the same assumption as in Lemma \ref{931735}, $F$ induces a functor $\tilde F$ from the category of affine $(\Gamma,G)$-schemes over $k$ to the category of affine $\Gamma$-schemes over $Y$.
\end{cor}
\begin{proof}
	Let $X$ be any $(\Gamma, G)$-scheme over $k$. For any $\gamma\in \Gamma$, the map $X \simto \gamma_G^*X$ sending $x$ to $\gamma(x)$ is an isomorphism of $G$-schemes. Applying the functor $\tilde{F}$, we get the following isomorphism of $G$-schemes over $Y$
	\[\varphi_{\gamma}: \tilde F X \simto \tilde F \gamma_G^*X \xrightarrow{\tilde{\theta}_{\gamma,X}} \gamma^*_Y \tilde F X. \]
	Moreover, for any $\eta, \gamma\in \Gamma$, the following diagram commutes
	\[\xymatrix{
		X \ar[d]^[@]{\sim} \ar[r]^-{\sim} & \eta_G^* X \ar[d]_[@]{\sim} \\
		(\gamma\eta)_G^*X \ar[r]^-{\sim} & \eta_G^*\gamma_G^*X
	}.\]
	Applying $\tilde{F}$, we have the commutative diagram 
	\[\xymatrix{
		\tilde FX \ar[d]_{\varphi_{\gamma\eta}}^-[@]{\sim} \ar[r]^{\sim}_{\varphi_\eta} & \eta_Y^*\tilde F X \ar[d]_[@]{\sim}^{\eta^*_G\varphi_\gamma} \\
		(\gamma\eta)_Y^*\tilde FX \ar[r]^{\sim}_{\xi^{-1}}& \eta_Y^*\gamma_Y^*\tilde FX
	},\]
	where $\xi: \eta_Y^*\gamma_Y^*\simto (\gamma\eta)_Y^*$ is the natural isomorphism between functors. This amounts to a $\Gamma$-action on $\tilde FX$. 
\end{proof}
 Denote by $G_0$ the underlying scheme of the algebraic group $G$ with the trivial right $G$-action. Let $a: G\times G_0\to G$ be the map sending $(g_1,g_2)$ to $g_1g_2$. Applying $\tilde F$ and by part (1) of Lemma \ref{931735}, this gives rise to a map 
\[\tilde a: \tilde F G\times G=\tilde F G\times_Y \tilde F G_0\to \tilde F G.\] It was shown in \cite[Lemma 4.7]{Broshi13} that $\tilde a$ gives a right $G$-action on $\tilde{F}G$, and furthermore $\tilde F G$ is a $G$-torsor over $Y$. In fact, we have more.
\begin{prop}
	$\tilde F G$ is a $(\Gamma,G)$-torsor over $Y$.
\end{prop} 
\begin{proof}
By Corollary \ref{cor_gamma_Y}, $\tilde{F}G$ is a $\Gamma$-scheme over $Y$. 
Also, as discussed above, $\tilde FG$ is a right $G$-torsor. Thus, it suffices to show that $\tilde a:\tilde FG\times G\to \tilde FG$ is $\Gamma$-equivariant. Since $G$ is a $(\Gamma,G)$-scheme, we have the following commutative diagram of $G$-schemes over $k$
	\[\xymatrix{
		G\times G_0 \ar[d]^[@]{\sim} \ar[r]^-{a} & G\ar[d]_[@]{\sim} \\
		\gamma_G^*G\times G_0 \ar[r]^-{a} & \gamma_G^*G
	},\]
	where the left vertical map is given by $(g_1,g_2)\mapsto (\gamma g_1, \gamma g_2)$, the right vertical map is given by $g\mapsto \gamma g$. Applying $\tilde F$, by Lemma \ref{931735} the following diagram commutes
	\[\xymatrix{
		\tilde F G\times G \ar[d]_{\varphi_\gamma\times \gamma}^-[@]{\sim} \ar[r]^-{\tilde a} & \tilde F G\ar[d]_-[@]{\sim}^{\varphi_\gamma} \\
		\gamma^*_Y \tilde F G\times G \ar[r]^-{\tilde a} & \gamma_Y^*\tilde F G 
	}.\]
	This implies that the map $\tilde a:\tilde FG\times G\to \tilde FG$ is $\Gamma$-equivariant.
\end{proof}

Given a morphism $\beta:F_1\to F_2$ in $\mathbf{TF}_{\Gamma}(G,Y)$, one can further show there is a morphism $\tilde \beta_G: \tilde F_1G\to \tilde F_2G$ between two $(\Gamma, G)$-torsors over $Y$. Define $\Psi(F):=\tilde FG$ and $\Psi(\beta):=\tilde \beta_G$. This gives a functor
\[\Psi:\mathbf{TF}_{\Gamma}(G,Y)\to \mathbf{Bun}_{\Gamma,G}(Y).\]

\begin{thm}\label{926976}
	The functor $\Psi:\mathbf{TF}_{\Gamma}(G,Y)\to \mathbf{Bun}_{\Gamma,G}(Y)$ is an equivalence. 
\end{thm}
\begin{proof}
	We will show $\Psi$ is inverse to the functor $\Phi$, which is well-defined by Proposition \ref{401313}. Let $\pi: \mc{F}\to Y$ be a $(\Gamma, G)$-torsor over $Y$, we have $\Psi\Phi(\mc{F}) = \mr{Spec}\,(\ms{F}_{k[G]})$, where $\ms{F}_{k[G]}$ denotes the sheaf of sections of $\mc{F}\times^G k[G]$ over $Y$. We would like to show $ \ms{F}_{k[G]}\simeq \pi_* \ms{O}_{\mc{F}}$, where $\ms{O}_\mc{F}$ denotes the structure sheaf of $\mc{F}$. For any open subset $U$ of $Y$, let $\mc{F}_U$ denote the $G$-torsor $\pi^{-1}(U)$ over $U$. Each section $s:U\to \mc{F}_U \times^Gk[G]$ uniquely corresponds to a section:
	\[\mc{F}_U \simeq \mc{F}_U \times_U U \xrightarrow{\mr{id}\times s} \mc{F}_U \times_U\mc{F}_U\times^Gk[G] \simeq (\mc{F}_U\times G)\times^G k[G] \simeq \mc{F}_U\times k[G].\]
	This is equivalent to a $G$-equivariant map $\phi_s:\mc{F}_U\to k[G]$. Composing with $\mr{ev}_e:k[G]\to k$, we get an morphism $\mr{ev}_e\circ \phi_s:\mc{F}_U\to k$. Conversely, for any morphism $f:\mc{F}_U\to k$, there is a $G$-equivariant map $\tilde{f}:\mc{F}_U\to k[G]$ defined by $\tilde{f}(p): g\mapsto f(pg)$. By the arguement above, this uniquely corresponds a section $U\to \mc{F}_U\times^G k[G]$. Thus, we have a $1$-$1$ correspondence between the sections $U\to \mc{F}_U\times^Gk[G]$ and the morphisms $\mc{F}_U\to k$. Thus, $ \ms{F}_{k[G]}\simeq \pi_* \ms{O}_{\mc{F}}$. One may further check this natural isomorphism is $(\Gamma,G)$-equivariant. 
	It follows that $\Psi\Phi(\mc{F}) =\mr{Spec}\,(\ms{F}_{k[G]})\simeq \mr{Spec}\,(\ms{O}_{\mc{F}})= \mc{F}$.

	Conversely, given any object $F\in \mathbf{TF}_{\Gamma}(G,Y)$, we will show there is an isomorphism $\Phi\Psi(F)\to F$ between $\Gamma$-equivariant tensor functors. 
	For each object $V$ in $\mr{Rep}(G)$, by definition $\Phi\Psi(F)(V)$ is the sheaf of sections of $\tilde{F}G\times^G V$ over $Y$. Observe that $F(V)$ is the sheaf of sections of $\tilde F(V)$ over $Y$. Thus, it suffices to show there are $\Gamma$-equivariant maps $\{\tilde{F}G\times^G V\to \tilde F(V)\mid V\in \mr{Rep}(G)\}$ functorial in $V$. Denote by $V_0$ (resp.\,$G_0$) the scheme $V$ (resp.\,$G$) with a trivial $G$-action. Consider the diagram 
	\[\xymatrix{
		G\times G_0 \times V_0 \ar[d]_{m\times \mr{id}} \ar[r]^-{\mr{id}\times a} & G\times V_0\ar[d]^{a} \\
		G\times V_0 \ar[r]^-{a} & V
	},\]
	where $a:G\times V_0\to V$ sends $(g,v)$ to $gv$, and $m:G\times G_0\to G$ sends $(g_1,g_2)$ to $g_1g_2$. Applying $\tilde{F}$, it induces an isomorphism $\tilde a:\tilde FG\times^G V\to \tilde{F}V$, cf.\,\cite[Theorem 4.8]{Broshi13}. From the construction of $\tilde{a}$, one may check $\tilde{a}$ is $\Gamma$-equivariant and functorial in $V$. 
\end{proof}

\section{A generalization of cohomology and base change theorem}\label{419108}
\begin{definition}
	An additive functor $T:\mf{C}\to \mf{D}$ between two additive abelian categories is called half-exact if for any short exact sequence $0\to M\to N\to L\to 0$ in $\mf{C}$ the sequence $T(M)\to T(N)\to T(L)$ is exact.
\end{definition}
Let $f:A \rightarrow B$ be a local homomorphism of Noetherian local rings. Let $\mf{m}$ be the maximal ideal of $A$, and $k=A/\mf{m}$. Let $\mr{Mod}^{\mr{fg}}(A)$ (resp.\,$\mr{Mod}^{\mr{fg}}(B)$) denote the category of finitely generated $A$-modules (resp.\,$B$-modules). Given an $A$-linear functor $T: \mr{Mod}^{\mr{fg}}(A) \to \mr{Mod}^{\mr{fg}}(B)$, one can attach a natural transformation 
\[T(A)\otimes_A\cdot \to T.\] 
This is constructed as follows. For any $A$-module $M$ and any element $m\in M$, they give rise to an $A$-morphism $\phi_{m}:A\to M$, given by $a\mapsto am$. Applying $T$, we get $T(\phi_m):T(A)\to T(M)$. This induces an $A$-bilinear map $T(A)\times M\to T(M)$ given by $(t,m)\mapsto T(\phi_m)(t)$. Thus, this induces a map  $T(A)\otimes_A M\to T(M)$. One may check this map is functorial in $M$.
\begin{prop}\label{259856}
	Let $T$ be a half-exact $A$-linear functor from $\mr{Mod}^{\mr{fg}}(A)$ to $\mr{Mod}^{\mr{fg}}(B)$. Assume that $T$ commutes with direct limits. Then,
	\begin{enumerate}
		\item If $T(k)=0$, then $T(M)=0$ for any $M\in \mr{Mod}^{\mr{fg}}(A)$.
		\item The following conditions are equivalent:
		\begin{enumerate}
		\item $T(A)\to T(k)$ is surjective.
		\item The functor $T$ is right exact.
		\item $T(A)\otimes_A\cdot \to T$ is an isomorphism.
		\end{enumerate}
	\end{enumerate}
\end{prop}
\begin{proof}
	cf.\,\cite[Theorem 2.2, Theorem 4.1]{OB72}.
\end{proof}
The following theorem is a generalization of cohomology and base change theorem, cf.\,\cite[Theorem 12.11]{Hartshorne77}.
\begin{thm}\label{259758}
	Let $\pi:X\to Y$ be a proper morphism of schemes over a Noetherian scheme $S$. Let $s\in S$, and let $X_s$ (resp.\,$Y_s$) be the schematic fiber $X\times_S \mr{Spec}\, k(s)$ (resp.\,$Y\times_S \mr{Spec}\, k(s)$), where $k(s)$ is the residue field of $s$; let $\pi_s:X_s\to Y_s$ be the induced morphism from $\pi$. Let $\ms{F}$ be the coherent sheaf, which is flat over $S$. Then,
	\begin{enumerate}
		\item If $R^i\pi_*\ms{F}|_{X_s}\xrightarrow{\varphi^i(s)} R^i(\pi_s)_*\ms{F}|_{X_s}$ is surjective, then $\varphi^i(s)$ is an isomorphism and it is true in a neighborhood of $s$ in $S$.
		\item Assume that $\varphi^i(s)$ is surjective. Then the following are equivalant:
		\begin{enumerate}
		\item\label{249662} $\varphi^{i-1}(s)$ is surjective.
		\item $R^i\pi_*(\ms{F})$ is flat in a neighborhood of $s$ in $S$.
		\end{enumerate}
	\end{enumerate}
\end{thm}
\begin{proof}
	Since the statements of the theorem are local, we may assume that $Y=\mr{Spec} B$, $S=\mr{Spec } A$, and the morphism $Y\to S$ is given by a local homomorphism $A\to B$ of local rings, where $A$ is Noetherian. We consider the functor 
	\[T^i(M):=R^i\pi_*(X, \ms{F}\otimes_A M).\] 
	Then $T^i$ is a half-exact $A$-linear functor from $\mr{Mod}^{\mr{fg}}(A)$ to $\mr{Mod}^{\mr{fg}}(B)$, as $\ms{F}$ is flat over $S$ and $\pi$ is proper. Part (1) immediately follows from part (2) of Proposition \ref{259856}.

	We now prove part (2). By part (1), $T^i(A)$ is flat if and only if $T^i$ is exact. Then part (2) follows from Proposition \ref{259856} and the following basic fact: $T^i$ is exact if and only if $T^{i-1}$ and $T^i$ are both right exact.
\end{proof}
\begin{cor}\label{173291}
	Let $\pi:X\to Y$ be a proper morphism of schemes over a Noetherian scheme $S$. Suppose that $X$ is flat over $S$. Given a point $s\in S$, if $R^1(\pi_s)_*(\ms{O}_{X_s})=0$, then 
	\begin{enumerate}
		\item $\pi_*(\ms{O}_X)|_{Y_s}\simeq (\pi_s)_*(\ms{O}_{X_s})$
		\item $\pi_*(\ms{O}_X)$ is flat over a neighborhood of $s\in S$.
	\end{enumerate}
\end{cor}
\begin{proof}
	By part (1) of Theorem \ref{259758}, $R^1\pi_*(\ms{O}_X)$ is $0$ in a neighborhood of $s\in S$. In particular, $R^1\pi_*(\ms{O}_X)$ is flat over a neighborhood of $s\in S$. By part (2) of Theorem \ref{259758}, $\pi_*(\ms{O}_X)|_{Y_s}\to (\pi_s)_*(\ms{O}_{X_s})$ is surjective. Then part (1) follows from part (1) of Theorem \ref{259758}.

For part (2), by part (2) of Theorem \ref{259758} again, $\pi_*(\ms{O}_X)$ is flat over a neighborhood of $s\in S$ (note that condition \eqref{249662} in Proposition \ref{259856} holds automatically when $i=0$). 
\end{proof}

\section{Equivariant line bundles on stacks}\label{appendC}
In this appendix, we collect some basics of line bundles over $k$-stacks, where $k$ is an algebraically closed field. We also prove that for any semi-direct $k$-group $G\rtimes H$, any $G\rtimes H$-equivariant line bundle on a $k$-space $X$ can descend to a $H$-equivariant line bundle on the quotient stack $[G\backslash X]$. 

Denote by $\mr{Aff}_k$ the category of affine $k$-schemes. A $k$-space (resp. $k$-group) is a functor from $\mr{Aff}_k^{\mr{op}}$ to the category of sets (resp. groups) which is a sheaf with respect to the fppf topology, cf.\,\cite{BL94, Zhu17}.

Let $X$ be a $k$-space. Following \cite[(3.7)]{BL94}, a line bundle $\ms{L}$ over $X$ consists of following data: a line bundle $\ms{L}_\eta$ over $S$ for each scheme $S$ with a morphism $\eta:S\to X$, a collection of isomorphisms $\{\phi_{\eta,f}: f^*\ms{L}_\eta\simeq \ms{L}_{\eta\circ f}\}$ for each morphism $ f:S'\to S$ of schemes, subject to obvious compatibility conditions. In fact, this definition coincides with the usual definition of line bundles when $X$ is either a scheme or an ind-scheme.

More generally, a line bundle $\mc{L}$ on a stack $Y$ is defined to be the following data: a line bundle $\mc{L}_\eta$ over $S$ for each scheme $S$ and $\eta\in Y(S)$, an isomorphism $\phi_\theta:f^*\mc{L}_\eta \simeq \mc{L}_{\eta'}$ for each morphism $f:S'\to S$ of schemes and each arrow $\theta:f^*\eta \to \eta'$ in $Y(S')$, subject to obvious compatibility conditions, cf.\,\cite[(3.7)]{BL94}. 

Let $G$ be a $k$-group, and let $X$ be a $k$-space with a left $G$-action. Denote by $a:G\times X\to X$ the action map, and by $\mr{pr}:G\times X\to X$ the projection map. A line bundle $\ms{L}$ over $X$ is called $G$-equivariant if there is an isomorphism $u: a^*\ms{L}\simeq \mr{pr}^*\ms{L}$, satisfying the cocycle condition. Let $[G\backslash X]$ be the quotient stack. By descent theory, a line bundle over $[G\backslash X]$ is the same as a $G$-equivariant line bundle over $X$, cf.\,\cite[Section 7]{BL94}. 

Suppose there is a $k$-group $H$ acting on the $k$-group $G$ preserving the group structure, and there is a $(G\rtimes H)$-action on the $k$-space $X$. Then there is an $H$-action on $[G\backslash X]$ defined as follows. For any scheme $S$, an $S$-point $\eta$ of $[G\backslash X]$ is equivalent to a $G$-torsor $\pi:\mc{P}\to S$ together with a $G$-equivariant morphism $\alpha:\mc{P}\to X$. This produces a $G$-equivariant morphism
\[\tilde\eta:\mc{P}\to X\times S.\]
For any $h\in H(S)$, there is a $G$-equivariant map 
\begin{equation}\label{258623}
	h:X\times S\to X\times S,
\end{equation}
sending $(x,s)$ to $(h(s)^{-1}x,s)$. Here the domain $X\times S$ is equipped with a new $G$-action by $g\bu (x,s):=(h(s)gh^{-1}(s)x,s)$. Let $\mc{P}_h$ denote the $G$-torsor $P$ with a new $G$-action given by $g \bu p:= h(\pi(p))gh^{-1}(\pi(p))p$. Then, the following composition map is $G$-equivariant
\[\alpha_h:\mc{P}_h\xrightarrow{\tilde\eta} X\times S\xrightarrow{h} X\times S\xrightarrow{\mr{pr}_X} X,\]
which sends $p$ to $h(\pi(p))^{-1}\alpha(p)$. Hence, the $G$-torsor $\mc{P}_h$ over $S$ together with the $G$-equivariant morphism $\alpha_h:\mc{P}_h\to X$ defines an $S$-point of $[G\backslash X]$, which will be denoted by $h\cdot\eta$. This gives us an $H$-action on $[G\backslash X]$. In this case, a line bundle $\mc{L}$ over $[G\backslash X]$ is called $H$-equivariant if there are isomorphisms $\phi_{h,\eta}:\mc{L}_\eta\simeq \mc{L}_{h\cdot \eta}$ satisfying the cocycle condition.

\begin{lem}\label{708537}
	Let $H$ be a $k$-group, and let $G$ be a $k$-group with an $H$-action preserving the group structure. Suppose $(G\rtimes H)$ acts on a $k$-space $X$, and let $\ms{L}$ be a $(G\rtimes H)$-equivariant line bundle over $X$. Denote by $\mc{L}$ the line bundle over the quotient stack $[G\backslash X]$ descending from $\ms{L}$. Then there is an $H$-equivariant structure on $\mc{L}$.
\end{lem}
\begin{proof}
	For any $\eta\in [G\backslash X](S)$, we have an morphism $\tilde\eta:\mc{P}\to X\times S$. Let $\ms{L}_S$ be the pull-back of  $\ms{L}$ via the projection $X\times S\to X$. As a $G$-equivariant line bundle, $\tilde\eta^*(\ms{L}_S)$ exactly descends to the line bundle $\mc{L}_\eta$ over $S$, where $\mc{L}_\eta$ is obtained from $\mc{L}$ via $\eta$. 

	Since $\ms{L}$ is $G\rtimes H$-equivariant, its pull-back $\ms{L}_S$ is also $G\rtimes H$-equivariant, i.e.\,given any scheme $T$, for any $h'\in H(T)$, $g\in G(T)$ and $x\in (X\times S)(T)$, there are isomorphisms $\phi_{h',x}:(\ms{L}_S)_x\simeq (\ms{L}_S)_{h'x}$ and $\psi_{g,x}: (\ms{L}_S)_x\simeq (\ms{L}_S)_{gx}$ such that the following diagram commutes
	\begin{equation}\label{405823}
		\xymatrix{
			(\ms{L}_S)_x \ar[r]^-{\phi_{h',x}} \ar[d]_{\psi_{g,x}} & (\ms{L}_S)_{h'x} \ar[d]^-{\psi_{h'gh'^{-1},h'x}} \\
			(\ms{L}_S)_{gx} \ar[r]^-{\phi_{h',gx}}& (\ms{L}_S)_{h'gx} 
		}.
	\end{equation}
	
	For any $h\in H(S)$, let $h:X\times S\to X\times S$ be the $G$-equivariant morphism as in \eqref{258623}. Take $h'$ to be the morphism $h\circ \mr{pr}_S\circ x: T\to H$. Then $h'\cdot x=h\circ x$. Hence 
	\[(h^*\ms{L}_S)_x= (\ms{L}_S)_{h\circ x}=(\ms{L}_S)_{h'\cdot x}.\]
	Similarly, $(h^*\ms{L}_X)_{gx}=(\ms{L}_S)_{h'\cdot gx}$. Set $\xi_{h,x}=\psi_{h'gh'^{-1},h'x}:(h^*\ms{L}_S)_x\simeq (h^*\ms{L}_X)_{gx}$, and $\zeta_{h,x}=\phi_{h',x}:(\ms{L}_S)_x\to (h^*\ms{L}_S)_{x}$.
	Then diagram \eqref{405823} becomes 
	\begin{equation}\label{436587}
		\xymatrix{
			(\ms{L}_S)_x \ar[r]^-{\zeta_{h,x}} \ar[d]_{\psi_{g,x}} & (h^*\ms{L}_S)_{x} \ar[d]^-{\xi_{h,x}} \\
			(\ms{L}_S)_{gx} \ar[r]^-{\zeta_{h,gx}}& (h^*\ms{L}_S)_{gx} 
		}.
	\end{equation}
	Note that the data of isomorphisms $\xi_{h,x}:(h^*\ms{L}_S)_x\to (h^*\ms{L}_S)_{gx}$ of line bundles over $S$ defines a $G$-equivariant structure on the line bundle $h^*\ms{L}_S$ over $X\times S$. Moreover, the data of isomorphism $\zeta_{h,x}:(\ms{L}_S)_x\to (h^*\ms{L}_S)_x$ amounts to an isomorphism 
	\[\zeta_h:\ms{L}_S\to h^*\ms{L}_S.\] Then, the commutativity of the diagram \eqref{436587} implies that $\zeta_h$ is $G$-equivariant. 
	
	Pulling back $\zeta_h$ via $\tilde\eta:\mc{P}\to X\times S$, we get a $G$-equivariant isomorphism \[\tilde\eta^*\ms{L}_S\to \tilde\eta^*h^*\ms{L}_S=(h\circ\tilde\eta)^*\ms{L}_S.\]
	Note that $\mr{pr}_X\circ h\circ\tilde\eta=\alpha_{h}$. This implies $\widetilde{h\cdot\eta}=h\circ\tilde\eta$. Hence, there is an $G$-equivariant isomorphism
	\[\tilde\eta^*\ms{L}_S\simeq (\widetilde{h\cdot\eta})^*\ms{L}_S,\]
	and it descends to an isomorphism $\mc{L}_\eta\simeq\mc{L}_{h\cdot\eta}$ of line bundles over $S$. One may check this isomorphism satisfies the standard cocycle condition.
\end{proof}

\end{document}